\documentclass[11pt]{amsart}

\RequirePackage{amsthm,amsmath,amsfonts,amssymb}
\RequirePackage[numbers]{natbib}
\RequirePackage[colorlinks,citecolor=blue,urlcolor=blue]{hyperref}
\RequirePackage{graphicx}

\usepackage[all]{xy}
\usepackage{amssymb}
\usepackage{enumerate}
\usepackage{mathrsfs}

\usepackage[usenames,dvipsnames]{xcolor}
\usepackage{centernot}
\usepackage{comment}

\numberwithin{equation}{section}

\newtheorem{theorem}{Theorem}[section]
\newtheorem{remark}[theorem]{Remark}
\newtheorem{lemma}[theorem]{Lemma}
\newtheorem{proposition}[theorem]{Proposition}
\newtheorem{corollary}[theorem]{Corollary}

\newtheorem{res}{Result}

\newtheorem{theo}[res]{Theorem}
\newtheorem{prop}[res]{Proposition}

\newcommand{\C}{\mathbf{C}}
\newcommand{\D}{\mathbf{D}}
\newcommand{\E}{\mathbf{E}}

\newcommand{\HH}{\mathbf{H}}
\newcommand{\h}{\mathbf{H}}

\newcommand{\p}{\mathbf{P}}

\newcommand{\R}{\mathbf{R}}

\newcommand{\CA}{\mathcal {A}}
\newcommand{\CB}{\mathcal {B}}

\newcommand{\CD}{\mathcal {D}}

\newcommand{\CF}{\mathcal {F}}

\newcommand{\CT}{\mathcal {T}}

\newcommand{\CV}{\mathcal {V}}
\newcommand{\CW}{\mathcal {W}}

\newcommand{\CG}{\mathcal {G}}

\newcommand{\SLE}{{\rm SLE}}
\newcommand{\CLE}{{\rm CLE}}
\newcommand{\QLE}{{\rm QLE}}

\newcommand{\re}{\mathrm{Re}}

\newcommand{\one}{{\bf 1}}

\newcommand{\wt}{\widetilde}
\newcommand{\wh}{\widehat}

\newcommand{\giv}{\,|\,}

\newcommand{\eps}{\epsilon}

\newcommand{\BCLE}{\mathrm{BCLE}}

\newcommand{\IG}{{\mathrm{IG}}}

\newcommand{\strip}{{\mathscr S}}

\newcommand{\jason}[1]{{#1}}
\newcommand{\jasonr}[1]{{#1}}
\newcommand{\ww}[1]{{#1}}

\begin{document}

\title[Simple CLEs on LQG]{Simple Conformal Loop Ensembles \\on Liouville Quantum Gravity}

\author{Jason Miller}
\address{Statslab, Center for Mathematical Sciences, University of Cambridge, Wilberforce Road, Cambridge CB3 0WB, UK}
\email {jpmiller@statslab.cam.ac.uk}

\author{Scott Sheffield}
\address{
Department of Mathematics, MIT,
77 Massachusetts Avenue,
Cambridge, MA 02139, USA}
\email{sheffield@math.mit.edu}

\author{Wendelin Werner}
\address{Department of Mathematics, ETH Z\"urich, R\"amistr. 101, 8092 Z\"urich, Switzerland} 
\email{wendelin.werner@math.ethz.ch}

\maketitle

\begin{abstract}
We show that when one draws a simple conformal loop ensemble ($\CLE_\kappa$ for $\kappa \in (8/3,4)$) on an independent $\sqrt{\kappa}$-Liouville quantum gravity (LQG) surface and explores the CLE in a natural Markovian way, the quantum surfaces (e.g., corresponding to the interior of the CLE loops) that are cut out form a Poisson point process of quantum disks.  This construction allows us to make direct links between CLE on LQG, $(4/\kappa)$-stable processes, and labeled branching trees. The ratio between positive and negative jump intensities of these processes turns out to be $-\cos (4 \pi / \kappa)$, which can be interpreted as a ``density'' of $\CLE$ loops in the $\CLE$ on LQG setting. Positive jumps correspond to the discovery of a CLE loop (where the LQG length of the loop is given by the jump size) and negative jumps correspond to the moments where the discovery process splits the remaining to be discovered domain into two pieces.

Some consequences of this result are the following:  (i) It provides a construction of a $\CLE$ on LQG as a patchwork/welding of quantum disks.  (ii) It allows \jasonr{us} to construct the ``natural quantum measure'' that lives in a $\CLE$ carpet.  (iii) It enables us to derive some new properties and formulas for $\SLE$ processes and $\CLE$ themselves (without LQG) such as the exact distribution of the trunk of the general $\SLE_\kappa(\kappa-6)$ processes.

The present work deals directly with structures in the continuum and makes no reference to discrete models, but our calculations match those for scaling limits of $O(N)$ models on planar maps with large faces and $\CLE$ on LQG.  Indeed, our L\'evy-tree descriptions are exactly the ones that appear in the study of the large-scale limit of peeling of discrete decorated planar maps such as in recent work of Bertoin, Budd, Curien and Kortchemski.

The case of non-simple CLEs on LQG \ww{is studied} in another paper.
\end{abstract}

\setcounter{tocdepth}{1}
\tableofcontents

\eject

\setcounter{tocdepth}{1}
\tableofcontents

\eject

\section{Introduction}
\label{sec:intro}

\subsection{Background} 
The present work \jasonr{gives} some new direct connections between conformal loop ensembles (CLE) and Liouville quantum gravity (LQG). Before describing our main results, let us give quick one-page surveys about each of the three main objects involved: CLE and their explorations;  LQG surfaces;  asymmetric stable processes and labeled trees.    

\subsubsection{Background on CLE and on CLE explorations}

The {\em Schramm-Loewner evolutions} ($\SLE_{\kappa}$) were introduced by Schramm in \cite{S0} and are individual random curves joining two boundary points of a simply connected domain.  They are defined as an infinitely divisible iteration of independent random conformal maps and they are classified by a positive parameter $\kappa$. The $\SLE_\kappa$ curves turn out to be simple when $\kappa \le 4$, and they have double points as soon as $\kappa > 4$ \cite{RS05}.  The {\em conformal loop ensembles} $\CLE_\kappa$ \cite{SHE_CLE,SHE_WER_CLE} are random families of non-crossing loops in a simply connected domain $D$. In a $\CLE_\kappa$, the loops are $\SLE_\kappa$-type curves (for the same value of $\kappa$). While $\SLE_\kappa$ corresponds to the conjectural scaling limit of a single interface in a statistical physics model in a domain with some special boundary conditions involving two marked boundary points, $\CLE_\kappa$ is the conjectural scaling limit of the whole collection of interfaces with some uniform boundary conditions. It turns out that the $\CLE_\kappa$ can be defined only in the regime where $\kappa \in (8/3, 8)$.  The phase transition at $\kappa =4$ for $\SLE_\kappa$ curves is mirrored by the properties of the corresponding $\CLE_\kappa$: When $\kappa \in (8/3, 4]$, which is the case that we will focus on in the present paper,  the $\CLE_\kappa$ loops are all disjoint and simple. They conjecturally correspond to the scaling limit of dilute $O(N)$ models for $N \in (0,2]$ (this is actually proved in the special case $N=1$ which is the critical Ising model \cite{s2010ising,ks2016fkconvergence,gw2019fk,BEN_HONG}). The set of points that are surrounded by no $\CLE_\kappa$ loop is called the $\CLE_\kappa$ carpet.  As shown in \cite{SHE_WER_CLE}, these simple CLEs can be constructed in different ways, including via the so-called Brownian loop-soups. However, in the present work, we will use mostly the original SLE branching tree construction proposed in \cite{SHE_CLE}, and further studied and described in \cite{SHE_WER_CLE,ww2013conformally,cle_percolations}. 

Let us give a brief intuitive description of some relevant results about the loop-trunk decomposition of (the totally asymmetric) branches of the $\SLE_\kappa$ branching tree  
(mostly from \cite{cle_percolations}) that will play an important role in the present paper: 
Suppose that $\kappa \in (8/3, 4)$, that one is given a $\CLE_\kappa$ in a domain $D$,
and that one chooses two boundary points $x$ and $y$. Then, it is possible to make sense of a random non-self-crossing curve (but with double points) from $x$ to $y$, that (i) stays in the CLE carpet, (ii) always leaves a CLE loop to its right if it hits it, (iii) possesses some conformal invariance and locality properties. These three properties in fact characterize the law of the curve, so that it can be interpreted as the critical percolation interface from $x$ to $y$ in the $\CLE_\kappa$ carpet (loosely speaking, it traces the outer boundary of percolation clusters in this carpet that touch the clockwise part of the boundary from $x$ to $y$). This process is called the CPI (conformal percolation interface) in the $\CLE_\kappa$ carpet. 

\begin{figure}[ht!]
\includegraphics[scale=.55]{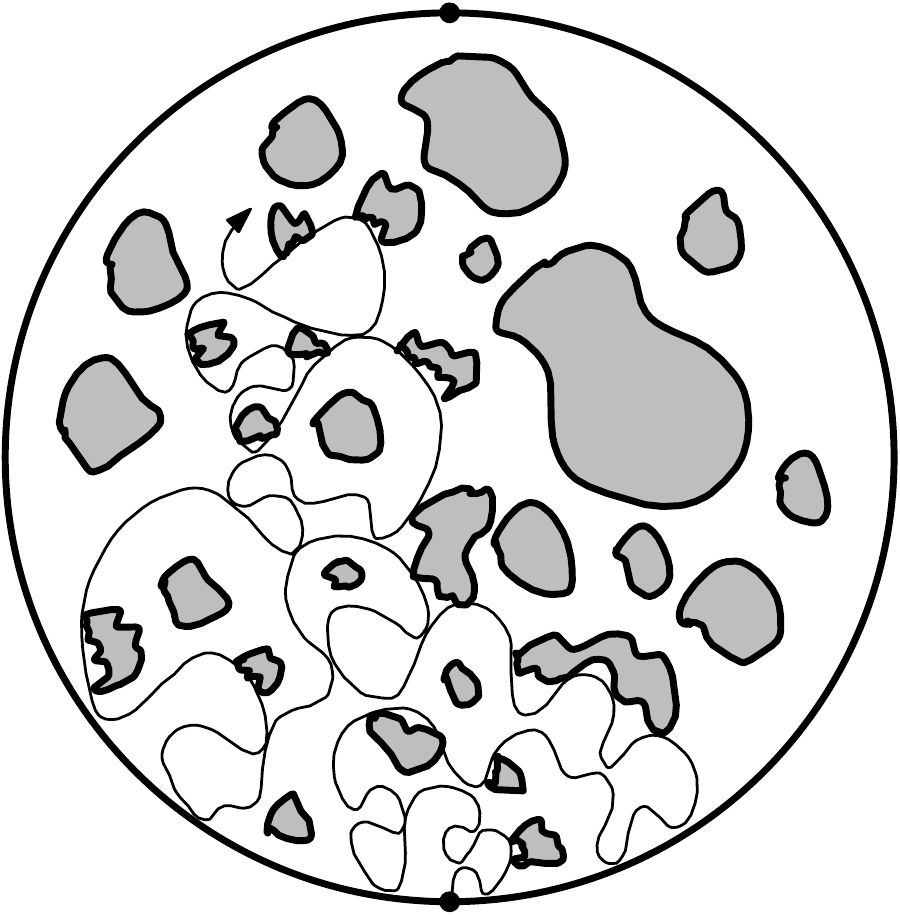} 
\includegraphics[scale=.55]{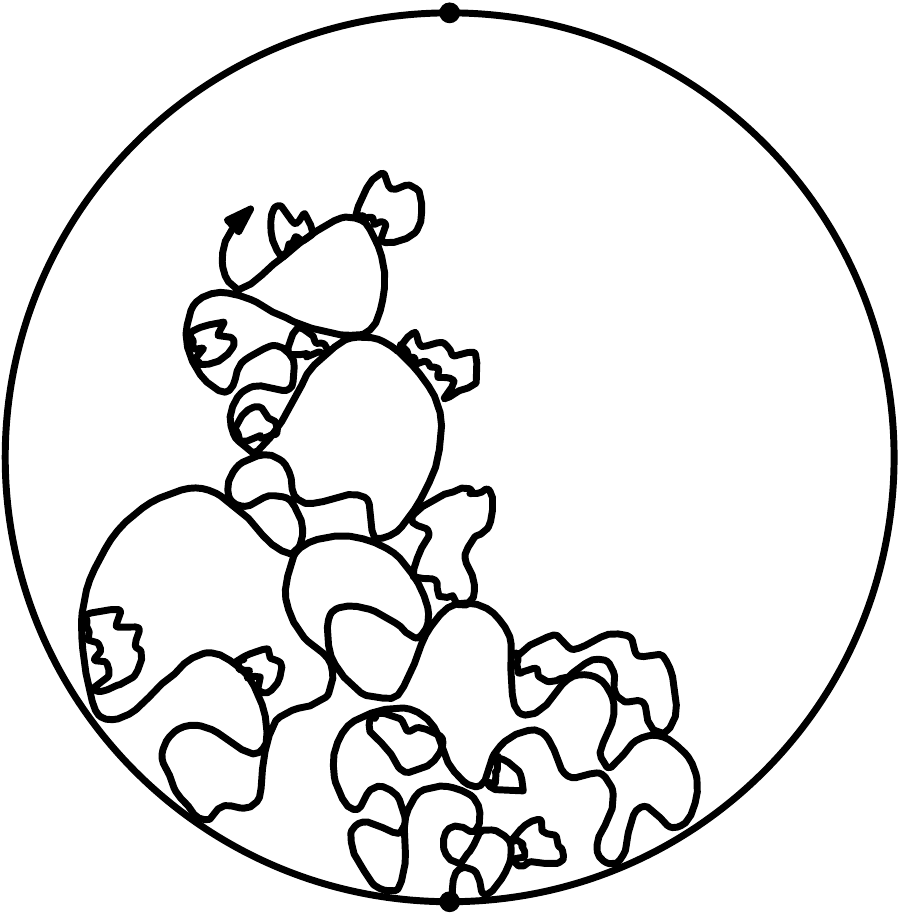} 
\includegraphics[scale=.55]{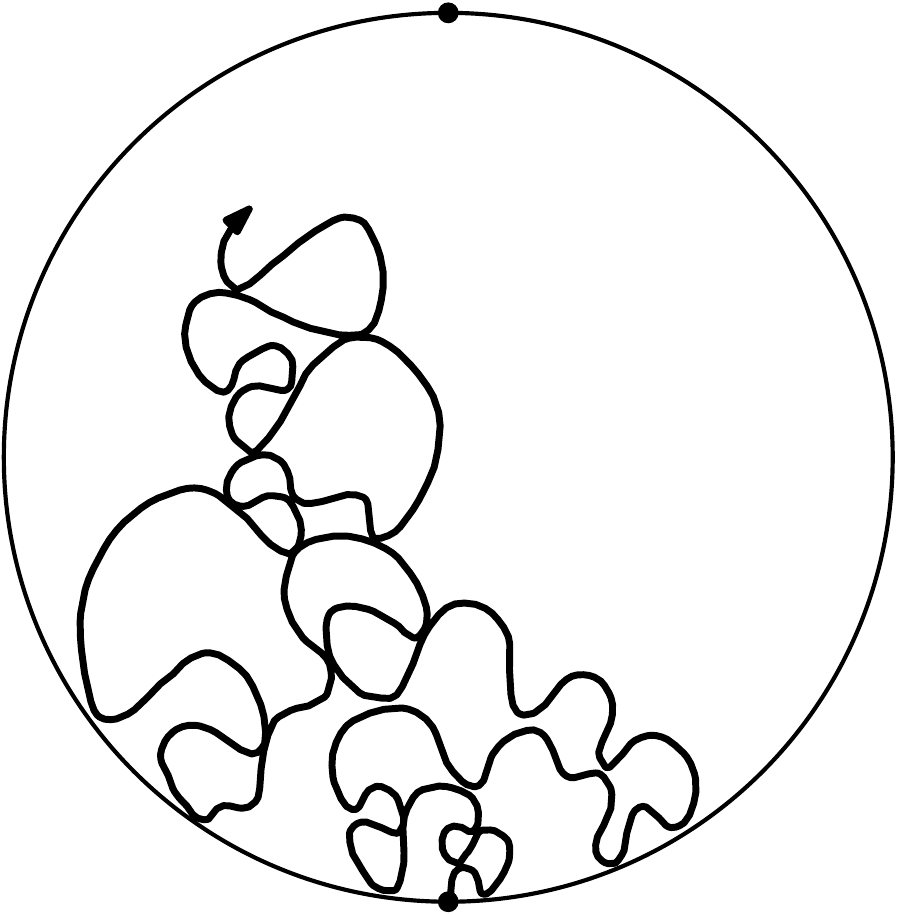} 
	\caption{CPI in a CLE, the curve $\eta$ and its trunk}
	\label{picCPI}
\end{figure}

The ``quenched'' law (i.e., averaged over all possible $\CLE_\kappa$) of the CPI turns out to be the natural target-invariant version of $\SLE_{\kappa'}$ for $\kappa' = 16 / \kappa$, called \jasonr{an} $\SLE_{\kappa'} ( \kappa' - 6)$ process.
In the same way as for ordinary percolation, it is actually possible to make sense of the whole branching tree of CPIs targeting all the points in the domain.

As one can expect from the properties of CLEs, the law of the ordered family of CLE loops that a CPI branch encounters can be viewed as coming from a Poisson point process of $\SLE_\kappa$ type bubbles. The process that one obtains by tracing the CPI and the CLE loops along the way in the order in which they are encountered (this process is called the $\SLE_\kappa(\kappa-6)$ process), can be reconstructed from the \jasonr{ordered} collection of $\SLE_\kappa$ bubbles \jasonr{(each marked by the first point visited by the trunk)}. 
The CPI is then called the {\em trunk} of this $\SLE_\kappa (\kappa -6)$ -- see Figure~\ref{picCPI}. In \cite[Theorem~7.4]{cle_percolations}, it is explained how to construct also this process by first sampling the entire 
CPI (which is an $\SLE_{\kappa'}( \kappa' -6)$ as mentioned above), and then only to attach the collection of discovered CLE loops to it -- this is referred to as the {\em loop-trunk} decomposition of the (totally asymmetric) $\SLE_\kappa (\kappa -6)$.

\subsubsection{Quantum surfaces and quantum disks} 

Let us first briefly recall that there are two essentially equivalent approaches to LQG surfaces -- one where one views it as a random metric and one where one views it in terms of a random area measure (or random lengths of some particular random curves). In the present paper, we will always stick to the latter one (that we now briefly discuss) and will not need to know about the former -- we can nevertheless mention that it has been recently shown that the random area measure defines also a random distance
(see \cite{dddf2019tightness,gm2019exunique} and the reference therein).

Suppose that one is given a simply connected domain $D$ in the complex plane (with $D \not= \C$) and an instance of the Gaussian free field (GFF) with Dirichlet boundary conditions on $D$, to which one adds some (possibly random and unbounded) harmonic function $h_0$, and let $h$ denote the obtained field (this includes for instance the case where $h$ is a GFF with ``free boundary conditions'').  It is by now a classical fact that can be traced back at least to work by H{\o}egh-Krohn \cite{hk1971quantum} and Kahane \cite{kahane1985gmc} that when $\gamma \in (0, 2)$, it is possible to define an {\em LQG area measure} $\mu_h= \mu_{h, \gamma}$ that can loosely be viewed as having a density $\exp (\gamma h)$ with respect to Lebesgue measure (one precise definition goes via an approximation/renormalization procedure, see \cite{DS08}). An important feature to emphasize is that the obtained area measure is conformally covariant, in the sense that the image of $\mu_{h, \gamma}$ under a conformal map $\Phi$ from $D$ onto $\wt D$ will be the measure $\mu_{\wt h, \gamma}$, where $\wt h$ is the sum of a Dirichlet GFF in $\wt D$ with the harmonic function $\wt{h}_0 = h_0 \circ \Phi^{-1} + Q \log |(\Phi^{-1})' |$, with $Q = (2 / \gamma + \gamma /2 )$.  This conformal covariance was \jasonr{proved} to hold a.s.\ for a \emph{fixed} conformal map in \cite{DS08} and to hold a.s.\ simultaneously for \emph{all} conformal maps in \cite{sw2016measure}.

In the case that $h_0$ is chosen so that the obtained field $h$ is absolutely continuous with respect to a realization of the GFF with free boundary conditions, a similar procedure can be used to define a measure on the boundary of $D$, that is referred to as the  {\em quantum boundary length measure $\nu_h = \nu_{h, \gamma}$}.  This boundary length measure is in fact a deterministic function of the area measure $\mu_h$ (it is in some sense its ``trace on the boundary'' and it is  a deterministic function of the random function $h_0$).  Again, this boundary length measure turns out to be conformally covariant in the same way the LQG area measure is.  This leads to the definition of a \emph{quantum surface}, which is an equivalence class of distributions where distributions $h$, $\wt{h}$ are equivalent if there exists a conformal transformation $\Phi$ so that $h = \wt{h} \circ \Phi + Q\log|\Phi'|$.  When one speaks of a quantum surface, it therefore means a distribution modulo this change of coordinate rule.  A representative from the equivalence class is an \emph{embedding} of a quantum surface.

 Another important feature, pioneered in \cite{SHE_WELD,dms2014mating} and extensively used in all the subsequent LQG/SLE papers is that when $\kappa = \gamma^2$ and $\kappa' = 16/ \kappa$ (we will keep these relations throughout the paper), then when one draws $\SLE_\kappa$-type curves in $D$ or $\SLE_{\kappa'}$-type curves in $D$ that are independent of the field $h$, then it is possible to make sense (again in a conformally \jasonr{covariant} way) of their quantum length -- which therefore provides a way to parameterize the curve using the additional random input provided by $h$.  
 \ww{An important feature to keep in mind is that the definition of the quantum length for $\SLE_\kappa$ and of the quantum length for $\SLE_{\kappa'}$ associated to a same LQG surface are somewhat different (and give rise to different scaling rules) -- see Remark \ref {GQLRemark}.}

 It turns out that some special choices for the random harmonic function $h_0$ are particularly interesting. One of these choices gives rise to the so-called $\gamma$-quantum disks.  One property of quantum disks (for $\gamma \in (0,2)$) is that almost surely, the total area measure is finite and the total boundary length is finite.  It is actually convenient to work with either the probability measure $P_L$ on quantum disks with a given boundary length $L$, or with the infinite measure on quantum disks $M = \int P_L dL/L^{\alpha + 1}$, where here and in the sequel, $\alpha$ and $\gamma$ are related by $\alpha = 4 / \gamma^2 = 4 / \kappa$.  The measure $P_L$ can be obtained from $P_1$ by adding $({2}/{\gamma} ) \log L$ to the field, which has the effect of multiplying the boundary length measure by the factor $L$ and the area measure by the factor $L^2$. 
 
\subsubsection{Relevant L\'evy processes and fragmentation processes}

Recall that when $\alpha \in (0,2) \setminus \{1\}$, it is possible to define a real-valued stable process $X$ of index $\alpha$ with no negative jumps. For such a process $X$ started from $0$,
the processes $(X_{ct})_{t \ge 0}$ and $(c^{1 / \alpha} X_t)_{t \ge 0}$ have the same law for all $c >0$. \jason{The ordered collection of jumps of $X$ form a Poisson point process with intensity $dt dh / h^{1+\alpha}$ on $(0, \infty) \times (0, \infty)$, and the obtained process will make a jump $h_i$ at time $t_i$ for each $(t_i, h_i)$ in 
this point process}. The sum of all the small jumps accomplished before time $1$ say is infinite when $\alpha \in (1,2)$ (which is in fact the case that will be relevant to the present paper),
but the process is nevertheless well-defined as a deterministic function of the Poissonian collection of jumps via  L\'evy compensation (it is the limit as $\eps \to 0$ of the process obtained by summing 
all the jump of $X$ of size at least $\eps$ with the deterministic function $-c_\eps t$ for a well-chosen $c_\eps$ that goes to $\infty$ as $\eps \to 0$). 

By considering the linear combination $X:= a'X' - a'' X''$ of two independent such stable processes $X'$ and $X''$ with no negative jumps for $a' \ge 0$, $a'' \ge 0$, one then gets a general stable process. We will say that $u:=a'/a'' \in [0, \infty]$ is the ratio between the intensity of its positive/upward jumps and the intensity of its negative/downward jumps. When $u=1$, then $X$ is a symmetric stable process.  

In the context of fragmentation-type processes, it appears natural to consider variants of these stable processes that are tailored so that they remain positive at all times. 
The idea is that if the process is at $x$, then the rates at which it jumps to $x+h$ and $x-h$ will not be proportional to $1/h^{1+ \alpha}$ anymore, but will also depend on $x$. 
Let us illustrate this with the following example that will be relevant in the present paper. 
Consider $\alpha \in (1,2)$ and a general stable L\'evy process $X$ started from $x_0 >0 $ 
with index $\alpha$ and $u \in [0, \infty]$ as defined above. Recall that when the process $X$ is at $x$, the rate at which it jumps to $x+h$ and $x-h$ respectively does not depend on $x$ and is $a' / h^{1+\alpha}$ and $a''/h^{1+\alpha}$.  It is possible to define a variant $Y$ of $X$ started from some positive $y_0$, such that when the process is at $y$, the rate of jumps to $y+l$ and to $y-l$ respectively with the rates
\[ \frac {a'y^{\alpha+1}} {l^{\alpha+1} (y+l)^{\alpha+1}} \quad\text{and}\quad \frac {a''y^{\alpha+1}} {l^{\alpha+1} (y-l)^{\alpha+1}} \one_{l < y/2}. \]
Note that this tends to diminish the rate of the positive jumps and to favor the negative ones (of size smaller than $l/2$) when compared to $X$, but that for very small $|l|$, the rates of jumps of $Y$ are close to those of $X$. It turns out that it is possible to define such a processs $Y$, with the property that for all positive $r$ and $T$, if $E_{T,r}$ denotes the event that the process remains larger than $r$ up to time $T$, then the law of $(Y_t, t \le T)$ on the event $E_{T,r}$ is absolutely continuous with respect to that of $(X_t, t < T)$ on $E_{T, r}$ (actually, the proper way to define $Y$ is via its Radon-Nikodym derivative with respect to the law of of $X$ on each $E_{T, r}$, and then to check that it has the desired jump rates). This process $Y$ is then well-defined up to its first hitting time of $0$ (which can be infinite).  This process can have positive jumps of any size, but it can never more than halve itself during a jump.

The previous condition $l < y/2$ in the negative jumps of $Y$ comes from the fact that the process $Y$ is in fact naturally related to a fragmentation process that can be understood as follows: The negative jumps $y \mapsto y-l$ correspond to the splitting of particle of mass $y$ into two particles of masses $l$ and $y-l$ respectively. Then, the two particles will evolve independently. In this setup, the process $Y$ corresponds to the fact that at each such splitting within this random branching process $\CT$, one follows only the evolution of the largest of the two offspring. But, by using a countable collection of independent copies of $Y$, it is then also possible to define the evolutions of all the other offspring and their offspring -- and to define an entire fragmentation process $\CT$. This is a random labeled tree-like structure, of a type that has been subject of extensive studies, see e.g., \cite{Bertoin,bbck} and the references therein, and that is also sometimes referred to as multiplicative cascades, and related to branching random walks as initiated by Biggins \cite{Biggins}.

The positive jumps of $Y$ can just be kept as they are (the particle of mass $y$ becomes a particle of mass $y+l$) as in $\CT$ above, or alternatively also viewed as a splitting into the creation of two particles of masses $y+l$ and $l$ respectively that then evolve independently (and this therefore gives rise to a larger tree structure $\wt \CT$). 

It should be noted that this tree $\CT$ appears already in the asymptotic study of peeling processes on some planar maps, see \cite{ccm2017lengths,bbck} and the references therein.  

\subsection {Results of the present paper}

\subsubsection{CLEs on quantum disks}

We now begin to describe the results of the present paper about explorations of a $\CLE_\kappa$ drawn on an independent quantum disk. Here and throughout this introduction, we suppose that $\kappa \in (8/3, 4)$ and we define (and we will use these relations throughout this introduction)
\begin{equation}
\label{eqn:basic_relations}
\gamma := \sqrt {\kappa}, \quad \kappa' := \frac{16}{\kappa}, \quad\text{and}\quad \alpha := \frac{4}{\kappa}.	
\end{equation}

Consider a $\gamma$-quantum disk of boundary length $y_0$ parameterized by a simply connected domain $D$  -- and let $m$ be a boundary point, chosen uniformly according \jasonr{to} the LQG boundary length measure.  Consider on the other hand an independent $\CLE_\kappa$ in $D$.

One can then define a CLE exploration tree starting from $m$ (recall \cite{msw2016notdetermined} that tracing such a tree involves yet additional randomness as it loosely speaking amounts to exploring CPI percolation paths within the CLE carpet) as described above -- which gives rise to a $\CLE_\kappa$ exploration tree, or equivalently to the $\SLE_\kappa (\kappa -6)$ branching tree. One can recall that independent $\SLE_\kappa$ and $\SLE_{\kappa'}$-type curves drawn in a $\gamma$-quantum disk both have a quantum length (i.e., that can be interpreted as its length with respect to the LQG structure) \cite{SHE_WELD,dms2014mating}. So in particular: 

\begin{itemize}
\item Each CLE loop, which is an $\SLE_\kappa$ type loop will have its quantum length. 
\item Each CPI branch, which is an $\SLE_{\kappa'}$ type curve can be parameterized by a constant multiple of its quantum length, which is implicitly what we will do in the following paragraphs. 
\item The boundaries of the connected components of the complement of an $\SLE_\kappa (\kappa -6)$ curve at a given time, which are $\SLE_\kappa$-type curves will also be equipped with their quantum length measure. [In the sequel, we will then just refer to this quantum boundary length as the boundary length]. 
\end{itemize}

In particular, we see that:

\begin{itemize}
\item When the CPI hits a CLE loop for the first time, then if one attaches the whole CLE loop at once, one splits a domain with boundary length $l$ into two domains with boundary length $l+y$ and $y$, where $y$ denotes the quantum length of the CLE loop (the domain $O_l$ with boundary length $y$ would correspond to the inside of the discovered CLE loop), see Figure~\ref{split}.

\begin{figure}[ht!]
\includegraphics[scale=.5]{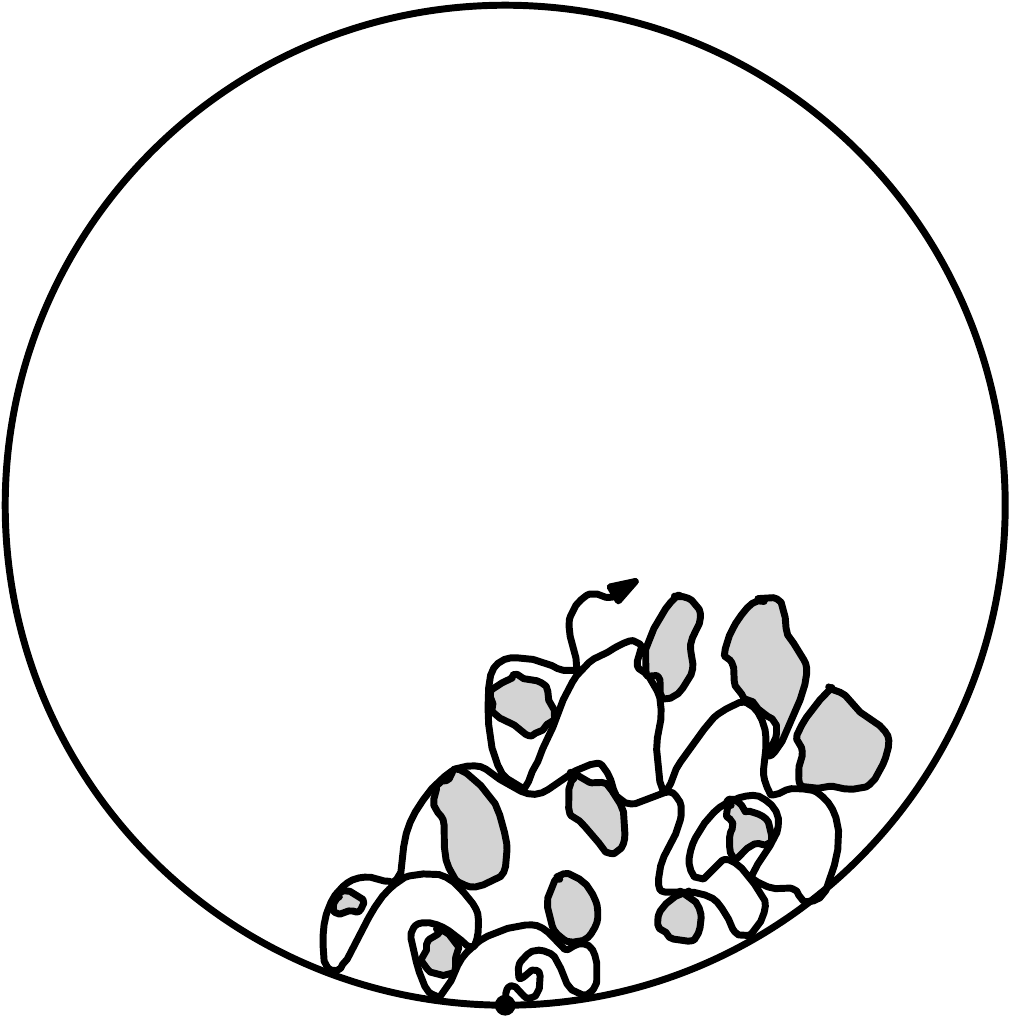} 
\includegraphics[scale=.5]{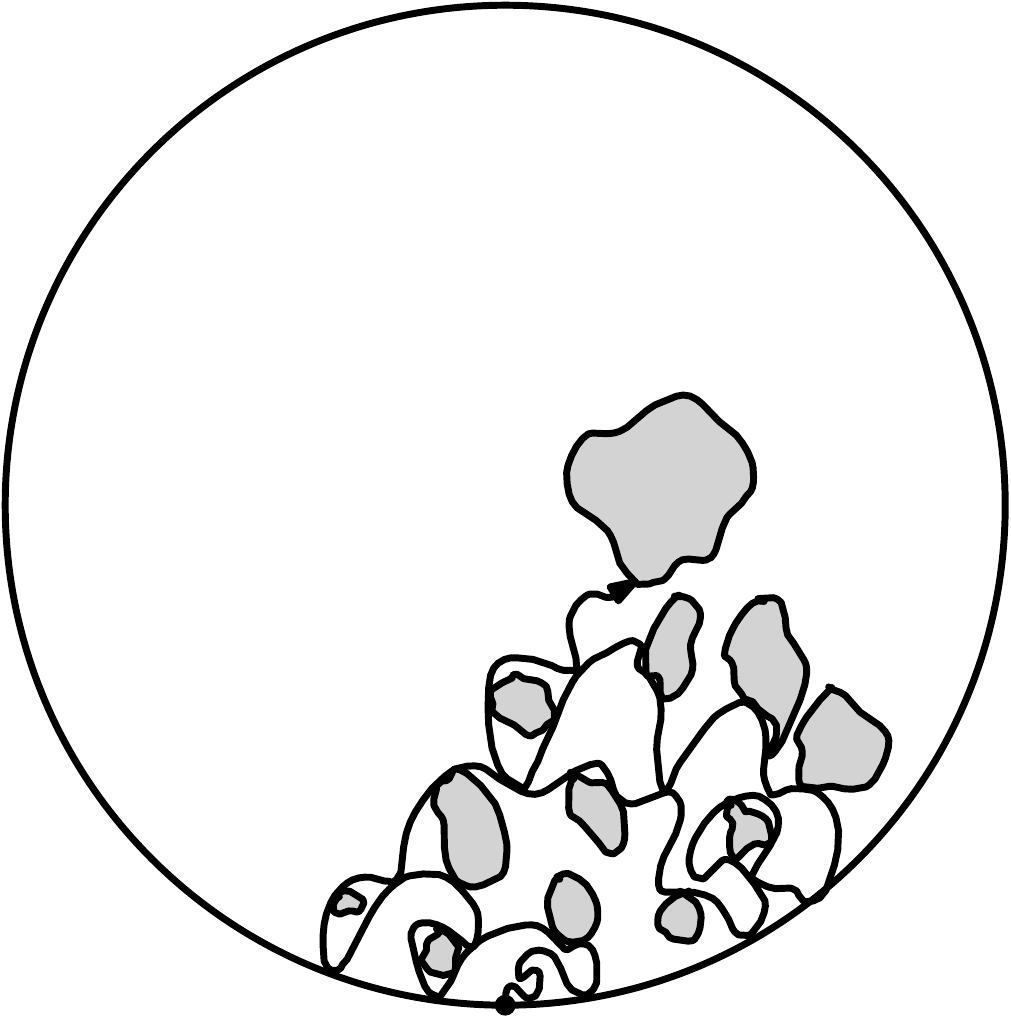} 
	\caption{When the trunk of the CPI from $-i$ to $i$ hits a CLE loop, the boundary length of the remaining to be explored domain makes a positive jump}
	\label{split}
\end{figure}

\item When the CPI (i.e., the $\SLE_{\kappa'}(\kappa'-6)$ process which is the trunk of the $\SLE_\kappa (\kappa -6)$) disconnects the remaining-to-be-explored domain into two pieces (this can happen for instance when it hits $\partial D$), then it splits the remaining-to-be explored domain of boundary length $y$ into two domains of boundary lengths $y_1$ and $y_2$ where $y_1 + y_2 = y$ (i.e. into one domain with boundary length $l$ and one with boundary length $y-l$, where $l < y/2$), see Figure~\ref{split2}.

\begin{figure}[ht!]
\includegraphics[scale=.5]{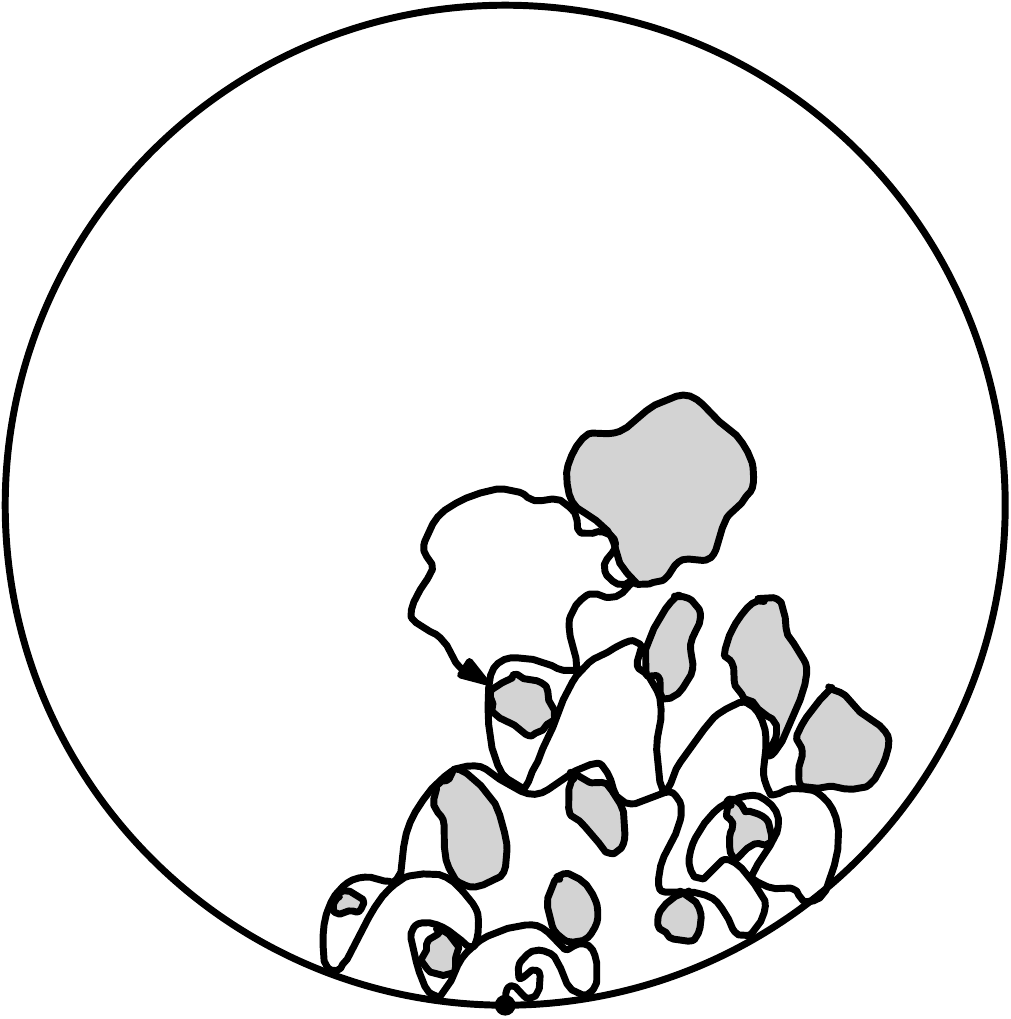} 
\includegraphics[scale=.5]{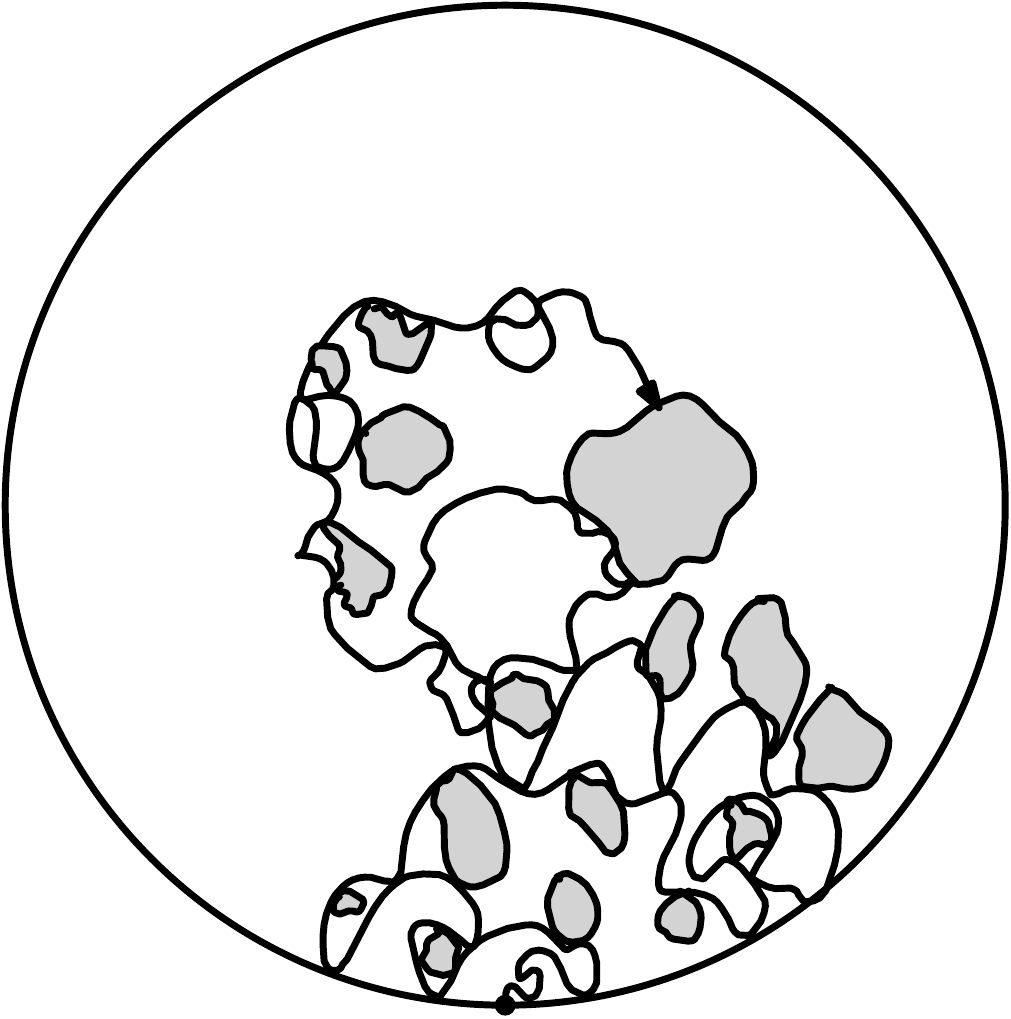} 
\includegraphics[scale=.5]{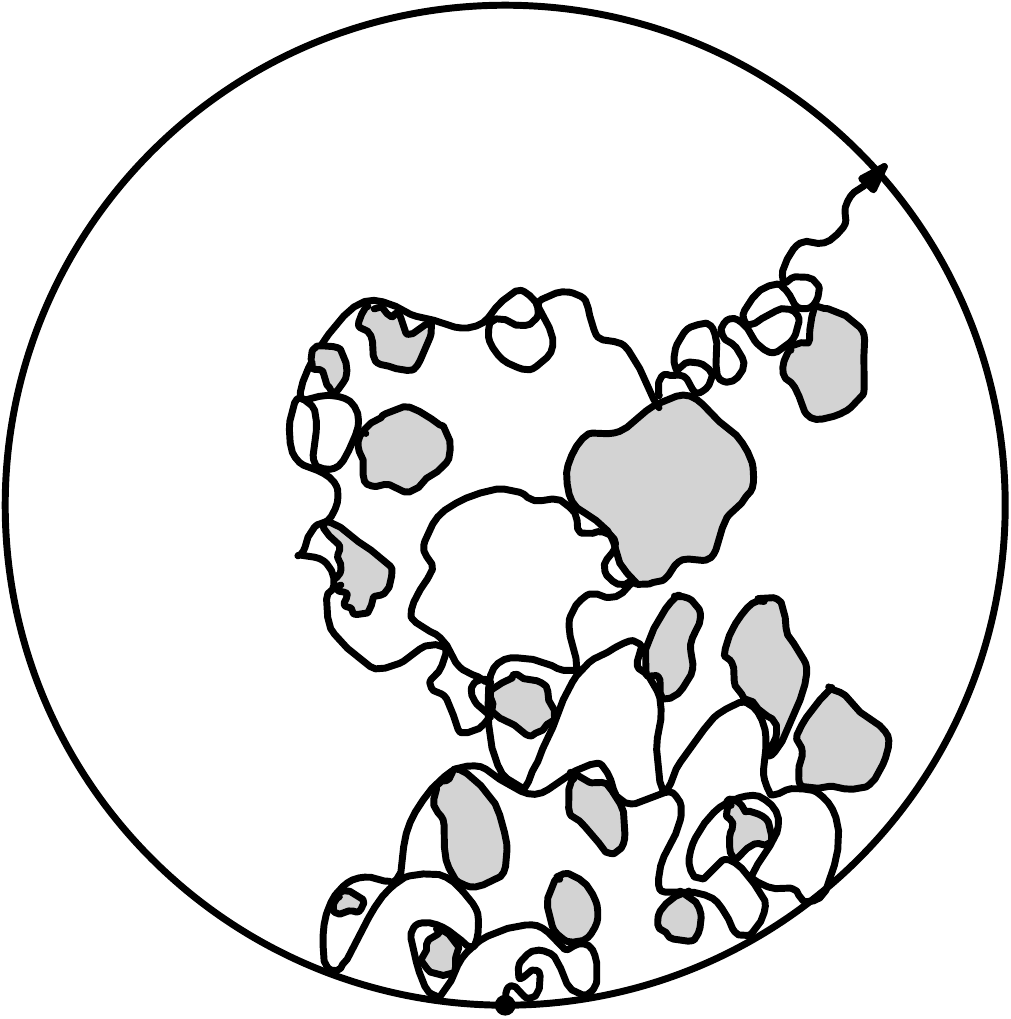} 
	\caption{Examples of configurations when the exploration process splits the to-be-explored domain into two pieces}
	\label{split2}
\end{figure}
\end{itemize}

In this way, when one discovers this ``CLE on LQG'' structure via the CPI tree (and when one hits a CLE \jason{loop} for the first time, one discovers it entirely), then one gets a fragmentation tree type-structure. It turns out to be handy to use the quantum length of the CPI branches as time-parameterizations for the tree structure. At any time along one branch, the label is the boundary length of the remaining to be discovered domain that this branch is currently discovering. If one follows one branch of the tree, these labels will have positive jumps that correspond to the discovery of a $\CLE_\kappa$ loop, and negative jumps that correspond to times at which the CPI splits the remaining-to-be-discovered domain into two pieces. 

Our first main  statement goes as follows:
\begin{theorem}[Exploration tree of a CLE on a quantum disk] 
\label{thm1}
The law of the obtained fragmentation tree (obtained from drawing a CLE exploration on an independent LQG disk as just described) is exactly that of the fragmentation tree-structure $\CT$ described above, with 
\[ u= -  \cos ( \pi \alpha) = - \cos ( 4 \pi / \kappa). \]
Furthermore, conditionally on $\CT$, the quantum surfaces encircled by the $\CLE_\kappa$ loops are independent quantum disks (the boundary length of which are given by the jumps in $\CT$). 
\end{theorem}
As a consequence, one for instance gets that the joint law of all the quantum lengths of all $\CLE_\kappa$ loops in a quantum disk is the same as the collection of all the
positive jumps appearing in $\CT$. 

We now give an equivalent reformulation of this theorem in \jason{terms} of the law of one particular branch of the exploration tree. 
Suppose that one traces the branch of the exploration tree (parameterized by its quantum length) defined by the following rule: Whenever the trunk disconnects the remaining-to-be discovered domain into two pieces, one continues exploring the branch of the tree in the 
domain with largest boundary length. At such a time, the boundary-length $L_t$ of the remaining to be discovered domain makes a negative jump from $y$ to some $y-l$ with $l < y/2$ --- and when the CPI discovers a CLE loop for the first time, this process has a positive jump of size given by the quantum length of that loop.

\begin{theorem}[Branch of the  CLE exploration tree  on a quantum disk]
\label{thm2}  \ww{Up to a linear time-change}, the law of this process $L$ is the law of the process $Y$ described above, with the relation 
 $u = - \cos ( \pi \alpha)$.  Furthermore, conditionally on the collection of jumps of $L$, the law of the quantum surfaces that are cut out by the trunk of the $\SLE_\kappa (\kappa -6)$ process (corresponding to the negative jumps of $Y$) and of the inside of the discovered CLE loops (that correspond to the positive jumps of $Y$) are all independent quantum disks. 
\end{theorem}

\subsubsection{The natural LQG area measure in the CLE carpet}

The previous description of the exploration mechanism in a quantum disk makes it possible to relate many questions to computations for these labeled L\'evy trees, for which there exists a rather substantial literature, including recent results motivated by the study of planar maps with large faces \cite {lgm2011large_faces,bck,bbck,dadoun}. 

We illustrate this here by constructing and discussing  
the natural LQG-measure that is supported on the $\CLE_\kappa$ carpet in the quantum disk.
This measure conjecturally corresponds to the scaling limit of the uniform measure on a planar map with large faces, as studied for instance in \cite{lgm2011large_faces}, and it is therefore an interesting and potentially very useful object to have at hand, if one tries to study the scaling limit of such maps -- see \cite {bbck,bck} for results in this direction. 

Let us first recall that  ``Kesten-Stigum''-type results on multiplicative cascades as developed in \cite {bbck,dadoun} make it possible construct a ``natural measure'' $\wt {\mathcal Y}$ on the boundary of ${\mathcal T}$; this measure is already called {\em the intrinsic area measure} in \cite {bbck} (this terminology is also motivated by the fact that it arises as the limit of the counting measure of some planar maps with large \jason{faces}, when encoded in a tree).

However there are twists that need to be overcome to check that this measure $\wt {\mathcal Y}$ corresponds to a natural quantum measure in the CLE carpet. The main one is arguably that in the continuun, the tree $\CT$ -- and therefore the measure $\wt {\mathcal Y}$ on its boundary, as well as the map from the tree and the carpet -- are functions of the additional randomness that is provided by the CPI exploration tree within the CLE carpet. It is therefore a priori not clear that the obtained measure in the CLE carpet that would correspond to $\wt {\mathcal Y}$ is a deterministic function of the CLE carpet and the LQG measure only. 

\begin{theorem}[The natural LQG measure on the CLE carpet]
\label {thmmeasure}
Consider a $\CLE_\kappa$ drawn on an independent $\jasonr{\gamma}$-quantum disk, for $\kappa \in (8/3, 4)$, embedded into the unit disk $D$. Consider then a CPI exploration tree of the CLE$_\kappa$ started from a boundary-typical point, that in turn defines a fragmentation-tree $\CT$ as in Theorem \ref {thm1}, and a measure $\wt {\mathcal Y}$ on its boundary. Then, there exists almost surely a unique measure ${\mathcal Y}$ on the CLE$_\kappa$ carpet with the property that its image in $\CT$ is well-defined and is the measure $\wt {\mathcal Y}$. Furthermore, this measure is a deterministic function of the CLE and the LQG measure only.    
\end{theorem}

We will prove this theorem by showing that the measure ${\mathcal Y}$ has necessarily the property that for a wide class of open sets $O$ (that includes $D$ itself), 
${\mathcal Y}(O)$ is the limit in probability of  $\eps^{\alpha + 1/2} N_{[\eps, 2 \eps]} (O)$, where $N_{[\eps, 2 \eps]} (O)$ is the number of $\CLE_\kappa$ loops entirely contained in $O$ that have a quantum length greater in $[\eps, 2 \eps]$. We then argue that this property (which is instructive in its own right) does almost surely characterize the measure ${\mathcal Y}$ uniquely. This in turn shows that ${\mathcal Y}$ does indeed not rely on the randomness coming from the CPI (because the quantities $ N_{[\eps, 2 \eps]} (O)$ do not rely on this randomness). 

Note that when one applies the same ideas using $\wt \CT$ instead of $\CT$ (i.e., exploring the entire domain $D$ instead of just the CLE carpet), one constructs the (usual) LQG area of the quantum surfaces using \jason{martingales which are similar to those used to construct the measure on~$\CT$}. 

\subsubsection {CLEs on quantum half-planes} 

We will derive these results on CLE on a quantum disk as consequences of results for CLE on a quantum half-plane that we will now briefly describe.  The results on a quantum half-plane somehow correspond to what happens in the quantum disk case when one zooms into the neighborhood of the starting point. 

Here, one considers a quantum half-plane (also known as a particular quantum wedge) which is another quantum surface that can be naturally defined in a simply connected domain.  This time, the choice of the harmonic function 
is such that the total quantum area of the domain and the total boundary length are infinite, but only in the neighborhood of one special boundary point (which is chosen to be $\infty$ when the domain is the upper half-plane).
One then considers: 

\begin{itemize}
\item A $\gamma$-quantum half plane, so that $0$ is a ``typical'' point for the boundary-length measure. 
\item An independent $\SLE_\kappa (\kappa -6)$ process from $0$ to $\infty$. 
The process traces some loops, that can be viewed as the $\CLE_\kappa$ loops encountered by the $\SLE_{\kappa'} (\kappa' -6)$ trunk (for $\kappa' = 16/ \kappa$). 
\end{itemize}

This CLE exploration can then be parameterized by the quantum length of the CPI/trunk, and along the way, it will: 

\begin{itemize}
\item Discover $\CLE_\kappa$ loops with quantum lengths $L_{t_i}$ at a countable random dense collection of times $(t_i)$. We can formally attach at once the entire loop to the CPI at this time $t_i$. When equipped with the LQG measure and with marked boundary point given by the point at which the loop is attached to the trunk, one obtains a quantum surface with one marked point. 
\item When the trunk bounces on the boundary of the remaining-to-be-explored region, it will disconnect some parts of $\h$ from $\infty$, by leaving them to its right or to its left. Again, the disconnected components (with marked point given by the disconnection point) is a quantum surface with one marked boundary point.
\end{itemize}

\begin{theorem}[Exploration of a CLE on a quantum half-plane] 
\label{mainthm<4}
The three ordered families of quantum surfaces (cut out to the left, cut out to the right and inside the traced loops) have the law of 
three independent Poisson point processes of quantum disks, defined under multiples $a_{l,-}$, $a_{r,-}$ and $a_+$ of the same natural infinite measure on quantum disks, where
 $a_{l,-} = a_{r,-}$ and $a_+  = - 2 a_{l,-} \cos (4 \pi /\kappa)$.
\end{theorem}

In other words, the CLE exploration on a quantum half-plane is related to the jumps of a $(4/\kappa)$-stable process with asymmetry parameter $u= - \cos (4 \pi / \kappa)$. This stable process is allowed to be negative, and can 
be thought off as the ``variation'' of the boundary length of the remaining-to-be explored region containing $\infty$. 

\subsubsection{Results for other CLE explorations}

We have so far only discussed the exploration of a $\CLE_\kappa$  by a totally asymmetric $\SLE_\kappa (\kappa- 6)$ processes, that correspond to the fact that all $\CLE_\kappa$ loops
are traced with the same orientation -- or equivalently, lie on the same side of the trunk. In the percolation interpretation of the CPI, this corresponds to the fact that all the $\CLE_\kappa$ loops are declared to be closed 
for the considered percolation mechanism. 

As explained and proved in \cite{SHE_CLE,SHE_WER_CLE,ww2013conformally}, this is not the only possible natural way to proceed in order to explore a CLE in a conformally invariant way. 
One can choose a parameter $\beta \in [-1, 1]$ and decide at each time at which one discovers a CLE loop, to trace it clockwise with probability $1-p := (1-\beta) / 2$ or counterclockwise with probability $p:=(1+ \beta)/2$ (these choices are made independently for each loop). This defines the so-called $\SLE_\kappa^\beta (\kappa-6)$ processes -- the previous totally asymmetric $\SLE_\kappa (\kappa -6)$ corresponds to the case $\beta =1$. In the CPI percolation interpretation, it corresponds to the fact that loops can be closed or open for the considered percolation.  
As explained in \cite{ww2013conformally}, this procedure allows one to trace loops of a $\CLE_\kappa$ using $\SLE_\kappa^\beta (\kappa-6)$ processes. Also, it has been shown in \cite[Section~8]{cle_percolations} (among other things) that this exploration indeed traces a continuous path. The continuous path that one obtains when one 
erases all the $\CLE_\kappa$-traced loops from this path is then a continuous curve, called the {\em trunk} of the $\SLE_\kappa^\beta (\kappa -6)$. When $\beta \in (-1,1)$, some of the loops lie on the 
left-hand side of the trunk, and some of the loops lie to its right. It is intuitively clear (and this is made rigorous in \cite{cle_percolations}) that in some appropriate sense, the 
respective proportion of loops attached to the left and to the right of the trunk are  $(1- \beta)/2$ and $(1+\beta)/2$. 

One can therefore wonder what type of quantum surfaces will be cut out by this process when drawn on a quantum half-plane or a quantum disk. Note that this time, one will 
obtain four types of quantum surfaces because now, the CLE loops that one traces can be on the right-hand side of the trunk (the counterclockwise loops) or to the left-hand side of the trunk. 
As one could have somehow expected, it turns out that $\beta$ only influences the left to right ratio of the positive jumps. In the   case of the exploration of a quantum half-plane goes, the exact statement goes as follows: 

\begin{theorem}[Results for other explorations] 
\label{mainthm<4-22}
The four ordered families of quantum surfaces (cut out to the left, cut out to the right, inside the left loops, inside the right loops) have the law of 
four independent Poisson point processes of quantum disks, respectively defined under $a_{l,-}$, $a_{r, -}$, $a_{l,+}$ and $a_{r,+}$ times the measure on quantum disks, where 
 $a_{r,-} = a_{l,-}$, $a_{l, +} = (1-p) a_+$,  $a_{r, +} 
 = p a_+$ and  $a_+ = - 2a_{l, -} \cos (4 \pi /\kappa)$.
\end{theorem}

Actually, the knowledge of these four Poisson point processes allows one to reconstruct both the \jasonr{quantum} half-plane and the $\SLE_\kappa^\beta (\kappa -6)$ drawn on it.  So, in a nutshell, for these Poisson point processes of disks, the only difference with the totally asymmetric case is that each time one discovers a quantum disk that corresponds to a CLE loop, 
one tosses an additional $(1+ \beta)/2$ versus $(1- \beta)/2$ coin to decide on which side of the trunk it will be. 
The similar result holds true for asymmetric explorations of quantum disks. 

One consequence of Theorem~\ref{mainthm<4-22} 
is that one can complete the description of the loop-trunk 
of \cite{cle_percolations}. It is worthwhile to stress that this is a statement that does not 
involve any LQG, but that at present, we know of no way to derive it other than the one that we will give in this paper and that relies heavily on the interplay between CLE and LQG.  
Let us first remind the reader of one result from \cite{cle_percolations} about the law of the trunk of an $\SLE_\kappa^\beta (\kappa -6)$ process:  
\begin{theo}[Theorem~7.4 from \cite{cle_percolations}]
\label{thm:cle_perc_trunk}
  For all $\kappa \in (8/3, 4)$ and $\beta \in [-1,1]$, there exists $\rho' = \rho' (\beta, \kappa) \in [\kappa'-6,0]$ such that 
the law of the trunk of a $\SLE_\kappa^\beta (\kappa -6)$ process is a $\SLE_{\kappa'}(\rho'; \kappa'-6- \rho')$ process for $\kappa' = 16/\kappa$.  
\end{theo}

The fact that these $\SLE_{\kappa'}(\rho'; \kappa'-6- \rho')$ processes for $\rho' \in [\kappa'-6,0]$ show up is not really surprising here: They are the natural target-invariant variants of $\SLE_{\kappa'} (\rho')$ processes (in the same way as $\SLE_{\kappa'} (\kappa'-6)$ is the target-invariant variant of $\SLE_{\kappa'}$ -- see for instance \cite{cle_percolations} for a more detailed discussion). Recall also that in \cite{cle_percolations} the conditional law, given the trunk, of the 
collection of $\SLE_\kappa$ that are attached to the trunk is described.  
Except when $\beta \in \{ -1,0 ,1\}$ that correspond to the totally asymmetric cases and to the symmetric case, 
we did not derive in \cite{cle_percolations} the values of $\rho' (\beta, \kappa) $.
This following theorem provides the relation between $\rho'$ and $\beta$: 
\begin{theorem}[The law of the trunk of $\SLE_\kappa^\beta (\kappa -6)$]
\label{thm:beta_rho}
The value of $\rho' \in [\kappa' -6, 0]$ in terms of $\beta \in [-1,1]$ in Theorem~\ref{thm:cle_perc_trunk} is determined by the relation (we write $1/0= \infty$ so that this covers also the case $\beta=-1$): 
\[
 \frac {1-  \beta}{1 + \beta} = \frac { \sin (-\pi \rho' /2 ) } { \sin ( -\pi (\kappa'-6-\rho') / 2)} .\]
 \end{theorem} 
This therefore completes the full description of the loop-trunk decomposition of these asymmetric CLE exploration mechanisms. 

The  readers acquainted with subtleties of stable processes may find some similarities between this result and the formulas describing aspects of the ladder height processes of stable processes. This is not a coincidence, as we will actually derive Theorem~\ref{thm:beta_rho} using such formulas.

\subsection{Remarks, generalizations, outlook}

\begin{enumerate}
\item
This paper \ww{has a counterpart \cite{mswupcoming} that describes} the structures that one obtains when one draws $\CLE_{\kappa'}$-explorations 
(for $\kappa' \in (4,8)$) on top of the corresponding quantum ``half-planes'' and ``disks''. The results are formally quite similar to the ones of the present paper, but some of the arguments differ (to start with, the half-planes and disks are of a different type). 

\item
The present paper (as well as \cite{mswupcoming}) lays the groundwork for stronger convergence results for scaling limits of discrete models to $\SLE$ on LQG to be made. In some sense, our results show that from the perspective of the continuum models that should appear in the scaling limit when one considers $O(N)$ models on well-chosen planar maps (or related models), the features that allowed physicists to use their {\em quantum gravity ideas} are indeed valid.  It should therefore not be surprising that some of our formulas mirror results that appear in the study of some special planar maps, such as the ones arising in  \cite{BBG1,ccm2017lengths,bbck} (see also 
\cite {budd,curienrichier} for further related results on the planar maps side). This can be explained by the fact that some of the  discrete peeling type processes used in the study of planar maps should indeed give rise to these loops on trunk processes on independent LQG in the scaling limit.
 
\item To be more specific, the results of this article open the doors to using the $\QLE$ approach (\cite{qlebm} and the references therein) to construct two particular metrics inside of the $\CLE$ carpet which correspond to the scaling limit of random planar maps with large faces determined by Le Gall and Miermont \cite{lgm2011large_faces}.  
\end{enumerate}

\subsection*{Acknowledgements}

JM was supported by ERC Starting Grant 804166 (SPRS). SS was supported by NSF award DMS-1712862.  WW was supported by the SNF Grant 175505. He also acknowledges the hospitality of the University of Cambridge, where part of this work has been written. 
We also thank an anonymous referee as well as Nina Holden and Matthis Lehmk\"uhler for very useful comments. 

\section{Background} 
\label{S.Background}

\subsection {Quantum wedges (thick and thin), half-planes and disks}

Let us very briefly survey the definition of the quantum surfaces that we will be using in the present paper.  The discussion that we give here will be somewhat informal because the precise definitions of these surfaces will not in fact be needed in this work.  We refer the reader to \cite[Section~1, Section~4]{dms2014mating} for a more detailed treatment.

One starting point is to consider the GFF in a simply connected domain with free boundary conditions. As this field turns out to be conformally invariant, it is sufficient to define it in $\h$, where it can be viewed as the Gaussian process $h_F$ indexed by the space of bounded measurable functions with compact support and mean zero (w.r.t.\ Lebesgue measure), with covariance given by 
\[ \E [ h_F(\phi) h_F(\psi) ] = \int  \phi(x) \psi(y) \left(\log \frac 1 {|x-y|} + \log \frac 1 {|x-\overline y|} \right) dx dy. \]
One way to interpret the fact that $h_F$ is defined on the set of functions of zero mean, is that $h_F$ is only defined ``up to constants'' (so in fact, the object that is defined and studied is the generalized function $\nabla h_F$). One can of course also define a proper generalized function $h_F$ and call it a GFF with free boundary conditions if the process $(h_F (\phi))$ has the law described above (on the space of functions with zero mean). 

For each choice of $\gamma \in (0,2)$, one can almost surely associate to each such generalized function $h_F$ a measure in the upper-half plane and a measure on the real line (i.e., that is thought of as the boundary of $\h$), for instance via a regularization procedure, that can be interpreted as the measures with densities $\exp (\gamma h_F)$ and $\exp ( \gamma h_F / 2)$ with respect to the area and length measure in the half-plane and the real line respectively. 
Such area and length measures can therefore also be defined for variants $h$ of the free-boundary GFF, such that the law of $h$ is locally absolutely 
continuous with respect to that of $h_F$ (or to that of the sum of $h_F$ with some random continuous function, or that of some randomly scaled version of it). 

The quantum wedges are variants of the free boundary GFF involving two marked boundary points. It is at first sight convenient and natural, when working in $\h$, to take these two boundary points to be $0$ and $\infty$. Loosely  speaking (we will make this more precise in a moment), the quantum wedges correspond to adding a constant times $\log |z|$ to the free boundary GFF.  In this context, keeping in mind that it is the area measure that is conformally covariant, it appears however somewhat more natural to work in a bi-infinite strip (with marked boundary points at both ends of the strip) rather than in $\h$ (so we just take the image of $\h$ under the logarithmic map).  In that setting, when $h_F$ is a free-boundary GFF in the strip $\strip = \R \times (0, \pi)$, 
we can consider $u_F(r)$ to be the mean-value of $h_F$ on the vertical segment $(r, r +i \pi)$. This process $(u(r/2)- u(0))_{r \in \R}$ turns out to \jason{be} a two-sided (i.e., bi-infinite) Brownian motion \jason{normalized to take the value $0$ at time $0$}.
Furthermore, if one defines the function $u_F$  on the strip by $\wt{u}_F (x+iy) = u(x)$, then the process $v_F := h_F - \wt{u}_F$ is independent of $u_F$. In other words, adding a given function $f( \re (z))$ amounts to twisting $u_F$ only and to leave $v_F$ unchanged.  

In order to define the wedge variant, one essentially just has to choose $u_F$ to be a Brownian motion with drift instead of a Brownian motion. However, it is well-known that a bit of care is needed when one defines a two-sided Brownian motion with drift\footnote{One usual way to define a two-sided Brownian motion with drift is to declare that time $0$ is the first time at which the Brownian motion with drift hits $0$.  Then, for positive times, the process is just a Brownian motion with drift, while on the negative side, it is Brownian motion with drift and conditioned not to hit the origin.} and one also needs to keep in mind that the free boundary GFF is defined only up to an additive constant.  

Let us briefly recall how to make sense of measures on one-dimensional drifted Brownian motion paths, starting from $- \infty$ at time $- \infty$. Note that here, we are somehow trying to make sense of measures on unparameterized paths. Since the Brownian motion's quadratic variation is constant, it is always possible to recover the difference between two times on the trajectory, but there is still one degree of freedom (for instance one can decide which point on the trajectory corresponds to time $0$).

(i) When the drift $a$ is positive, the natural measure on drifted Brownian paths is a probability measure, and the path will come from $-\infty$ (at time $-\infty$) and go to $+\infty$ (at time $+ \infty$). 
This makes it possible to choose time $0$ to be the first time at which the path hits $0$. In that case, $(f(-t))_{t \ge 0}$ and $(f(t))_{t\ge 0}$ will be independent, the latter being just a Brownian motion 
with drift $a$, while the former will be a Brownian motion with drift $a$, but conditioned to never hit $0$ (and there are several equivalent easy ways to make sense of this). 
An equivalent way to define this process is to use the fact that drifted Brownian motion can be viewed as a time-changed Bessel process \jasonr{(see, e.g., \cite[Proposition~3.4]{dms2014mating})}. Here (with the positive drift), this means that one can start 
with a Bessel process $Y$ of dimension $\delta$ greater than $2$ (which is a process that starts from $0$ at time $0$, never hits $0$ again, and tends to $\infty$ as time goes to $\infty$), and to view 
the drifted Brownian motion as a time-change of $\exp (Y)$ (the time-change ensures that the quadratic variation of the obtained process is constant, and the scaling property of the Bessel process 
then ensures that the obtained process behaves like a drifted Brownian motion). Then, again, one can choose where time $0$ lies based on this drifted Brownian trajectory itself, for instance the first time at which it hits $0$. 
For each positive value of $a$, one can then simply define the {\em thick quantum wedge}
to be the surface obtained by adding the field $v_F$ to this drifted Brownian motion. It defines an area measure 
in the bi-infinite strip, it has two special boundary points $- \infty$ and $+ \infty$, and the area measure as well as the boundary length measure are finite in the neighborhood of $-\infty$ but infinite in the neighborhood of $+ \infty$. 

(ii) It is also natural to define measures on bi-infinite Brownian paths coming from $-\infty$ but with {\em negative} drift. This gives rise to infinite measures of Brownian-type paths that tend to $-\infty$ when time goes to $-\infty$ and also when time goes to $+ \infty$. One can for instance define this as the appropriately renormalized limit when $M \to -\infty$ of the law $P_M$ of Brownian motion with negative drift $a$ started from $M$. For this limit to have a limit, one has to renormalize this probability measure $P_M$ by (for instance) the probability that this paths hits $0$. In other words, one is looking at 
the excursion measure away from $-\infty$ by the negatively drifted Brownian motion. 
The same alternative description via Bessel processes as above turns out to be handy as well. This time, the Bessel processes will have a dimension $\delta \in (0,2)$ and the infinite measure that one starts with 
is simply the infinite measure on excursions away from $0$ by this Bessel process.  
In any case, this leads to an infinite measure on surfaces with finite area and finite boundary length. The obtained measure is then self-similar in the sense that the measure on quantum surfaces that one 
obtains by multiplying the area measure by a given constant will be a multiple of the initial measure (the scaling factor can be read off from the drift of the Brownian motion or from the dimension of the Bessel process). 
This makes it then very natural to consider an ordered bi-infinite Poisson point process of quantum surfaces with intensity given by this infinite measure. The obtained infinite chain of quantum surfaces is then 
called a {\em thin quantum wedge} -- to define it, one can therefore start with a bi-infinite Bessel process of dimension $\delta \in (0,2)$ (that is equal to $0$ at time $0$) and defined on both positive and negative times. 

The \jasonr{relationship} between the dimension $\delta$ of the Bessel process and the so-called weight $W$ of the quantum wedge (one can view this formula as a definition of the weight $W$) is 
\begin{equation}
\label{eqn:bessel_dim_wedge_weight}
\delta =  1 + \frac {2W}{\gamma^2}
\end{equation}
\jasonr{(see, e.g., \cite[Table~1.1]{dms2014mating}).}  The threshold between thin and thick wedges is then at $W=\gamma^2/2$, i.e.,  $\delta=2$.

It is well-known that if one ``conditions a Brownian motion with positive drift $a$ to tend to $-\infty$'' (it is easy to make rigorous sense of this conditioning), then one obtains a Brownian motion with the negative drift $-a$. This then corresponds to the classical relation between Bessel processes of dimension $\delta$ and $4-\delta$ (a Bessel process of dimension $\delta > 2$ ``conditioned to hit $0$'' will be a Bessel process of dimension $4 - \delta$). In terms of wedges, this corresponds to a natural ``duality'' relation between the thick quantum wedge of weight $W \in (\gamma^2 /2 , \gamma^2)$ and the beads of a thin quantum  wedge of weight $W'= \gamma^2 - W$. We will come back to this later.

There are two special quantum surfaces that will play an important role here: 

\begin{itemize}
\item For each $\gamma \in (0, 2)$, there is one very special thick quantum wedge for $W=2$, that we will here call a quantum half-plane. Roughly speaking, this is the case where the boundary point $-\infty$ in the strip is not a particularly special boundary point of the quantum surface. For instance, when one samples such a quantum half-plane, and one chooses any fixed positive real $b$, one can define the point $c$ on the bottom part of the boundary such that the boundary length of the half-line to the left of $c$ is exactly equal to $b$. Then, one can consider a conformal map $\phi$ from the strip onto itself that maps $c$ to $-\infty$ and keeps $+ \infty$ unchanged (this map $\phi$ is then defined up to a horizontal translation). The special feature of the quantum half-plane is that the law of the obtained surface is again that of a quantum half-plane \cite{SHE_WELD,dms2014mating}. 
\item For each $\gamma \in (0,2)$, among the measures on surfaces with finite area (beads of thin wedges, corresponding to excursions of Bessel processes), there is also one for which neither $-\infty$ nor $+\infty$ are particularly special boundary points -- this is the case where $W= \gamma^2 -2$, a bead of which we refer to as a quantum disk. Let us first define a simple operation on quantum surfaces as follows: Choose two boundary points $c$ and $c'$ independently according to the boundary length measure (renormalized to be a probability measure), and then consider a map $\phi$ from the strip onto itself, that maps $c$ and $c'$ onto $-\infty$ and $+ \infty$ respectively (again this map is defined up to a horizontal translation). Then, when one applies this procedure, the measure on marked quantum surfaces induced by the Bessel excursions is invariant \cite[Proposition~A.8]{dms2014mating}.  By scale invariance, it is possible to decompose this measure according to the total boundary length of the obtained surface, and to define a probability measure on quantum disks with a prescribed boundary length \cite[Section~4.5]{dms2014mating}. (We remark that at this stage, the wedge has a negative weight when $\gamma < \sqrt {2}$ -- but we will anyway use only wedges with positive weight in the present paper).  
\end{itemize}

Note that the weight $\gamma^2 - 2$ of the quantum disks and the weight $2$ half-planes are related by the above-mentioned ``duality relation'' $W+W' =\gamma^2$.

If we parameterize a quantum wedge (or a bead of a thin quantum wedge) by $\h$ and want to emphasize that the marked points are at $0$ and $\infty$, we will use the notation $(\h,h,0,\infty)$.  Similarly, if we parameterize it by $\strip$ and want to emphasize that the marked points are at $-\infty$ and $+\infty$, we will use the notation $(\strip,h,-\infty,+\infty)$.

\begin {remark}
Due to the LQG conformal covariance, there are several 
natural ways to actually define the quantum disk (i.e., to choose the domain $D$ in which one defines it as well as the actual normalization one uses among all the conformal automorphisms of $D$).  If one chooses the reference domain to be the infinite strip $\R \times (0,\pi)$ and chooses the embedding so that two boundary typical points are taken to $\pm \infty$, then one obtains the definition of the quantum disk developed in \cite{SHE_WELD,dms2014mating}.  If one alternatively chooses to fix the embedding using three points, then the definition one obtains is as described in \cite{hrv2018disk}. 
 A proof of the equivalence between these definitions can be found in \cite{c2019disk}. This mirrors the similar story for the definitions of LQG spheres, where the approaches developed in \cite{SHE_WELD,dms2014mating} and in \cite{lqg_sphere} were proven to be equivalent in \cite{ay2017twoperspectives}.  In the present article, we will use the version of the disk with two marked points as it is amenable to $\SLE$ techniques (the two points corresponding to the seed and target of the chordal $\SLE$-type curve).
 \end {remark}

 \begin {remark} 
 \label {beadboundaries}
At some point in this paper, it will be useful to work with the laws of a bead of a thin quantum wedge {\em conditioned} on its left and right boundary lengths (these are the quantum lengths of the parts of its boundary between the two marked point). The fact that this conditioning makes sense can be easily worked out from the orthogonal decomposition of the field in the bead. In other words, one can decompose the infinite measure on beads into a measure on its left and right boundaries, and then use the probability measure describing the law of the bead when conditioned on the lengths of its left and right boundaries. 
 \end {remark}

\subsection{$\SLE$ explorations of quantum surfaces}

We will now recall some of the basic welding operations which are proved in \cite{dms2014mating} and which we will make use of later in the proofs of our main results.  The first result is regarding the case of welding quantum wedges along their boundaries (or equivalently cutting with an independent $\SLE_\kappa$-type curve):

\begin{theo}[Theorems~1.2 and~1.4 from \cite{dms2014mating}]
\label{thm:wedge_welding}
\hskip 1cm
\null 
\begin{enumerate}[(i)]
\item Let $\CW = (\h,h,0,\infty)$ be a thick quantum wedge   of weight $W$.  Fix $\rho_1,\rho_2 > -2$ with $\rho_1 + \rho_2 + 4 = W$.  Let $\eta$ be an independent $\SLE_\kappa(\rho_1;\rho_2)$ process in $\h$ from $0$ to $\infty$.  If we let $\CW_1$ (resp.\ $\CW_2$) be the quantum surface which is parameterized by the part of $\h \setminus \eta$ which is to the left (resp.\ right) of $\eta$, then $\CW_1,\CW_2$ are independent quantum wedges with weights $W_1= \rho_1 + 2$ and $W_2 = \rho_2 +2 $. Furthermore, $\eta$ and $\CW$ are both almost surely determined by $\CW_1$ and $\CW_2$.
\item The same result holds true if $\CW$ is a thin quantum wedge of weight $W$ except that $\eta$ is defined to be a concatenation of independent $\SLE_\kappa(\rho_1;\rho_2)$ processes, one for each bead of $\CW$.
\end{enumerate}
\end{theo}

We will in fact mostly be using statement (i) in the case where $\rho_2 = 0$, in which case $\CW_2$ is a quantum half-plane. Note that for statement (ii) to hold for a thin 
wedge, the formula $\rho_1 + \rho_2 + 4 = W$ will require $\rho_1$ {\em and } $\rho_2$ to be negative. 

\medbreak

The second result that we will restate is the analogous result for $\SLE_{\kappa'}(\rho_1;\rho_2)$ processes for $\kappa' \in (4,8)$: 

\begin{theo}[Theorem~1.17 from \cite{dms2014mating}]
\label{thm:sle_kp_rp}
Assume that $\gamma \in (\sqrt{2},2)$.  Fix $\rho_1,\rho_2 \geq \kappa'/2-4$, let $W_i = \gamma^2 -2 + ({\gamma^2 \rho_i / 4})$ for $i=1,2$, and $W = W_1 + W_2 + 2-\gamma^2/2$.  Suppose that $\CW = (\h,h,0,\infty)$ is a quantum wedge of weight~$W$.  Let~$\eta'$ be an independent $\SLE_{\kappa'}(\rho_1;\rho_2)$ process in~$\h$ from~$0$ to~$\infty$.  Then the quantum surface parameterized by the components which are to the left (resp.\ right) of $\eta'$ is a quantum wedge of weight $W_1$ (resp.\ $W_2$) and the quantum surfaces parameterized by the components which are completely surrounded by $\eta'$ are all quantum disks given their boundary lengths. 	
\end{theo}

This makes it in particular quite natural to draw an $\SLE_{\kappa'}$ on top of a quantum wedge of weight $3 \gamma^2 /2 - 2$. 
In that case, let us recall the boundary length evolution description: 

Assume that $\gamma \in (\sqrt{2},2)$.  Suppose that $\CW = (\h,h,0,\infty)$ is a quantum wedge of weight $W = 3\gamma^2/2 - 2$.  Let $\eta'$ be an independent $\SLE_{\kappa'}$ in $\h$ from $0$ to $\infty$. We parameterize $\eta'$ by its quantum length (induced by $\CW$).  For each $t \geq 0$, we let $L_t$ (resp.\ $R_t$) denote the change in the left (resp.\ right) side of the outer boundary of $\h \setminus \eta'([0,t])$ relative to time $0$ (i.e., $L_0 = R_0 = 0$).  Note that each downward jump of $L$ (resp.\ $R$) corresponds to a component separated from $\infty$ by $\eta'$ at the corresponding time and the length of the jump gives the quantum boundary length of the disconnected region.  Recall from~\eqref{eqn:basic_relations} that $\alpha = 4/\kappa = \kappa'/4$.

\begin{theo}[Theorem~1.18 and Corollary~1.19 from \cite{dms2014mating}]
\label{thm:sle_kp_explore} 
\hskip 1cm
\begin{enumerate}[(i)]
\item The processes $L$ and $R$ are independent $\alpha$-stable L\'evy processes. This describes in particular the law of the jumps   Moreover, the quantum surfaces parameterized by the components are conditionally independent quantum disks given their boundary lengths. 
\item Furthermore, for each $t \geq 0$, we can consider the unbounded connected component $H_t$ of the complement of $\eta' [0,t]$, and view this as a quantum surface. We can also consider the complement $K_t$ of $H_t$ in $\h$, and also view it as a quantum surface. Then $H_t$ is a quantum wedge of weight $3 \gamma^2 /2 - 2$, that is independent of the quantum surface $K_t$ decorated by $\eta'$ up to time~$t$. 
\item Finally, $\CW$ and $\eta'$ are a.s.\ determined by the boundary length processes $L,R$ and corresponding collection of quantum disks, each marked by the first point on their boundary visited by $\eta'$.
\end{enumerate}
\end{theo}

\begin {remark}[About the  quantum length of the non-simple $\SLE_{\kappa'}$-type curves] 
\label {GQLRemark}
In Theorem \ref {thm:sle_kp_explore}, we used the quantum length of the $\SLE_{\kappa'}$ curve $\eta'$. It is worth making a few comments about the definition and properties of quantum lengths of non-simple SLE curves since they feature also in the statements of our main theorems. Recall first that the definition of the the quantum length of a simple $\SLE_\kappa$ curve (and its variants) follows from the definition of the boundary length for GFFs with Neumann-type boundary conditions via the quantum zipper properties as initiated in \cite {SHE_WELD}. 
One way to view/define the quantum length of an $\SLE_{\kappa'}$-type curve is in fact precisely that it is the parameterization for which the processes $R$ and $L$ in Theorem \ref {thm:sle_kp_explore} are stationary with independent increments. We can note that the jumps of $R$ and $L$ correspond to the boundary lengths of the cut-off domains, which have (simple) $\SLE_{\kappa}$-type boundaries. In particular, it then follows immediately from the theorem that if one shifts the GFF in such a way that all (usual) quantum lengths of simple $\SLE_\kappa$ curves are multiplied by a factor $u$, then the quantum length of the (non-simple) $\SLE_{\kappa'}$ will actually be multiplied by a factor $u^\alpha$.

Another way to describe this is to say that the quantum length of a piece of $\SLE_{\kappa'}$ curve is (up to a multiplicative constant factor) the limit as $\eps \to 0$ of $\eps^\alpha$ times the number of cut out pieces (along that portion of curve) of boundary length in $[\eps, 2 \eps]$. This type of description is now ``local'' and can therefore be used for other $\SLE_{\kappa'}$-type curves in other $\gamma$-quantum surfaces than the wedge $\CW$. 

A final remark is that there is no clear canonical normalization for quantum lengths of $\SLE_{\kappa'}$-type curves, but for the purposes of the present paper, we do in fact not need to care about this, since all our results will anyway be unchanged when one changes all these quantum lengths by the same mulitplicative constant. 
\end {remark}

\subsection{The loop-trunk decomposition of $\SLE_\kappa^\beta(\kappa-6)$}
\label{subsec:trunk_construction}

Let us make a few additional details about the  loop-trunk decomposition of the $\SLE_\kappa^\beta(\kappa-6)$ processes, which Theorem~\ref{thm:cle_perc_trunk} is part of.

The results of \cite[Section~7]{cle_percolations} do in fact provide also 
 the description of the conditional law of the collection of CLE loops encountered by the process, given its trunk, in terms of \emph{boundary conformal loop ensembles} ($\BCLE$).   

 Let us briefly detail some aspects this conditional law that will be used in the present work: 
Suppose that $\kappa \in (8/3, 4)$ and that $\beta=1$. Then, it is shown in \cite[Section~8]{cle_percolations} that 
the $\SLE_\kappa (\kappa -6)$ process (from $0$ to $\infty$ in $\h$) is a continuous curve $\eta$. Each of the excursions away from $0$ made by the Bessel process 
used to define this $\SLE_\kappa (\kappa -6)$ process corresponds to a closed simple loop made by $\eta$ -- which can be viewed as the loops in a CLE$_\kappa$. If one excises all these loops from $\eta$, one obtains a (non-simple) curve $\eta'$ from $0$ to $\infty$, which is the trunk of $\eta$. The law of this 
trunk is shown in \cite[Section~7--9]{cle_percolations} to be an $\SLE_{\kappa'} ( \kappa' -6)$. Let us stress here the following point that we have not mentioned so far: As opposed to the other SLE$_\kappa(\rho)$ processes that we discuss in the present paper, this particular process $\eta'$ has its marked point ``to the right of its tip'' and not to its left. In other words, it is an SLE$_{\kappa'} (0, \kappa'-6)$ process. 

Suppose now that $\tau'$ is a finite stopping time for the trunk (this stopping time can use some additional randomness 
that does not come from $\eta$, we will typically later take $\tau'$ to be the first 
time at which the quantum length of $\eta'$ reaches a certain value). This time $\tau'$ then almost surely corresponds to a single time $\tau$ for $\eta$. 
We now have a first concrete description of $\eta [0,\tau]$: Run the $\SLE_\kappa (\kappa -6)$ until time at which its trunk time is $\tau'$. 

One main result of \cite{cle_percolations} is the following alternative description: Run first the trunk alone until time $\tau'$, and then use the description of 
the conditional law of $\eta$ given $\eta'[0,\tau']$ that is given in \cite{cle_percolations} in terms of boundary conformal loop ensembles. One consequence of 
this description is the following: 

\begin{prop} [\cite{cle_percolations}]
\label{looptrunkprop}
The conditional law of the unbounded connected component of $\HH \setminus \eta[0,\tau]$ given $\eta'[0, \tau']$ can be obtained by running an $\SLE_\kappa (3 \kappa /2 - 6)$ process $\eta''$ from the right-most intersection 
point $x$ of $\eta_{[0 , \tau']}'$ with the real line, to $\eta'_{\tau'}$  (with marked point at its starting point). This component is then distributed as the unbounded connected component of 
$\HH \setminus (\eta'[0, \tau'] \cup \eta''$).
\end{prop}

Recall that $\SLE_\kappa (\rho)$ processes for $\rho > -2$ are reversible (see \cite{MS_IMAG2}), and also that an $\SLE_\kappa (\rho)$ process from $a$ to $b$ with marked point at $c$ can be viewed as an $\SLE_\kappa (\kappa- 6- \rho)$ process from $a$ to $c$ with marked point at $b$ -- which leads to a number of equivalent descriptions of this conditional law of $\eta''$ (for instance as an $\SLE_{\kappa} ( - \kappa /2)$ process). 

In Proposition~\ref{looptrunkprop}, each excursion of $\eta''$ away from the trunk $\eta'$ corresponds to a portion of the same simple $\CLE_\kappa$ loop traced by the $\SLE_\kappa (\kappa -6)$ process $\eta$. To finish this loop, it is shown in \cite{cle_percolations} that one needs to draw an $\SLE_\kappa (- \kappa /2)$ $\eta'''$ in the ``pocket'' in between this excursion and $\eta'$. To then find the loops of $\eta$ squeezed in between $\eta'$ and $\eta'''$, one can iterate the procedure, by drawing alternatively $\SLE_\kappa(3 \kappa /2 -6)$ and $\SLE_\kappa(- \kappa/2)$ processes.

One first remark is that it is shown in \cite {cle_percolations} that the mapping $\beta \mapsto \rho'(\beta)$ in 
Theorem \ref{thm:cle_perc_trunk} is one-to-one from $[-1, 1]$ onto $[\kappa'-6, 0]$. Furthermore, it actually turns out that $\rho'$ is in fact a increasing function of $\beta$ -- this can at first appear counterintuitive since it means that ``the more loops are on the right of the trunk, the more to the right trunk tends to be'' -- but one has to keep in mind that there is a ``L\'evy compensation'' mechanism embedded in the construction of these 
$\SLE_\kappa^\beta(\kappa-6)$ processes.

\subsection {An instance of imaginary geometry coupling}
\label{Imgeom}

In the derivation of the previously stated results, the ``imaginary geometry'' coupling of several SLE-type curves with an auxiliary GFF was instrumental. 
A particular instance of these flow-line/counterflow line interaction will be the following (which follows from the statements in \cite{MS_IMAG}): 

Consider the upper half-plane $\h$, and a GFF with boundary conditions $\lambda' = \pi/\sqrt{\kappa'}$ on \jasonr{$\R_-$} and $-\lambda'$ on \jasonr{$\R_+$}. One can then define the ``counterflow-line'' of this GFF starting at $0$ aiming $\infty$ -- this is an $\SLE_{\kappa'}$ curve $\eta'$ from $0$ to $\infty$.  It is also possible, for the same GFF, to consider a ``flow line'' from $\infty$ to $0$, of appropriate angle, so that this curve $F$ is an $\SLE_\kappa (3 \kappa /2 -6)$ from $\infty$ to $0$, see Figure~\ref{picImGeom}. This is a simple curve that intersects the positive half-axis many times, but not the negative half-axis. We can define the infinite connected component $H$ of the complement of $F$. Then, for this coupling, a feature that will be very useful in our proofs is that {\em the conditional law of $\eta'$ given $F$ is that of an $\SLE_{\kappa'} (\kappa'-6)$ in $H$}.

\begin{figure}[ht!]
\includegraphics[scale=.55]{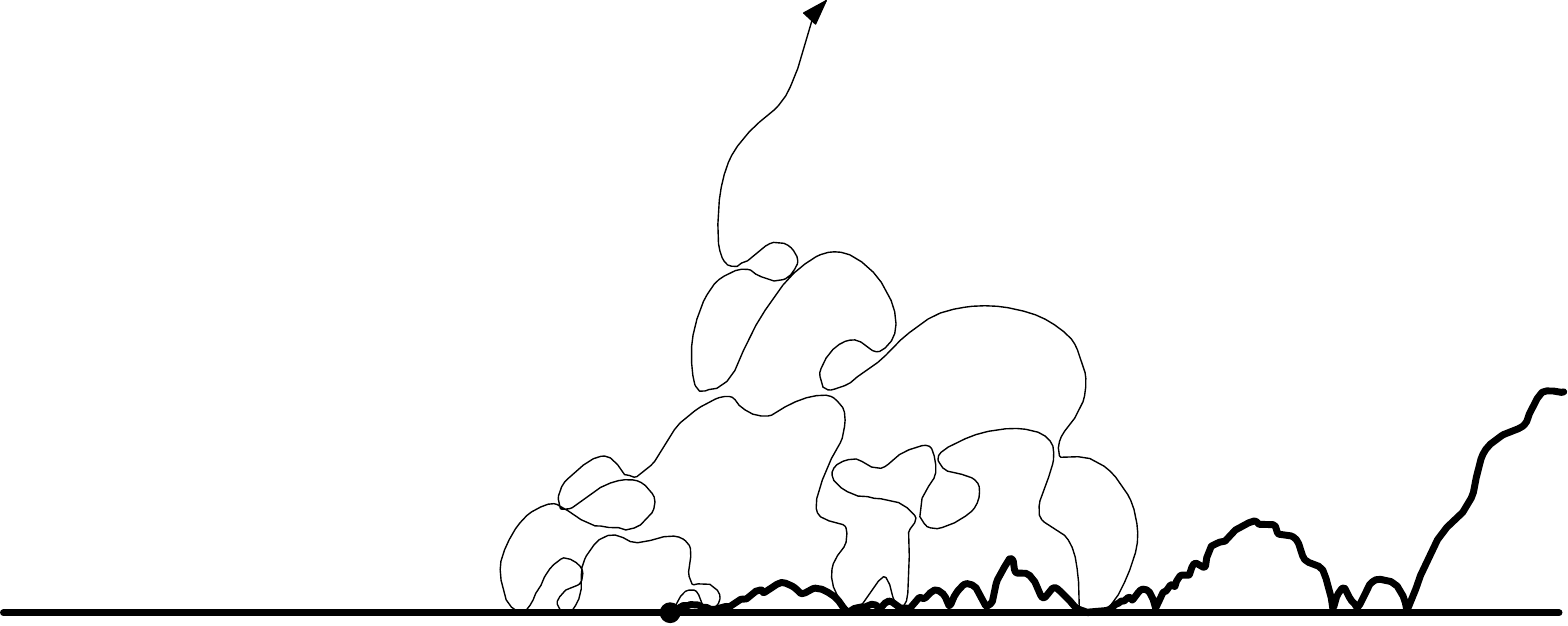} 
	\caption{The coupling of the $\SLE_{\kappa'} (\kappa' -6)$ with the $\SLE_\kappa (3 \kappa /2 -6)$}
	\label{picImGeom}
\end{figure}

A variant of the previous statement occurs if one considers $\eta'$ up to some stopping time $\tau'$ (that possibly involves additional randomness). Indeed (and this is of course related to the conformal Markov property of $\eta'$), up to the imaginary geometry change of coordinates formula described in \cite{MS_IMAG}, the GFF in the unbounded connected component $H'$ of the complement of $\eta'[0, \tau']$ has boundary conditions $\lambda'$ and $-\lambda'$ on the two sides of $\eta_{\tau'}'$. One can therefore consider in $H'$ the flow line $\wt F$ from $\infty$ to $\eta_{\tau'}'$ of the same angle as $F$, see Figure~\ref{picImGeom2}. This flow-line will  then coincide with $F$ until the first point at which it touches $\eta'[0, \tau']$, and the remaining part of $\wt F$ (that we call $F'$ and will play a key role in the proof of Proposition~\ref{prop:stationarity})  will be an $\SLE_\kappa (3 \kappa /2 - 6)$ from $x$ to $\eta_{\tau'}'$ in the remaining domain.

\begin{figure}[ht!]
\includegraphics[scale=.55]{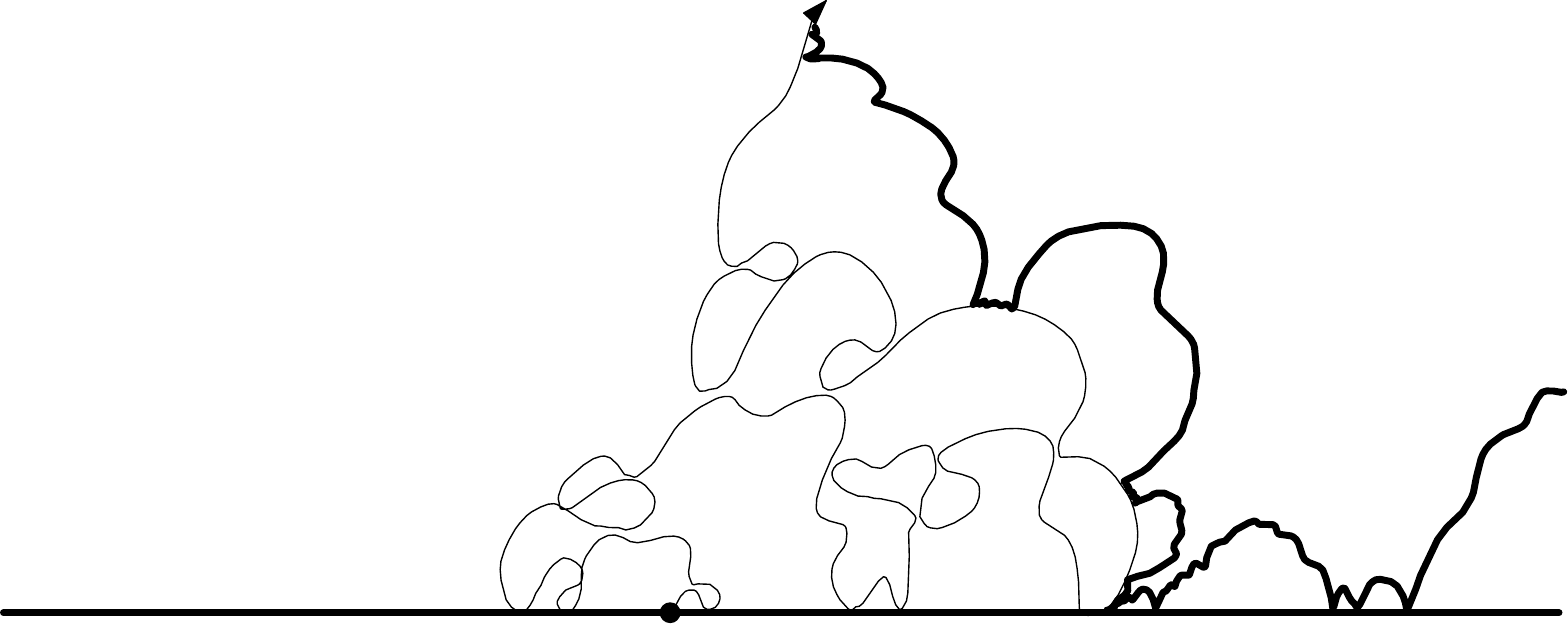} 
\includegraphics[scale=.55]{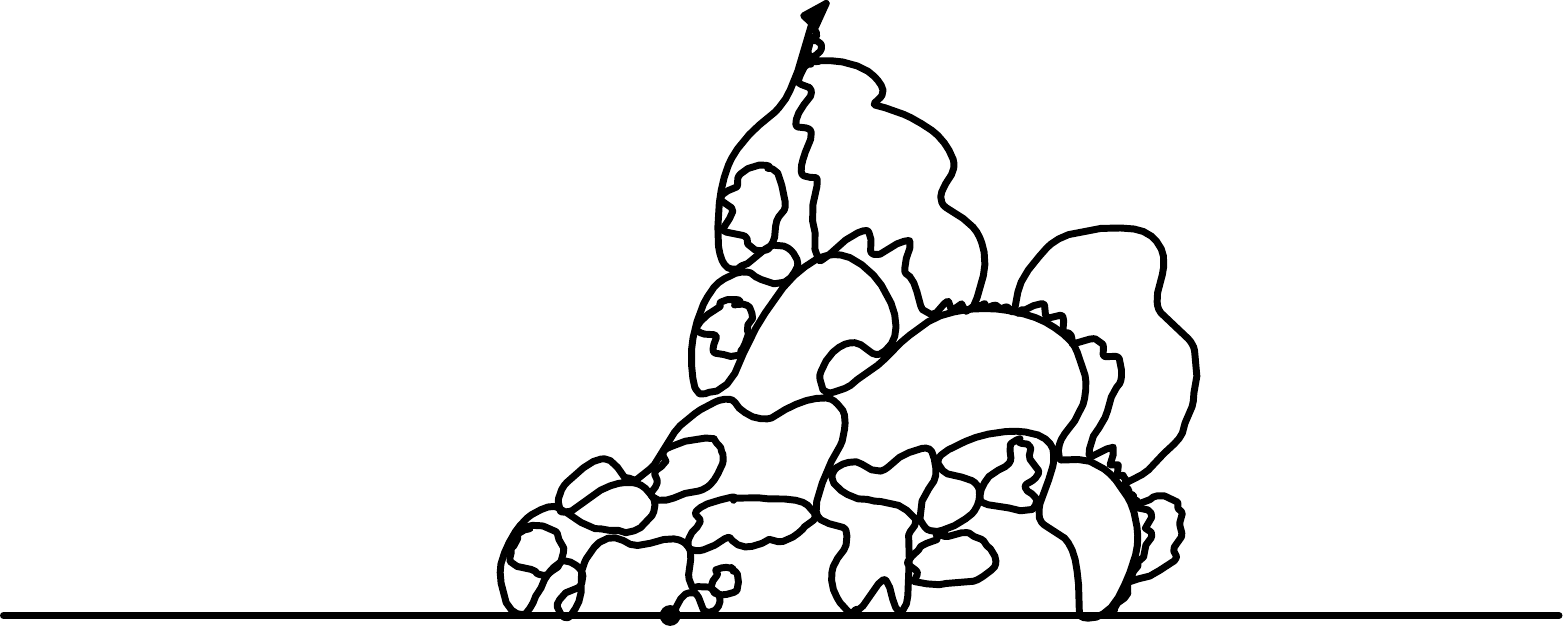} 
	\caption{The variant at finite time of the coupling. We will use in the sequel that the latter part of the flow-line can be viewed as the outer boundary of an $\SLE_\kappa (\kappa -6)$ process as on the bottom figure.)}
	\label{picImGeom2}
\end{figure}

In Section \ref {41}, we will also use a further variant of this type of coupling, where we this time couple $\eta'$ with two flow-lines $F_L$ and $F_R$ that lie of the two sides of $\eta'$.

\section{Totally asymmetric exploration of CLEs on quantum half-planes} 
\label {S3}

Throughout this section, we will consider $\gamma \in (\sqrt {8/3}, 2)$ and $\kappa = \gamma^2 \in (8/3, 4)$.  We will also assume that $\kappa'=16/\kappa$ and $\alpha = 4/\kappa = \kappa'/4$ as in~\eqref{eqn:basic_relations}.  We will begin in Section~\ref{subsec:unexplored_region} by showing that the law of the quantum half-plane is invariant under the operation of cutting along an independent $\SLE_\kappa (\kappa-6)$ process (so we are dealing here with the case $\beta=1$ where the trunk stays ``to the left'' of all CLE loops that it encounters) up to a given amount of quantum length for its trunk, and that the collection of quantum surfaces that are cut off to the left of the trunk form a Poisson point process of quantum disks.  We emphasize that in Section~\ref{subsec:unexplored_region}, we do not yet prove that the surfaces that are surrounded by CLE loops or cut out to the right of the trunk are also quantum disks (we will only describe their boundary lengths) -- this is then the purpose of Section~\ref{subsec:all_disks}.

\subsection{Stationarity in the totally asymmetric case}
\label{subsec:unexplored_region}
\label{S31} 
 
Consider a quantum half-plane $\CW = (\h,h,0,\infty)$.  In the sequel, when $O$ is an open subset of $\h$ (with some marked points), we will always use on $O$ the area measure given by $\CW$ when we refer to $O$ as a quantum surface. For presentation purposes, we choose to first study the case where one explores this quantum half-plane with an independent totally asymmetric $\SLE_\kappa (\kappa -6)$, as we believe that it will help the reader to see the strategy in the case of non-totally-asymmetric explorations in Section \ref {S4}.

Consider a totally asymmetric $\SLE_\kappa (\kappa -6)$ process $\eta$ in $\h$ from $0$ to $\infty$, with $\kappa = \gamma^2 \in (8/3, 4)$. Along the way, this process cuts a countable family of quantum surfaces away from $\infty$. This happens (i) when $\eta$ closes a CLE loop, (ii) when the trunk hits $\R$ or a point that had already been visited by $\eta$ (note that only countably many such double-points correspond to times at which a surface is actually cut out from infinity -- for example, in the set of times at which the trunk hits $\R$, only the times that are isolated from the left will satisfy this property). 
Those special double points of $\eta$ will correspond to one of the countably many times at which $\eta$ splits the remaining to be discovered domain into two parts.  It can disconnect a domain to its right or to its left.  We denote by $\CF_t$ the filtration generated by these three collections of quantum surfaces (corresponding to loops, cut out by the trunk to its left and cut out by the trunk to its right) up to time $t$ \jason{(where here we use the capacity time parameterization for $\eta$)}. We view these quantum surfaces as {\em marked} quantum surfaces with one or two marked boundary points: When it corresponds to a $\CLE$ loop, there is only one marked point, which is the position of the tip of the trunk at this disconnection time. When it corresponds to a domain cut out by the trunk which does not intersect $\R$, one also has one marked point corresponding to the disconnection time.  Finally, when it corresponds to a domain cut out by the trunk which intersects $\R$, one has two marked points corresponding to the first and last boundary point visited by $\eta$.  

Since the trunk of $\eta$ is an $\SLE_{\kappa'}(\kappa' - 6)$ process independent of $\CW$, it is possible to define its LQG-quantum length \jason{(see \cite[Theorem~1.18]{dms2014mating} and the surrounding text)} . We denote by $\tau'=\tau_t'$ the first time at which this quantum length reaches $t$, and we denote by $\tau=\tau_t$ the corresponding time for $\eta$.  The $\sigma$-field $\CF_{\tau}$ therefore contains the information about all the quantum surfaces that have been cut out by $\eta[0, \tau]$ from $\infty$. We define $\CG_t := \CF_{\tau_t}$.

The first main key result is the following: 
\begin{proposition}
\label{prop:stationarity}
For each $t \geq 0$, the quantum surface parameterized by the unbounded component of 
$\h \setminus \eta([0,\tau_t])$ with marked points $\eta(\tau)$ and $\infty$ is a quantum half-plane, which is independent of $\CG_t$.
\end {proposition}

\begin{proof}
It is convenient to start with a quantum wedge $\wt{\CW} = (\h,\wt{h},0,\infty)$ of weight $3\gamma^2/2-2$, which is ``wider'' than the quantum half-plane, and to view the quantum half-plane as a subset of this quantum wedge. More precisely, Theorem~\ref{thm:wedge_welding} implies that if one draws an independent $\SLE_\kappa (3\kappa/2 -6)$ from $\infty$ to $0$ (with marked point at $\infty$ to start with) that we call $F$, then the unbounded connected component of the complement of $F$ will be a quantum half-plane (that is independent of the quantum surfaces with bounded area that are cut out from $\infty$ by the path), see Figure~\ref{F5}.

\begin{figure}[ht!]
\includegraphics[scale=.5]{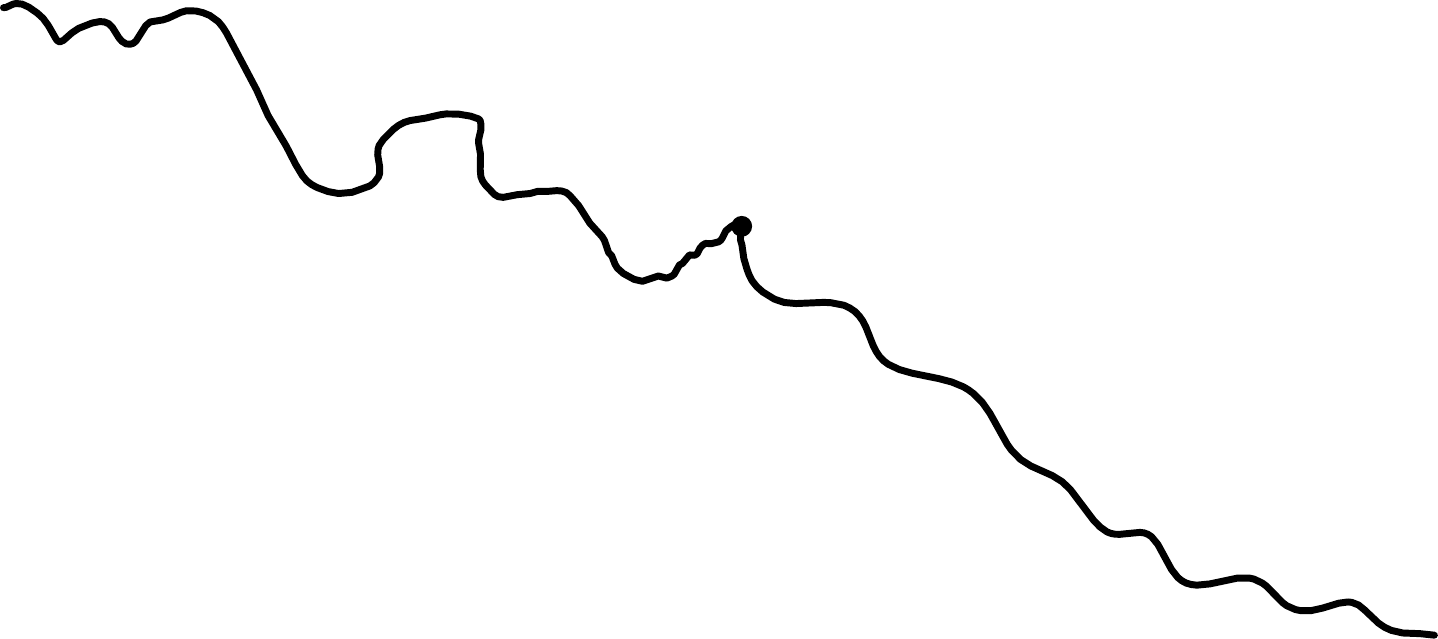} 
\quad
\includegraphics[scale=.5]{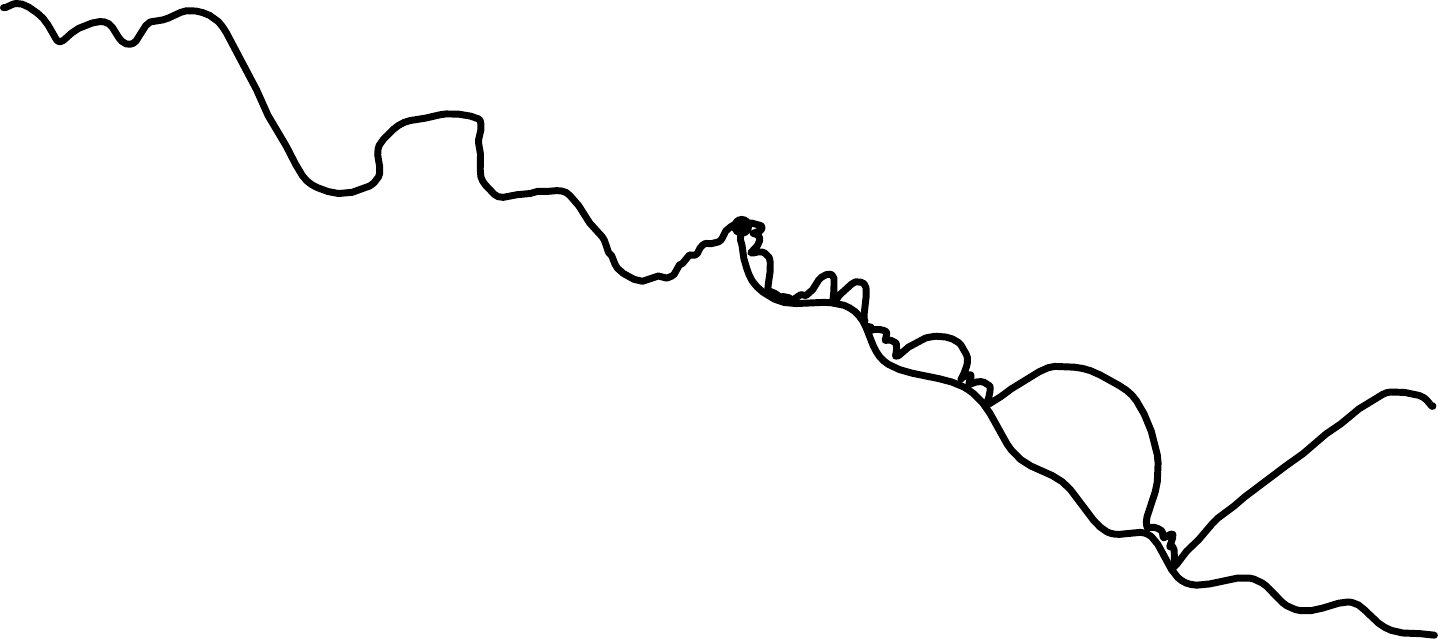} 
	\caption{Symbolic sketch: The wedge (above the curve), the $\SLE_\kappa (3\kappa/2 - 6)$ and the half-plane (above the curves)}
	\label{F5}
\end{figure}

We now consider $\eta'$ to be an independent $\SLE_{\kappa'}$ from $0$ to $\infty$ in $\h$. We couple $\eta'$ with an $\SLE_{\kappa} (3 \kappa /2 - 6)$ process $F$ from $\infty$ to $0$ as described in Section~\ref{Imgeom}. We let $H$ be the unbounded connected component of the complement of $F$. 
Recall also the definitions of $x$, $\wt F$ and $F'$ from Section~\ref{Imgeom}.

Then, using Theorem~\ref{thm:sle_kp_rp} on the one hand, and the properties of the coupling of $F$ with $\eta'$, we see that the triplet $(H, 0, \infty)$ equipped with the area measure from $\wt{\CW}$ is a quantum half-plane, and that conditionally on $F$, the process $\eta'$ is an $\SLE_{\kappa'} ( \kappa' -6)$ from $0$ to $\infty$ in $H'$. So, we can use $\eta'$ in $H$ as a model for our $\SLE_{\kappa'} (\kappa'-6)$ on an independent quantum half-plane, see Figure~\ref{F6}.

\begin{figure}[ht!]
\includegraphics[scale=.5]{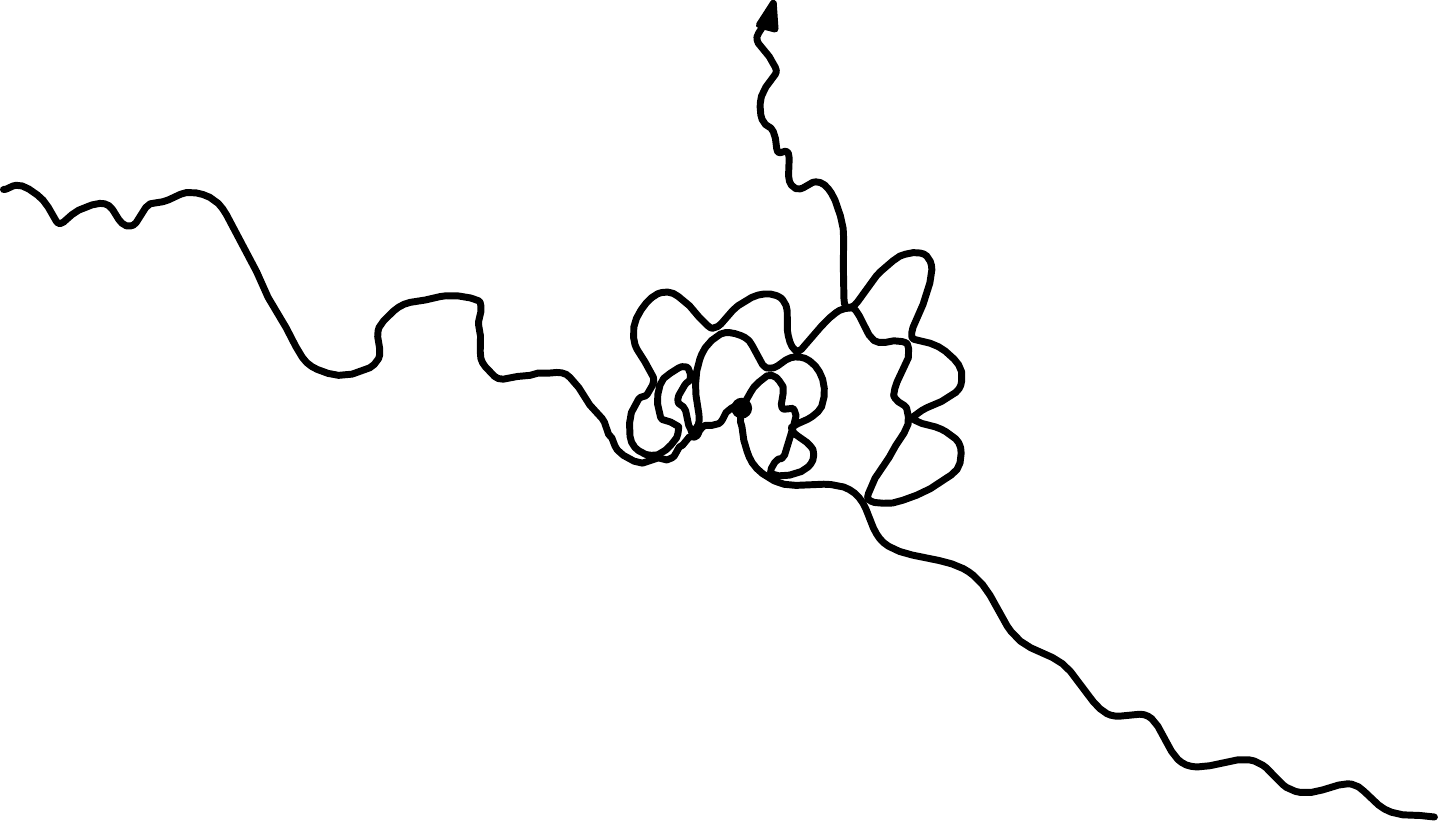} 
\quad
\includegraphics[scale=.5]{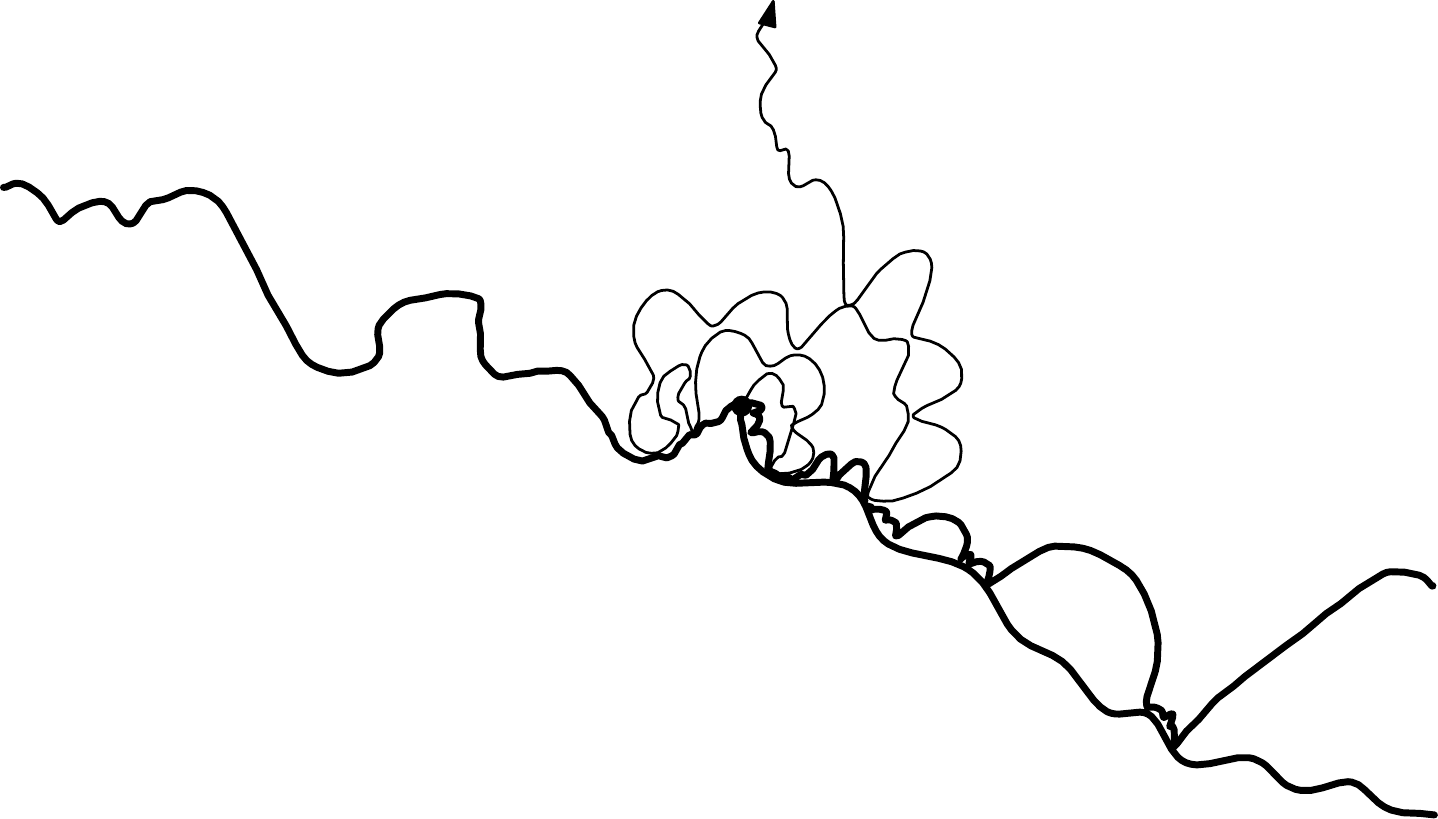} 
	\caption{Symbolic sketch: The $\SLE_{\kappa'}$, viewed in the wedge -- the complement is still a wedge. The curve viewed in the half-plane is an $\SLE_{\kappa'}(\kappa'-6)$}
	\label{F6}
\end{figure}

The crucial point now is that the conditional law of the outer boundary of $\eta[0, \tau]$ is given by an additional $\SLE_{\kappa}(3 \kappa/2 - 6)$ from $x$ to $\eta_{\tau'}'$ in $H' \setminus \eta'[0, \tau']$ -- which happens to be the same as the conditional law of $F'$. We can therefore choose our coupling in such a way that this outer boundary is exactly $F'$.

\begin{figure}[ht!]
\includegraphics[scale=.55]{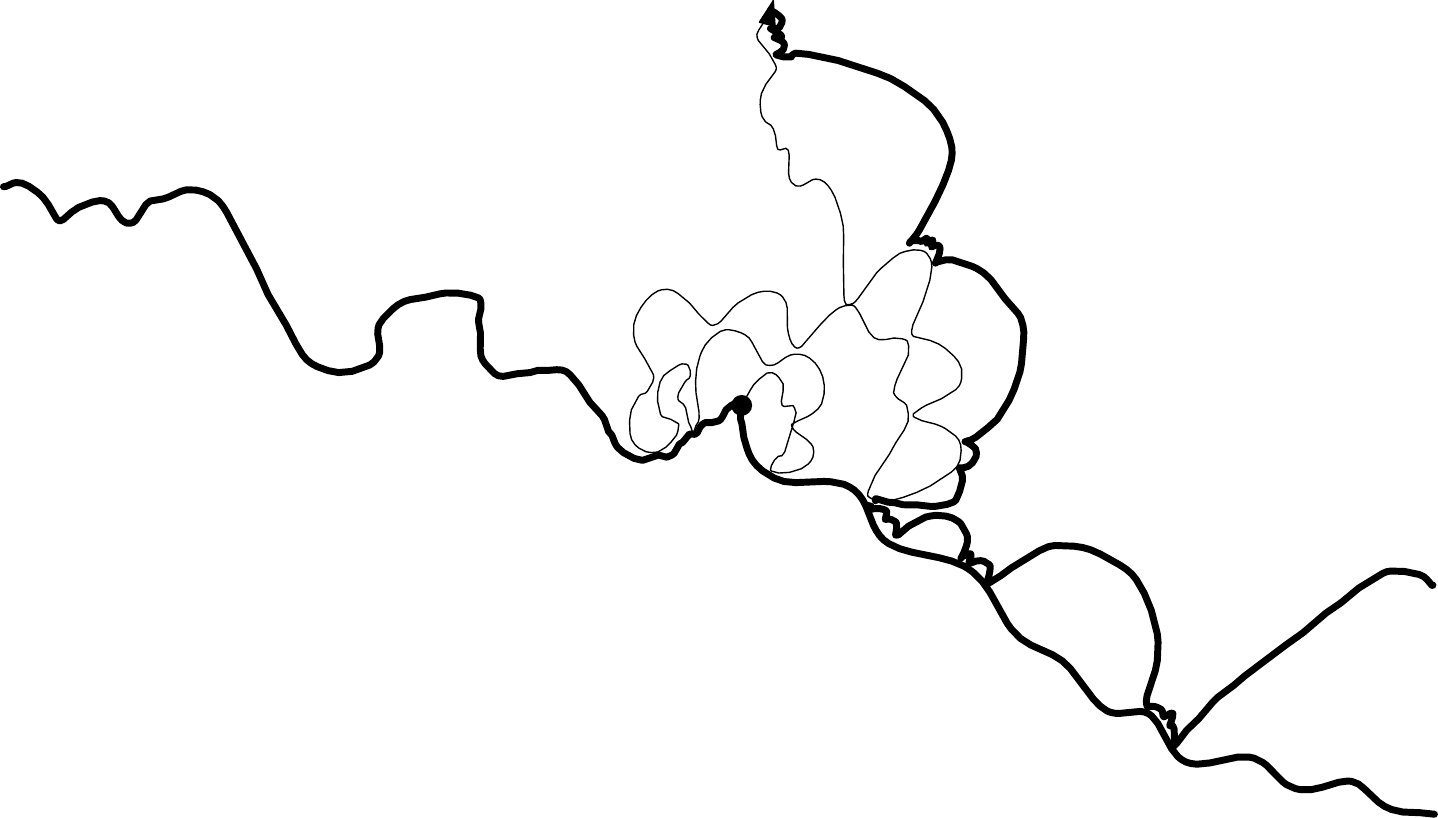}  
	\caption{Symbolic sketch: Creating a half-plane by running the $\SLE_\kappa (3\kappa /2 -6)$ in the wedge of the left of Figure~\ref{F6}.  Its latter part can be viewed as the right boundary of the $\SLE_\kappa (\kappa -6)$ that $\eta'$ is the trunk of.}
	\label{F7}
\end{figure}

Hence, the quantum surface parameterized by the unbounded component of 
$\h \setminus \eta([0,\tau])$ with marked points $\eta(\tau)$ and $\infty$ can be realized as follows: In the unbounded connected component $H''$ of $\h \setminus \eta'[0, \tau']$, draw an independent $\SLE_{\kappa} (3\kappa /2 - 6)$ from $\infty$ to $\eta_{\tau'}'$ 
and consider the infinite connected component of its complement, with marked points $\eta_{\tau'}'$ and $\infty$. 

But we know that $H''$  with marked points $\eta'(\tau')$ and $\infty$ is a quantum wedge of weight $3 \gamma^2 /2 -2$, that is independent of all the quantum surfaces cut out of $\eta'$ up to this time. So, when one draws this $\SLE_{\kappa} (3 \kappa /2 - 6)$ as in Figure~\ref{F7}, one obtains a quantum half-plane, which shows the proposition.
\end {proof} 

We can note the following by-product of the proof: 
Let $(S^l_{t_i})_{t_i > 0 }$ denote the point process of quantum surfaces that are cut out to the left of the trunk of $\eta$. (In other words, $S_{t_i}^l$ is not empty if exactly at $\tau_{t_i}$, the trunk cuts out a bounded connected component $O_{t_i}$ from $\infty$ to its left.)  Then, the process $(S_{t_i}^l)_{t_i > 0}$ is a Poisson point process of marked quantum disks 
because by construction, the process $(S_{t_i}^l)_{t_i > 0}$ is the same as that cut out on its left by the $\SLE_\kappa$ process in the quantum wedge $\wt{\CW}$, and we know by Theorem~\ref{thm:sle_kp_explore} that this is a Poisson point process of marked quantum disks.

Proposition~\ref{prop:stationarity} also implies the following fact: 
\begin {corollary} 
 Let $L_t$ (resp.\ $R_t$) denote the change in the boundary length of the left (resp.\ right) side of the unbounded component of $\h \setminus \eta([0,\tau_t])$ relative to time $0$.  Then $L$ and $R$ are independent $\alpha$-stable L\'evy processes.
\end {corollary}
\begin {proof} 
Proposition~\ref{prop:stationarity} implies that  the processes $L$ and $R$ have stationary independent increments and are therefore L\'evy processes. That $L$ and $R$ are independent follows because a.s.\ the two processes do not have a simultaneous jump \jason{(see, e.g., \cite[Chapter~0, Section~4]{bertoin96levy})}.  

Recall from Remark \ref {GQLRemark} that when one scales the quantum half-plane in such a way that its boundary length is multiplied by a factor $u$, then the  natural quantum length of the trunk is scaled by a factor $u^\alpha$. It therefore follows that in the present setting, $(uL_{t})_{t \ge 0}$ has the same law as $(L_{u^\alpha t})_{t \ge 0}$. Since we have just seen that $L$ is a L\'evy process, we can conclude that the process $L$ is $\alpha$-stable. The same argument shows that $R$ is $\alpha$-stable as well.   
\end{proof}

Let us now define by $S^r$ and $S^+$ to be the two point processes of quantum surfaces cut out to the right of the trunk and encircled by a CLE loop hit by the trunk. The same arguments as in the corollary for $L$ and $R$ can be used for these point processes: Proposition~\ref{prop:stationarity} shows that the law of the point processes $(S^l, S^r, S^+)$ after time $\tau_t$ is independent of the point processes of surfaces that have appeared before that time and is stationary. 
Furthermore, we note that two different quantum surfaces can never appear at the same time on the trunk (because of the continuity of the trunk and of the Markovian property of $\SLE_\kappa (\kappa -6)$ when it is on the trunk).  It therefore follows that the three processes $S^r$, $S^l$ and $S^+$ are independent Poisson point processes of marked quantum surfaces (when time is parameterized via the natural quantum length of the trunk). 

The quantum boundary lengths of the surfaces appearing in $S^l$, $S^r$ and $S^+$ are respectively the negative jumps of $L$, the negative jumps of $R$ and the positive jumps of $R$ (by construction, the process $L$ has no positive jumps). 
Their intensities are therefore multiples of the measure $dl / l^{\alpha + 1}$. We will denote the respective multiplicative constants by $a_{l, -}$, $a_{r, -}$ and $a_+$. In the next sections, we will not determine these constants, but we will prove the following relations between them: 
$ a_{l, -} = a_{r, -}$ and $a_+ = - 2\cos (\kappa \pi / 4) a_{l,-}$. 

Note that we have argued that 
$S^l$ is a Poisson point process of quantum disks, but we have not yet shown that it is also the case for $S^r$ and $S^+$. 

\subsection{All cut-out surfaces are quantum disks}
\label{subsec:all_disks}
\label {S32} 

Our goal is now to explain why $S^+$ and $S^r$ are point processes of quantum disks. We use the same notation as in the previous section, i.e., we draw $\eta$ on an independent quantum half-plane $(\h,h,0,\infty)$  and denote its trunk by~$\eta'$.

We know by Theorem~\ref{thm:sle_kp_rp} that $\eta'$ slices the quantum half plane into three independent quantum wedges of respective weights $\gamma^2-2$, $2- \gamma^2/2$ and $2 - \gamma^2/2$, corresponding respectively to the surfaces that are cut out to the left of the trunk (with boundaries intersecting $\R_-$), the surfaces between the left and right boundaries of the trunk, and the surfaces that lie to the right of the trunk (with boundaries that intersect \jason{$\R_+$}). 

The beads of the first quantum wedge (to the left of the trunk, components touching $\R_-$) are quantum disks, which is of course consistent with the fact that the process $S^l$ is a point process of quantum disks (mind however that not all surfaces of $S^l$ will be part of this wedge -- they will form a ``forested wedge'' in the terminology of \cite{dms2014mating}). Let us now focus on the latter (``rightmost'') quantum wedge $\CV$ with weight $2- \gamma^2 /2$. (Note that $2 - \gamma^2/2$ is smaller than $\gamma^2 -2$ because $\kappa > 8/3$, and that $\CV$ is in fact the ``dual'' of the thick wedge of weight $3 \gamma^2 /2 -2$ that we used before).  
By the loop-trunk decomposition of $\eta$, we know  the ``right boundary'' $\eta''$ of $\eta$, will consist of the concatenation of $\SLE_\kappa (3 \kappa /2 -6)$ processes (one independent one in each of the beads of the wedge $\CV$, from one marked point of the wedge to the other one).  We will deduce that $S^r$ and $S^+$ are also Poisson point processes of quantum disks from the following facts involving beads of thin wedges (see Remark \ref {beadboundaries}):  

\begin{lemma} 
\label{qdright}
\label{QDRIGHT}
\hskip 1cm
\begin{enumerate}[(i)]
 \item 
Consider a bead of a quantum wedge of weight $W= 2 - \gamma^2 /2$ conditioned on its left and right boundary lengths. Draw an independent $\SLE_\kappa (3 \kappa / 2 -6)$ from one of its marked points to the other one. Then, conditionally on its boundary length, the surface cut out to the right of the path is a quantum disk, and conditionally on their (left and right) boundary lengths, the surfaces cut out to its left are beads of a quantum wedge of weight $W'= 3 \gamma^2 /2 - 4$. 
 \item 
Consider a bead of a quantum wedge of weight $W'= 3 \gamma^2 / 2 -4$ conditioned on its left and right boundary lengths. Draw an independent $\SLE_\kappa (- \kappa /2)$ from one of its marked points to the other one. Then, conditionally on its boundary length, the surface cut out to the right of the path is a quantum disk, and conditionally on their (left and right) boundary lengths, the surfaces cut out to its left are beads of a quantum wedge of weight $W= 2- \gamma^2 /2$.  
\end{enumerate}
\end{lemma}

The two parts of Lemma~\ref{qdright} can be viewed as the ``finite-volume'' counterpart of the following two special cases of Theorem~\ref{thm:wedge_welding}: 
\begin{itemize} 
 \item An $\SLE_\kappa (3 \kappa / 2 -6)$ cuts a wedge of weight $3 \gamma^2 /2 -2 $ (the dual of $W$) into a half-plane and a wedge of weight $W'= 3 \gamma^2 /2 - 4$. 
 \item An $\SLE_\kappa (- \kappa /2)$ cuts a wedge of weight $4- \gamma^2 /2$ (the dual of $W'$) into a half-plane and a wedge of weight $W = 2- \gamma^2 /2 $.  
\end{itemize}

In order to keep the pace of the paper going, we choose to give the proof of Lemma~\ref{qdright} in the appendix.

\begin{remark} If we would have started with an entire thin quantum wedge in Lemma~\ref{qdright} and traced the $\SLE_\kappa (\rho)$ in each of the beads, then it is not possible to apply the thin-wedge part of Theorem~\ref{thm:wedge_welding} because of the condition that $W_1$ and $W_2$ have to be positive (i.e., a thin wedge of  weight $\gamma^2 -2$ can not be contained in a thin wedge of weight $2 - \gamma^2 /2 $ because we are in the regime where $\kappa > 8/3$). 
\end{remark}

Let us now discuss some consequences of this lemma: 

\begin {itemize}
 \item 
The first statement of Lemma~\ref{qdright} shows that all the quantum surfaces cut out to the right of the trunk $\eta'$ by $\eta$ and whose boundaries intersect $\R$ are independent quantum disks when conditioned on their respective boundary lengths (i.e., resampling each one does not change the law of this process). 
We should keep in mind that (just as for $S^l$), most of the quantum surfaces in $S^r$ do not actually end up corresponding to connected components that intersect $\R$. However, we can use the stationarity statement of Proposition~\ref{prop:stationarity}, and consider the picture of the quantum half-plane lying in front of $\eta$ after a rational quantum time $q$, and note that for each quantum surface $S_{t_i}$ in $S^r$, there almost surely exists a rational time $q < t_i$ such the corresponding connected component has part of $\eta[0,\tau_q]$ on its boundary. We can then invoke our previous observation to see that the law of this quantum surface conditionally on its boundary length is also a quantum disk.  We can therefore conclude that just as $S^l$, the Poisson point process $S^r$ is a Poisson point process of quantum disks. 
\item
Each excursion of $\eta''$ away from $\eta'$ corresponds to exactly one $\CLE$ loop that is partially traced by $\eta''$. As explained in \cite{cle_percolations}, in order to complete this loop, one has to ``branch inside'' into the pocket disconnected by this excursion of $\eta''$ and $\eta'$, and trace an $\SLE_\kappa (- \kappa /2 )$ process there. This corresponds to tracing an $\SLE_\kappa ( - \kappa /2)$ in the bead of a quantum wedge of weight $3 \kappa / 2 -4$. Part (ii) of Lemma~\ref{qdright} then allows one to conclude that the inside of the traced CLE loop is a quantum disk. Hence, this ``first outside layer of loops'' (the ones whose outer boundaries are part of $\eta''$) is formed of quantum disks.  To conclude that this is the case of the inside of all CLE loops traced by $\eta$, one can just use the very same argument as above (combining the stationarity statement after any given rational time with this fact). 
\end {itemize}
This concludes the proof of the following statement:
\begin {proposition} 
\label {threeindep}
 For a totally asymmetric $\SLE_\kappa (\kappa -6)$ drawn on a independent quantum half-plane, the three Poisson point processes of quantum surfaces $S^l$, $S^r$ and $S^+$ are independent Poisson point processes of quantum disks. 
\end {proposition} 

We emphasize at this point that this only determines the intensities of these three Poisson point processes up to the multiplicative constants $a_{l,-}$, $a_{r,-}$ and $a_+$. 

\begin{remark}
It is worth noticing, that using similar arguments as in \cite[Section~9]{dms2014mating}, one can show that it is possible to actually recover the whole initial quantum half-plane from the three Poisson point processes of discovered quantum disks.  The analogous feature will hold true also for the exploration of quantum disks and for the other explorations discussed in the next section. 
\end{remark}

\section {General CLE explorations on quantum half-planes}
\label {S4}

\subsection {Generalization of the previous results} 
\label {41}

We now want to study what happens when one explores/cuts a quantum half-plane with an independent  $\SLE_\kappa^\beta(\kappa-6)$ process for $\beta \in (-1, 1)$. We first explain how to generalize to those explorations the results of 
Section \ref {S3} that were dealing only with the case $\beta=1$ (and by symmetry, it yields the result for $\beta=-1$).

There are two strategies to try to derive the results for these more general explorations. One would be to use the results for the totally asymmetric cases and to consider approximations of the general case by ``side-swapping'' totally asymmetric explorations defined as follows:  For fixed $\eps>0$, one can approximate the law of $\eta$ by that of a process $\eta^\eps$ as follows. Let $p=(1+\beta)/2$. One tosses first a $p$ v.s.\ $1-p$ coin to decide if up to the first time at which the quantum length of its trunk is $\eps$, the process $\eta$ evolves as an $\SLE_\kappa^1 (\kappa- 6)$ process or an $\SLE_\kappa^{-1} (\kappa -6)$ process. Then, one tosses another independent coin to decide about the behavior of $\eta$ up to the time at which the trunk has quantum length $2 \eps$ and so on. This then requires to control well how the quantum surfaces cut out by $\eta$ can be approximated by those that are cut out by $\eta^\eps$. 
The other strategy is to essentially directly derive the results by simply generalizing the proofs that we have presented in the totally asymmetric cases. We opt here for the latter one, since the proofs will actually go along very similar lines:

Consider a quantum half-plane $\CW = (\h,h,0,\infty)$ and an $\SLE_\kappa^\beta (\kappa -6)$ process $\eta$ in $\h$ from $0$ to $\infty$, with $\kappa = \gamma^2 \in (8/3, 4)$. 
Just as in the cases $\beta=1$ and $\beta=-1$, we can consider the point processes of surfaces that are cut out by this process (when parameterized by the quantum length of the trunk) -- the only difference is that this time, some of the CLE loops will be traced clockwise and will lie to the left of the trunk, and some of the CLE loops will be traced counterclockwise and lie to the right of the trunk. 
We define the $\sigma$-field $\CF_{\tau}$ that contains the information about all the quantum surfaces that have been cut out by $\eta[0, \tau]$ from $\infty$. We define $\CG_t := \CF_{\tau_t}$. We then have the following generalization of Proposition \ref {prop:stationarity}:

\begin{proposition}[Stationarity for asymmetric explorations]
\label{prop:stationarity2}
For each $t \geq 0$, the quantum surface parameterized by the unbounded component of 
$\h \setminus \eta([0,\tau_t])$ with marked points $\eta(\tau_t)$ and $\infty$ is a quantum half-plane, which is independent of $\CG_t$.
\end{proposition}
\begin{proof}
We will use a combination of three inputs: (a) The loop-trump decomposition of $\SLE_\kappa^\beta (\kappa -6)$ -- in this proof, $\rho'$ will be chosen to be the value such that the trunk of such a process is an $\SLE_{\kappa'} (\rho'; \kappa' -6 - \rho')$ [this value is at this point unknown, but we know from Theorem \ref{thm:cle_perc_trunk} that it exists]. (b) The decomposition of wedges results like Theorem \ref {thm:wedge_welding}. (c) The imaginary geometry couplings from \cite {MS_IMAG} that allow to couple various SLE-type curves with each other (via an auxiliary GFF).

As in the proof of Proposition \ref {prop:stationarity}, we will construct the quantum half-plane as a subset of
a weight ${3\gamma^2}/{2}-2$ quantum wedge $\wt{\CW} = (\h,\wt{h},0,\infty)$.  
Let us first consider an independent $\SLE_{\kappa'}$ process $\eta'$ in $\h$ from $0$ to $\infty$ which is subsequently parameterized by its quantum length.  For each $t \geq 0$, let $\h_t$ be the unbounded component of $\h \setminus \eta'([0,t])$.  Then,  we know that the quantum surface $(\h_t,h,\eta'(t),\infty)$ is a $3\gamma^2/2-2$ quantum wedge.

If $f_t$ denotes the unique conformal map from $\h_t$ to $\h$ such that $f_t(z)/z \to 1$ as $z \to \infty$ and $f_t(\eta'(t)) = 0$, then this means that
\[ \wt{h}_t := \wt{h} \circ f_t^{-1} + Q \log|(f_t^{-1})'| \] 
has the same law of $\wt {h}$
(when we view the left and right sides modulo a global rescaling).  That is, $(\h,\wt{h}_t,0,\infty)$ is a weight $3 \gamma^2/{2}-2$ quantum wedge.  

Instead of coupling $\eta'$ with one SLE-type curve $F$ that lies to the right of $\eta'$ as in Section~\ref{Imgeom} and in the proof of Proposition~\ref{prop:stationarity}, we are now going to couple $\eta'$ with two curves $F_R$ and $F_L$ that respectively lie to its right and to its left. The domain in between these two curves will be a quantum-half plane, and conditionally on $F_R$ and $F_L$, the process $\eta'$ will be an $\SLE_{\kappa'} (\rho'; \kappa'-6- \rho')$ in this domain. 
The totally asymmetric case as treated in the proof of Proposition \ref {prop:stationarity} corresponds to the case where $F_L$ is a the negative half-line and $\rho'=0$.

To describe the coupling between $F_R$, $F_L$ and $\eta'$, it is convenient to view them as all obtained via the imaginary geometry from the same auxiliary GFF:  
We consider a GFF $h^\IG$ on $\h$ with (Dirichlet) boundary conditions given by $\lambda'$ on $\R_-$ and by $-\lambda'$ on $\R_+$ that is independent of $\wt {h}$, and we realize $\eta'$ as the counterflow line of $h^\IG$ from $0$ to $\infty$.

We now realize $F_L$ and $F_R$ as the flow lines from $\infty$ to $0$ with some well-chosen respective angles 
$\theta_L$ and $\theta_R$ as described below. For this specific choice, we know in particular from \cite {MS_IMAG} that: 
(i) The marginal law of $F_L$ is that of an $\SLE_\kappa (\rho_1; \rho_2)$ for some explicit values $\rho_1, \rho_2$. 
(ii) The conditional law of $F_R$ given $F_L$ is that of an $\SLE_\kappa (\rho)$ in the domain that lies to the right of $F_L$, for some explicit value $\rho$. (iii) The curve $\eta'$ stays in the ``middle'' domain inbetween $F_L$ and $F_R$ conditional law of $\eta'$ given $F_R$ and $F_L$  is that of an $\SLE_{\kappa'} ( \rho'; \kappa'-6- \rho')$ in this domain (for the particular value of $\rho'$ related to $\beta$).

To be specific: 
Choosing $\theta_L$ and $\theta_R$ (with conventions as described in Figure \ref {fig:angle_convention} -- a bit of care is needed as we are here in the ``opposite'' setup as in \cite {MS_IMAG}: the flow lines are going from $\infty$ to $0$ and the counterflow line $\eta'$ is going from $0$ to $\infty$) in such a way that $\theta_R - \theta_L =  {2\pi(6-\kappa)}/(4-\kappa)$ ensures that (iii) holds for some value $\rho'$, and choosing $\theta_L$ correctly ensures that (iii) holds for the particular value $\rho'$ related to $\beta$.   

If one then combines this with Theorem \ref {thm:wedge_welding} (applied twice, using (i) and (ii)), we then get that 
$F_L$ and $F_R$ divide the quantum wedge $\wt{\CW}$ into three independent wedges, and the choice of $\theta_R-\theta_L$ ensures that the quantum wedge inbetween $F_L$ and $F_R$ is a quantum half-plane. 

Similarly, for any positive $t$, we can apply the same procedure in the quantum wedge $(\h,\wt{h}_t,0,\infty)$.  
We can define the GFF $h_t^\IG$ in $\h$ obtained from $h^\IG$ via the imaginary geometry change of coordinates formula 
\[ h_t^\IG := h^\IG \circ f_t^{-1} - \chi \arg (f_t^{-1})',\]
where $\chi =  2/\sqrt{\kappa} - \sqrt{\kappa}/2 $. 
Then we define $\wt{F}_L^t$ (resp.\ $\wt{F}_R^t$) be the flow line of $h_t^\IG$ from $\infty$ to $0$ with angle $\theta_L$ (resp.\ $\theta_R$).  Then the quantum surface parameterized by the region of $\h \setminus (\wt{F}_t^L \cup \wt{F}_t^R)$ which is between $\wt{F}_t^L$ and $\wt{F}_t^R$ is again a quantum half-plane.  

Let $F_t^L = f_t^{-1}(\wt{F}_t^L)$ and $F_t^R = f_t^{-1}(\wt{F}_t^R)$.  We can then observe that  $F_t^L$ (resp.\ $F_t^R$) agrees with $F_L$ (resp.\ $F_R$) until it first hits $\eta'([0,t])$.  Moreover, we have that the region disconnected from $\infty$ in the middle region of $\h \setminus (F_L \cup F_R)$ by $F_t^L$ and $F_t^R$ is a quantum half-plane.  This proves the result since this corresponds exactly to the hull of an $\SLE_\kappa^\beta(\kappa-6)$ (here we use the loop-trunk decomposition of these processes) i.e. it shows that 
the unbounded component of 
$\h \setminus \eta([0,\tau_t])$ with marked points $\eta(\tau_t)$ and $\infty$ is a quantum half-plane.

\begin{figure}[ht!]
\begin{center}
\includegraphics[scale=0.85]{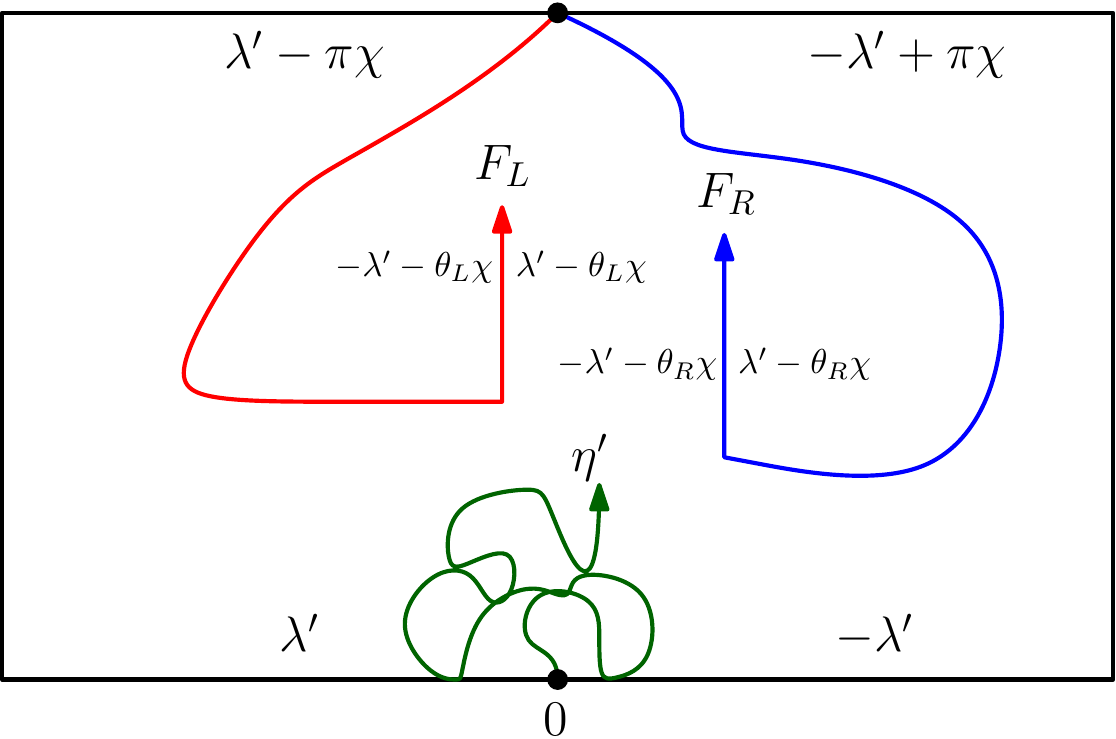}	
\end{center}
\caption{\label{fig:angle_convention}  The convention for the angles of the flow-lines $F_R$ and $F_L$ (with $\h$ described by a rectangle) from $\infty$ to $0$ -- together with some
the GFF boundary data. When we say that $F_L$ is the flow line from $\infty$ with angle $\theta_L$, we mean that if the path wraps around {\em to the right} then the boundary data is $-\lambda'-\theta_L \chi$ (resp.\ $\lambda' - \theta_L \chi$) on its left (resp.\ right) side when it is traveling north.  When we say that $F_R$ is the flow line from $\infty$ with angle $\theta_R$, we mean that if the path wraps around to the left then the boundary data is $-\lambda' - \theta_R \chi$ (resp.\ $\lambda' - \theta_R \chi$) on its left (resp.\ right) side when it is traveling north. For this convention, when $\theta_R - \theta_L =  {2\pi(6-\kappa)}/(4-\kappa)$ -- the piece inbetween the two curves is an quantum half-plane} 
\end{figure}

The fact that this quantum half-plane is independent of $\CG_t$ is derived exactly as in the case of Proposition~\ref{prop:stationarity}.
\end{proof}

We then define $L_t$ (resp.\ $R_t$) to be  the change in the boundary length of the left (resp.\ right) side of the unbounded component of $\h \setminus \eta([0,\tau_t])$ relative to the boundary lengths at time $0$. 
Exactly as in Section \ref {S31}, Proposition \ref{prop:stationarity2} implies the following statement: 
\begin {corollary} 
The processes
 $(L_t)_{t \ge 0}$ and $(R_t)_{t \ge 0}$ are independent $\alpha$-stable L\'evy processes
\end {corollary}

We now look at the Poisson point process of quantum surfaces that are cut out by the $\SLE_\kappa^\beta (\kappa -6)$. There are now four independent point processes to consider, corresponding to the quantum surfaces cut out by the trunk to its left, to those cut out by the trunk to its right, to  those encircled by CLE loops that lie to the left of the trunk, and to those that are encircled by CLE loops that lie to the right of the trunk -- we denote these point processes respectively by $S^{l,-}$, $S^{r, -}$, $S^{l, +}$ and $S^{r,+}$. 
 
Then exactly as in the final part of Section \ref {S32}, we see that Proposition \ref{prop:stationarity2} implies that
the four point processes $S^{l,-}$, $S^{r, -}$, $S^{l, +}$ and $S^{r,+}$ are independent Poisson point processes of marked quantum surfaces. By definition, the jumps 
of $R$ and $L$ correspond exactly to the boundary lengths of these four point processes of quantum surfaces. 
The four Poisson point processes of boundary lengths of these quantum surfaces have intensities that are multiples of $dl/l^{\alpha +1}$. We denote the four multiplicative constants  by $a_{l,-}$, $a_{r,-}$, $a_{l,+}$ and $a_{r,+}$ (with obvious notation). Note that all these multiplicative constants do depend also on $\beta$. We will also define 
\[ a_+ := a_{l, +} + a_{r, +} \quad\text{and}\quad a_- := a_{l, -} + a_{r, -},\]
which describe the intensities of positive and negative jumps of the stable process $(R_t + L_t)_{t \ge  0}$.

We can remark at this point that for each given CLE loop that is hit by the exploration mechanism, the fact that it will lie to the left or to the right of the trunk does not depend of the behavior of the exploration mechanism until it hits that loop. The conditional probability that the loop is on the left of the trunk (and therefore corresponds to a positive jump of $L$) is always  $(1-\beta)/2$. Hence, 
\begin {equation} 
\label {firstrelation}
a_{l,+}  = \frac { 1- \beta}{2} a_+ \quad\text{and}\quad a_{r,+} = \frac {1 + \beta}2 a_+.
\end {equation} 

\jason{
The next step is to show that these four point processes of quantum surfaces are Poisson point processes of quantum disks. 
This will be a consequence of the following generalization of Lemma~\ref{qdright}.

\begin{lemma} 
\label{qdrighttwo}
\label{QDRIGHTTWO}
Fix $\rho \in (-2,\kappa-4)$ (which we recall is the range of $\rho$ values so that $\BCLE_\kappa(\rho)$ is defined for $\kappa \in (2,4)$).
\begin{enumerate}[(i)]
 \item 
Consider a bead of a quantum wedge of weight $W= \gamma^2-(\rho+4)$ conditioned on its right and left boundary lengths. Draw an independent $\SLE_\kappa(\rho)$ from one of its marked points to the other one. Then, conditionally on its boundary length, the surface cut out to the right of the path is a quantum disk, and conditionally on their (right and left) boundary lengths, the surfaces cut out to its left are beads of a quantum wedge of weight $W'= \rho+2$.    
 \item 
Consider a bead of a quantum wedge of weight $W'= \rho+2$, conditioned on its right and left boundary lengths. Draw an independent $\SLE_\kappa(\kappa-6-\rho)$ from one of its marked points to the other one. Then, conditionally on its boundary length, the surface cut out to the right of the path is a quantum disk, and conditionally on their (right and left) boundary lengths, the surfaces cut out to its left are beads of a quantum wedge of weight $W= \gamma^2-(\rho+4)$.  
\end{enumerate}
\end{lemma}
Again, we will explain how to prove this lemma in the appendix. 
In exactly the same way as in Section \ref {S32}, Lemma~\ref{qdrighttwo} then implies the following 
generalization of Proposition \ref {threeindep}.}
\begin {proposition} 
\label {fourindep}
\label {FOURINDEP}
 The four Poisson point processes of quantum surfaces $S^{l,-}$, $S^{r,-}$, $S^{l,+}$ and $S^{r,+}$  are independent Poisson point processes of quantum disks. 
\end {proposition} 

We emphasize again that at this point that we have not determined the relative intensities of these four Poisson point processes of quantum disks (except for $S^{l,+}$ and $S^{r, +}$).

\subsection {Conditional proof of the relation between $\beta$ and $\rho'$} 

The following paragraphs are now devoted to the (conclusion of the) proofs of Theorems~\ref {mainthm<4}, \ref{mainthm<4-22}  and \ref{thm:beta_rho} (the ratios between the intensities of the Poisson point processes of quantum disks described in the previous sections, and the law of the trunk of $\eta$), 
{\em assuming the following statement} that we will derive in Section \ref {Scons} as a consequence of the analysis of exploration of CLEs drawn on 
quantum disks. 

\begin {lemma} 
\label {a-a+}
For any $\beta \in [-1, 1]$, one has $a_{l,-} = a_{r,-}$ (i.e., the intensities of negative jumps of $R$ and of $L$ are the same). 
\end {lemma}

\begin {proof}[Proofs of Theorems~\ref {mainthm<4}, \ref{mainthm<4-22}  and \ref{thm:beta_rho} (assuming Lemma \ref {a-a+})]

Let $\eta$ denote an SLE$_\kappa^\beta (\kappa -6)$ drawn in the upper half-plane. Its trunk $\eta'$ 
is an $\SLE_{\kappa'} ( \rho' ; \kappa' -6- \rho')$ process for some $\rho'$ (apart in the special cases $\beta =-1, 0, 1$, we do not know yet what the value of $\rho'$ is, and one of the goals of the coming paragraphs is actually to derive this value).

When we draw $\eta$ in $\h$, we know that the points at which it intersects $\R$ are exactly the same points at which the trunk $\eta'$ intersects $\R$. If we endow the half-plane with a quantum half-plane structure as before, then we will be able to interpret the quantum lengths intervals on the \jason{real line} in between these hitting points in two different ways -- one in terms of a L\'evy process related to $\eta'$ (and therefore of $\rho'$) and one in terms of a L\'evy process related to $\eta$ (and therefore to $\beta$).   

 Let $L$ (resp.\ $R$) denote the process of left (resp.\ right) boundary length variation (when parameterized by the quantum length of the trunk) of $\eta$.  We proved in Proposition~\ref{prop:stationarity2} that $L$ and $R$ are independent $\alpha$-stable L\'evy processes.  The properties of stable processes have been studied quite extensively (see \cite{bertoin96levy} and the references therein). In particular, one has a complete description of the Poisson point process of jumps of their running infimum. These jumps correspond exactly to the quantum lengths of intervals in between the points at which $\eta'$ hits $\R_-$ and $\R_+$.

By \cite[Chapter VIII, Lemma 1]{bertoin96levy} we know that the L\'evy measure for these jumps (corresponding to the ``ladder height process'') for $L$ is a constant times $x^{-\alpha P_L -1} dx$  where $dx$  denotes the Lebesgue measure on $\R_+$ and $P_L$ denotes the so-called positivity parameter of $L$, that is related to the ratio $u_L=u_L (\beta):= a_{l,+} / a_{l,-}$ between the intensity of positive versus negative jumps of $L$ by the formula 
\[ 
u_L =  \frac {\sin ( \pi \alpha(1- P_L) )}{\sin ( \pi \alpha P_L)}
.\] 
One has of course the same formula relating $P_R$ and $u_R$. 

But our LQG description allows one to get another expression for this exponent. We consider this time only the connected components of the complement of $\eta'$ that intersect $\R_-$ and $\R_+$, respectively. 
By Theorem~\ref{thm:sle_kp_rp}, the collection of these LQG surfaces form quantum wedges $\CW_L$ and $\CW_R$ of respective weights 
\[ W_L = \kappa-2 + {\kappa} \rho' / 4, \quad\text{and}\quad  W_R = \kappa-2+ {\kappa}(\kappa'-6-\rho')/4.\]
Recall from~\eqref{eqn:bessel_dim_wedge_weight} that a quantum wedge of weight $W$ is encoded by a Bessel process of dimension $1+ 2W/ \gamma^2$.  The dimensions of the Bessel processes which encode $\CW_L$, $\CW_R$ are therefore
\begin{align*}
\delta_L  := 3 - \alpha + \frac{\rho'}{2} \quad\text{and}\quad \delta_R := \alpha - \frac{\rho'}{2}.
\end{align*}
Recall that each excursion of the encoding Bessel process corresponds to a bead of the thin quantum wedge. We note that the intensity measure for Poisson point process of the  \jason{maxima of the successive excursions from $0$ of a Bessel process of dimension $\delta$ is given by} a constant times $x^{\delta- 3} dx$, and that the boundary length of the bead is a function of this excursion that scales like this maximum. It therefore follows that the intensity measure of the Poisson point process of successive boundary lengths of the beads is also a constant times $x^{\delta- 3} dx$. By independence of the beads, we similarly get that the lengths of the boundary lengths of the intervals in between the hitting points of $\eta'$ on $\R$ have the same scaling behavior.  \jason{That is, the process of boundary lengths of these intervals has an intensity measure given by a constant times $x^{\delta-3} dx$.}

Identifying $\delta_L - 3$ with $- \alpha P_L -1 $, we  get $ \alpha P_L  = \alpha -  ( \rho' / 2)  - 1$
and 
\[ u_L (\beta) = \frac { \sin (- \pi  \rho'/2  )}{ \sin (\pi(- \alpha +  \rho' /2) )} \]
By symmetry, one obtains the similar expression for $u_R$ replacing $\rho'$ by $\kappa' - 6 - \rho'$. We are now ready to conclude: 

\begin {itemize}
 \item
Let us first deduce the value of $a_{l,+}/ a_{l,-}$ when $\beta =-1$. We know that in this case $\rho' =  \kappa'-6$. Plugging this into the previous formula, we see that in this $\beta = -1$ case
\[ u_L (-1) = -\frac {\sin (2 \pi \alpha)}{\sin (\pi \alpha)}  
= -2 \cos (\pi \alpha). \]  
Similarly, one has the same formula for $a_{r,+}/ a_{r,-}$ in the case $\beta=1$,  
which completes the proof of Theorem~\ref {mainthm<4}. 

\item
For general $\beta \in (-1, 1)$, the previous formulas for $u_L(\beta)$ and $u_R(\beta)$ show readily that  expression for $$\frac {u_L} {u_R} = 
\frac {  \sin ( - \pi \rho' /2 )   } { \sin ( - \pi (\kappa' -6 - \rho') /2   )  } .$$ 
On the other hand, Lemma \ref {a-a+} tells us that  
$$\frac {u_L} {u_R} = \frac { a_{l,+} / a_{l,-} }{a_{r,+} / a_{r, -}  }= \frac {a_{l,+}}{a_{r,+}}.$$ 
But we also know from (\ref {firstrelation}) that $a_{l,+} / a_{r, +} = (1- \beta) / (1+\beta)$. 
Putting the pieces together, one then gets the relation between $\beta$ and $\rho'$ stated in   Theorem~\ref{thm:beta_rho}.

\item
Finally, we can note that for all $\beta \in (-1, 1)$, the ratio $u$ between the intensities of upward versus downward jumps of $L+R$ (for any given $\beta$) is 
$$ u = \frac { a_{l,+} + a_{r,+}}  {a_{l,-} + a_{r, -}} = \frac 1 2 ( u_L + u_R ) $$
and one can check (using the expressions of $u_L$ and $u_R$ in terms of $\rho'$ derived above) that $u$ is indeed always equal to $- \cos (\pi \alpha)$ as one would have expected. Since we also know that $a_{l,-}=a_{r,-}$ and 
$a_{l,+} /a_{r,+} = (1-\beta)/(1+\beta)$, this completes the proof of Theorem~\ref{mainthm<4-22}.
\end {itemize}
\end {proof} 

\ww{
Let us conclude this discussion with the following remarks:
\begin {itemize} 
 \item For each $\beta$, we have derived the {\em relative} intensities of the jumps of $R$ and of the jumps of $L$ -- this does however not provide the actual value of $a_{l,-} (\beta)$, say. 
\item Note that we did also not address here (or anywhere else in the paper) the issue of whether  $a_- := a_{r.-} + a_{l,-}$ and $a_+:= a_{r,+}+ a_{l, +}$ do depend on $\beta$ or not. 
\end {itemize}  
}

\section{CLE explorations of quantum disks} 
\label{subsec:quantum_disk_explore}

We are now going to derive the form of the jump law for the boundary length evolution of an $\SLE_\kappa^\beta (\kappa-6)$ process on an independent quantum disk.   
We will first study the chordal case (i.e., the exploration process starts from one boundary-typical point of the quantum disk and is targeted at another) by combining our results in the 
quantum half-plane with an absolute continuity argument. We will then deduce from this the results for other explorations. 

The absolute continuity ideas in this section will be actually quite similar to those of \cite{gm2017sle6_disk} in the special case $\kappa=6$. 

\subsection {The Radon-Nikodym derivative in the chordal case}

Suppose now that $\CD = (\D,h,-i,i)$ is a doubly marked quantum disk such that the quantum boundary length along the clockwise (resp.\ counterclockwise) segment of $\partial \D$ from $-i$ to $i$ is equal to $L_0$ (resp.\ $R_0$).  Let $\eta$ be an independent $\SLE_\kappa^\beta (\kappa -6)$ in $\D$ from $-i$ to $i$. We consider as before, the parameterization of its trunk by its quantum length. For each $t$, 
we denote by $E_t$ the event that the trunk has not yet reached $i$ before having quantum length $t$. 
On $E_t$, we define the quantum length of the left and right boundaries $(L_t, R_t)$ of the remaining-to-be-discovered domain $\CD_t$ (at the time $\tau_t$ for $\eta$ at which the quantum length of its trunk reaches $t$). 
The process $\Delta_t = L_t + R_t$ corresponds to the quantum boundary length of $\CD_t$, and this process will hit $0$ when $t$ reaches the total length of the trunk. 

The key proposition of this section is the following: 
\begin{proposition}
\label{prop:quantum_disk_bl}
On the event $E_t$, the Radon-Nikodym derivative between $(L_s,R_s)_{s \le t}$ 
and the corresponding stable L\'evy processes started from $(L_0, R_0)$ as described in the setting of an $\SLE_\kappa (\kappa -6)$ exploration of a quantum half-plane is $\Delta_0^{\alpha+1} / \Delta_t^{\alpha +1} $. 
Furthermore, the domain $\CD_t$ is (when conditioned on its boundary length) a quantum disk. 
\end{proposition}
\begin{proof}
Let us consider a quantum half-plane in some domain $D$ with marked boundary points that we call $0$ and at $\infty$.  Let $x_r$ (resp. $x_l$) be the point on the boundary which is $r_0$ (resp. $l_0$) units of quantum boundary length to the right (resp. left) of the origin. Recall that the domain with marked points $x_r$ and $\infty$ is also a quantum half-plane.

\begin{figure}[ht!]
\includegraphics[scale=.5]{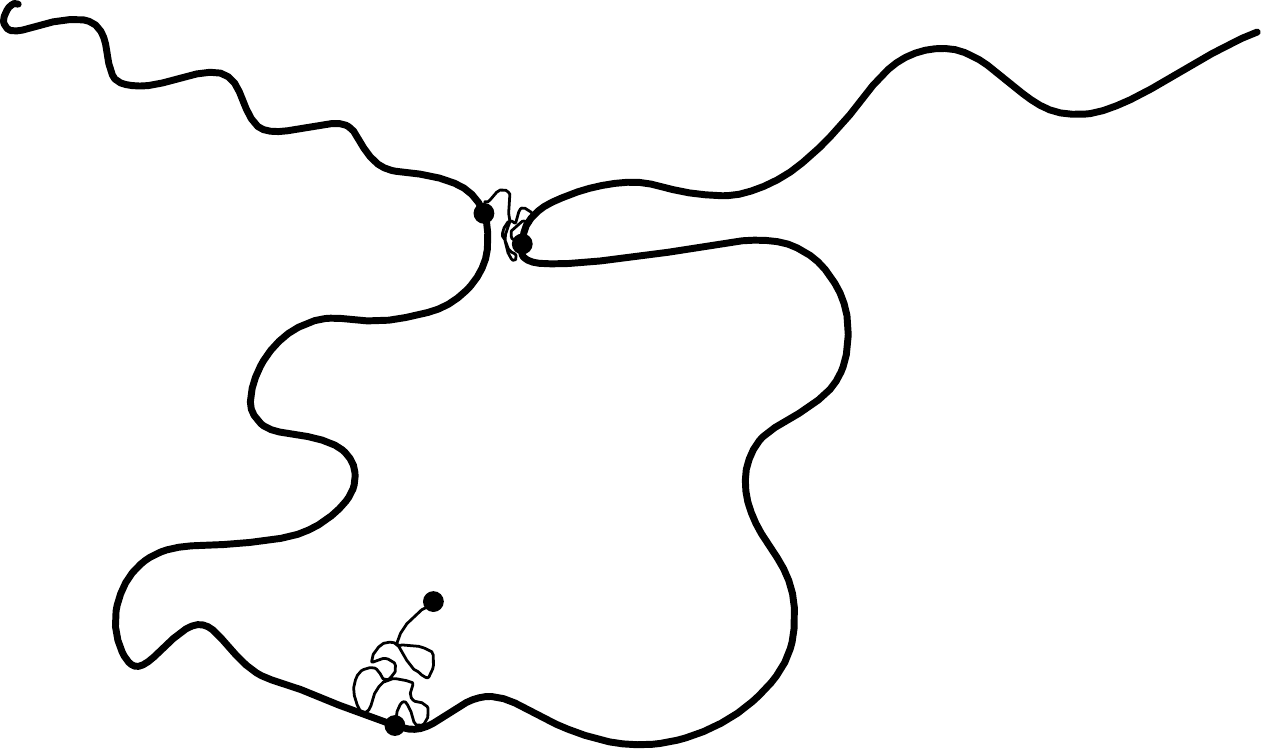} 
	\caption{Symbolic sketch: The two pieces of $\eta$ and $\eta_r$ starting from $0$ and $x_r$ respectively, and the event $F$.}
	\label{abscont}
\end{figure}

We let  $\eta$ and $\eta_r$ be two $\SLE_\kappa^\beta(\kappa -6)$ processes starting from $0$ and $x_r$ respectively, both targeting $\infty$ and that are coupled so that they are part of the same $\SLE_\kappa^\beta (\kappa -6)$ exploration tree.

We are going to prove the proposition by considering an event $F$ where $\eta_r$ almost immediately cuts off a quantum disk of length close to $l_0 + r_0$ with $0$ on its boundary.  Conditionally on this event, $\eta$ will locally evolve like an $\SLE_\kappa^\beta (\kappa-6)$ drawn in this quantum disk. On the other hand, one can also first let $\eta$ evolve for a while (and we know what LQG surfaces it cuts out, since it is an $\SLE_\kappa^\beta (\kappa -6)$ on a quantum half-plane), and then look at how likely it is that $F$ occurs. The sketch in Figure~\ref{abscont} is done in such a way to indicate that the event $F$ in fact very much depends on the quantum half-plane (and less on the behavior of $\eta_r$).  

More specifically, we fix $\epsilon > 0$ and let $\sigma_\epsilon$ be the first time that the left boundary length process associated with $\eta_r$ makes a downward jump of size at least $\epsilon$.  Fix $\zeta \in (0,1/2)$.  We then let $F_\epsilon$ be the event that the quantum boundary length of the arc joining $\eta_r(\sigma_\epsilon)$ and  $0$ is contained  in the interval $[l_0, l_0 +\epsilon^\zeta]$.

Let us now evaluate the probability of $F_\eps$ when $\eps \to 0$: 
We note that the supremum of the absolute value of the left boundary length process of $\eta_r$ up to time $\sigma_\epsilon^-$ is very unlikely to be larger than $\epsilon^{2\zeta}$ when $\eps \to 0$, provided $\zeta > 0$ is chosen sufficiently small. We can for instance choose $\zeta$ so that this probability is bounded by $C \eps^3$.  
Noting that the jump from time $\sigma_\epsilon^-$ to $\sigma_\epsilon$ is taken from the intensity measure for the left boundary length process conditioned on the jump having size at least $\epsilon$, which has density 
 $c \epsilon^\alpha u^{-\alpha-1} \one_{[\epsilon,\infty)}(u)$, $c > 0$ a constant, 
 with respect to Lebesgue measure, we see that
\[ - C\eps^{3}  + \int_{l_0 +r_0+\eps^{2\zeta}}^{l_0+r_0-\eps^{2\zeta}+\eps^{\zeta}} c \eps^\alpha u^{-\alpha-1} du \le \p[ F_\epsilon ]  \leq C\eps^{3}  + \int_{l_0 +r_0 -\eps^{2\zeta}}^{l_0+r_0+\eps^{2\zeta}+\eps^{\zeta}} c \eps^\alpha u^{-\alpha-1} du, \]
 so that 
\begin{equation}
\label{eqn:f_eps}
\p[ F_\epsilon ] = c (1+o(1)) \eps^{\alpha + \zeta} (r_0+l_0)^{-\alpha -1}.
\end{equation}

Let $\CF_t$ be the filtration generated by the quantum surfaces cut off by $\eta$ up to the moment when its trunk has reached quantum time $t$. Recall that the unbounded (i.e., with $\infty$ on its boundary) connected component of the complement of $\eta[0,\tau_t]$ with marked points $\eta_{\tau_t}$ and $\infty$ is a quantum half-plane 
$\CW_t$. 

When $x_r$ and $x_l$ have not yet been swallowed at this time (we call this event $S_t$), we can define 
 $\wt{L}_t$ (resp.\ $\wt{R}_t$) as the boundary length (in $\CW_t$) between $x_r$ and $\eta_{\tau_t}$, and between $x_l$ and $\eta_{\tau_t}$ respectively. Identity \eqref{eqn:f_eps} applied to the evolution of $\eta_r$ in $\CW_t$ shows that
\[ \p[ F_\epsilon \giv \CF_t, S_t]  = c (1+o(1)) \eps^{\alpha + \zeta} (\wt{L}_t + \wt{R}_t)^{-\alpha-1}. \]
Note also that the domain cut out by $\eta_r$ at this time is then also a quantum disk of boundary length close to $\wt{L}_t + \wt{R}_t$ (in the $\eps \to 0$ limit, this will then ensure that the law of $\CD_t$ in the proposition is a quantum disk). 

Hence, if $A$ is an $\CF_t$-measurable event with $\p[A, S_t] > 0$ and $\p[ A, S_t \giv F_\epsilon] > 0$ then, as $\eps \to 0$, 
\[ \frac{\p[ A, S_t \giv F_\epsilon]}{\p[ A, S_t ]}
=  \frac{\p[ A, S_t, F_\epsilon]}{\p [F_\eps] \p[ A, S_t ]} = \frac{\p[ F_\epsilon \giv A, S_t ]}{ \p[ F_\epsilon]} =(1+ o(1)) \times \frac { (\wt{L}_0 + \wt{R}_0)^{\alpha+1}}{(\wt{L}_t + \wt{R}_t)^{\alpha+1}}, \]
from which the proposition easily follows (one can for instance use scaling to compare the evolution in a disk of fixed boundary length $\Delta_0$ with the evolution in a disk a boundary length close to $\Delta_0$).
\end{proof}

\subsection {Proof of Lemma \ref {a-a+}}
\label {Scons}
Proposition \ref {prop:quantum_disk_bl} provides a complete description of the c\`adl\`ag Markov process $(R_t, L_t)$ corresponding to the explorations of quantum disks as an $h$-transform of the pair of 
L\'evy processes corresponding to the exploration of a quantum half-planes.

This allows in particular to describe the jump-rates of  
the process $(L,R)$ in the chordal exploration of a quantum disk. Loosely speaking, by taking into account the Radon-Nikodym derivatives just before and just after a jump (that do favor or penalize jumps according to whether this Radon-Nikodym derivative increases of decreases), we infer that the rate of jumps from $(L,R)$ to $(L - \ell, R)$ (in the quantum disk setting) is now 
\[ \frac {a_{l,-}}{\ell^{\alpha+1}} \times \frac {\Delta^{\alpha+1}}{ (\Delta -\ell)^{\alpha +1}}
\times \one_{ \ell < L} \]
and the corresponding formula for the jumps from $(L,R)$ to $(L, R-\ell)$.

To see that $a_{r,-} = a_{l,-}$, we will use the fact that two $\SLE_\kappa^\beta (\kappa -6)$ processes that are targeting different points can be chosen to coincide up to the first time at which they separate the two target-points from one another.
More specifically, let us consider a quantum disk of boundary length $4$, and pick a boundary-typical point $x_0$. Then, we choose the boundary point $x_1$ (resp. $x_1'$) that lie $1$ boundary length unit to the right (resp. to the left) of $x_0$, and we let $x_2$ denote the antipodal point of $x_0$ (lying at boundary length $2$ in either direction). 
The explorations starting from $x_0$ targeting $x_1$ and $x_1'$ can be chosen to coincide up to the first time at which $x_1$ and $x_1'$ gets discronnected. In terms of the processes $(R,L)$, this means that two processes started from $(3,1)$ and $(1,3)$ respectively  can be made to correspond to the same exploration (until it disconnects the two points at boundary distance $1$ from the starting point from one another). In particular, their jumps would be corresponding to each other up to this time.  We can then evaluate and compare the behavior as $t \to 0$ of the probability that before time $t$, 
the exploration process has hit the boundary interval of boundary length $1$ centered around the antipodal point of $x_2$. This amounts to compare the jump rates from $(1, 3)$ to the set $\{ (1, x), x \in (1/2, 3/2) \}$ and from $(3, 1)$ to the set $\{ (x, 1), x \in (1/2, 3/2) \}$. 
The fact that these two rates are equal implies immediately that $a_{r,-} = a_{l,-}$.

\subsection {Proof of Theorems~\ref{thm1} and~\ref{thm2}} 
\label {S.Explo}

We can now also complete the proofs of Theorems~\ref{thm1} and~\ref{thm2}.  For this, we will 
explore quantum disks by totally asymmetric $\SLE_\kappa (\kappa -6)$,  but we now use different rules to decide in which direction to branch when it separates two domains).  

A natural variant of the chordal exploration of the quantum disk is, at each splitting time at which a quantum surface is actually cut off by the trunk (so the domain with quantum boundary length $\Delta$ splits into two domains of boundary length $\ell$ and $\Delta -\ell$), to choose to branch into 
the domain with largest boundary length. In other words, one chooses to go into the domain with boundary length $\ell$ if and only if $\ell > \Delta/ 2$. We will refer to this as the $(q=\infty)$-exploration mechanism. 

One way to approximate this process is, at each time at which the quantum length of the trunk is a multiple $n \eps$ of $\eps$, to change the target point of the chordal exploration, and to choose it to be the ``antipodal'' point of the exploration point (so that at this time, the boundary lengths of the two boundary arcs from that point to the target point are identical and equal to $\Delta_{n\eps}/ 2$). Then, during the time-interval $[n \eps, (n+1) \eps]$ one uses the chordal exploration with this target point. When $\eps$ tends to $0$, this approximation converges to this exploration variant (indeed, it is going to be the same on any bounded time-interval, provided $\eps$ is small enough). 

Our previous description of the boundary length process in the chordal case can be transcribed in terms of the jumps of the chordal exploration for this variant. First, the positive jumps of the total boundary length process $\Delta_t$ of the currently-explored disk will remain the same as for the chordal exploration, as it is not affected by the branching rule (i.e., the process does not branch at those times). So, when $\Delta_t= \Delta$, it will jump to $\Delta+\ell$ with a rate proportional to 
$$\frac {a_{r,+} + a_{l, +}} {\ell^{\alpha+1}}  \times \frac {\Delta^{\alpha+1}}{(\ell+\Delta)^{\alpha+1}}.$$
For the negative jumps, one gets a rate 
$$  \frac {a_{r,-} + a_{l,-}} {\ell^{\alpha+1}} \times \frac {\Delta^{\alpha+1}}{(\Delta -\ell)^{\alpha +1}}
\times \one_{ \ell < \Delta/2}. $$ 

This completes the proofs of Theorem~\ref{thm1} and Theorem~\ref{thm2} (up to proving Lemma~\ref{qdright} which will be completed in the appendix).

It is worth mentioning here that other exploration mechanisms (i.e., rules on how to branch when the domain is split into two smaller ones) of quantum disks are quite natural to consider. 
In particular, one can choose a fixed $q > \alpha$, 
and, at each splitting time -- when the two ``available'' domains have respective boundary lengths $\ell$ and $\Delta - \ell$, to choose to branch into them with respective probabilities 
$\ell^q / ( \ell^q + (\Delta- \ell)^q)$ and $(\Delta- \ell)^q / 
( \ell^q + (\Delta- \ell)^q)$. The choice of $q$ ensures that during each positive time interval, the probability that this variant does actually coincide with the exploration that always chooses the domain with largest available boundary length is positive, so that it is possible to define this $q$-exploration mechanism as a simple modification of the one that branches into the largest one (which corresponds to $q= \infty$). 
The case where $q=\alpha+1/2$ will actually turn out to be natural in view of Section~\ref{S.measure}.

\subsection {BCLE on LQG discussion} 
Let us rephrase some of the previous results in terms of the boundary conformal loop ensembles introduced in \cite[Section~7]{cle_percolations} (we refer to this paper for a definition of these natural ensembles of loops).  Let us now consider the trunk $\eta'$ of the totally asymmetric $\SLE_\kappa (\kappa -6)$ drawn on an independent quantum disk (but starting from a boundary typical point and targeting another boundary typical point). On the one hand, we know that the connected components of the complement of $\eta'$ that are totally surrounded by $\eta'$ are all quantum disks conditionally on their boundary lengths. If we then consider one of these connected components $\CD$ that is surrounded clockwise, then the loop-trunk decomposition of $\eta$ indicates that in order to trace the rest of $\eta$ in this connected component, one has to trace a $\BCLE_\kappa (3 \kappa /2 -2)$ (or equivalently a $\BCLE_\kappa (- \kappa /2)$ if one orients it in the other way).  But we have just shown that the connected components traced by this BCLE will consist of independent quantum disks given their boundary length. In other words, when one traces an independent $\BCLE_\kappa (3 \kappa /2 - 2)$ on top of an independent quantum disk, then all the connected components will be quantum disks conditionally on their boundary lengths. 

In fact, if we apply the very same argument to the side-swapping $\SLE_\kappa^\beta (\kappa -6)$, we get the following general statement that is interesting on its own right:  
\begin{proposition}[BCLE on quantum disks] 
Suppose that $\kappa \in (8/3, 4)$ and that $\rho$ is in the admissible interval $ (\kappa/2 -4, \kappa/2 -2 )$ so that one can define a $\BCLE_\kappa (\rho)$.  Conditionally on their boundary lengths, the connected components of the complement of a  $\BCLE_\kappa (\rho)$ traced on an independent quantum disk are all quantum disks. 
\end{proposition}

This result actually also holds for $\kappa \in (2, 8/3]$, and can be derived more directly, but the proof used here (for presentation purposes) only works for $\kappa \in (8/3, 4)$. Actually, see \cite{mswupcoming}, the statement is also valid for the BCLEs with $\kappa'$ instead of $\kappa$.  

\section{The natural quantum measure in the CLE carpet} 
\label {S.measure} 

\subsection {Preamble} 

Theorems~\ref{thm1} and~\ref{thm2} show that it is possible to transcribe questions about CLE on LQG disks in terms of features of some 
explicit multiplicative cascade-type processes that are built on variants of asymmetric L\'evy processes.
These multiplicative cascades have actually been studied in their own right (they are also closely related to branching random walks) 
and a number of results that are actually available in the literature can be rather directly applied to 
construct natural objects related to CLE on LQG surfaces.
This is illustrated in this section with the natural LQG measure that lives in the CLE carpet -- this is a natural object to consider in view of the scaling limit of decorated planar maps or of planar maps with large faces \cite{bck,bbck,ccm2017lengths,dadoun} (as the suitably renormalized number of points in the latter should converge to the natural LQG measure on the CLE carpet). 

We would like to emphasize one point here: As shown in \cite{msw2016notdetermined}, when one explores a CLE using an $\SLE_\kappa(\kappa -6)$ curve for $\kappa \in (8/3, 4)$, one uses some extra randomness that is not present in the $\CLE_\kappa$ (i.e., the $\SLE_\kappa (\kappa -6)$ process is not a deterministic function of the $\CLE_\kappa$). In other words, in the exploration tree of CLE on LQG that gives rise to the branching trees and to the multiplicative cascades, one uses three random inputs: 
(i) The randomness used to define the CLE, (ii) the randomness of the LQG and (iii) the additional randomness of the CPI curve used to draw the trunk of the $\SLE_\kappa (\kappa -6)$. So, when one constructs some random variables using the branching tree, some additional work is needed to check (if it is the case) that they are in fact a deterministic function of (i) and (ii) only. In the present case, the measure is the one that appears in the papers \cite {bbck,dadoun} and the contribution of this section is to show that it is indeed a function of the CLE and LQG measure only.  

This section will be structured as follows. For the readers who are not so acquainted with the results on these multiplicative cascades and as a warm-up to Section~\ref{5.4}, we briefly survey in Section~\ref{5.2} some of the basic general features that we will use here. We recall some results of \cite {bbck} who construct the ``natural area measure'' on the boundary of the branching tree ${\mathcal T}$. Finally, in Section~\ref{5.4}, building also on a result of \cite{dadoun}, we show Theorem~\ref{thmmeasure} which proves that this measure is indeed a function of the CLE and the LQG measure only. 

\subsection {Background and heuristics: Construction of the measure on $\CT$} 
\label {5.2}

Suppose that we are given a law ${\mathcal U}$ on random decreasing families of positive reals $U:= (U_n)_{n \ge 0}$ such that 
$\E [ \sum_{ n \ge 0} U_n ] = 1$. This is viewed as the law of the collection of smaller fragments that an object of size 1 gets fragmented into (an object of size $x$ would be fragmented into a collection with  sizes distributed like $(xU_n)_{n \ge 0}$). 
We suppose that we are given an i.i.d.\ copies of $U$ that we denote by $U^{n_1, \ldots, n_j}$ that are indexed by 
finite sequences of non-negative integers.
The multiplicative cascade $(M_j)_{j \ge 0}$ is then defined by $M_0 = 1$ and  
$$ 
M_j := \sum_{n_1, \ldots , n_j} U_{n_1} U^{n_1}_{n_2} \ldots U^{n_1, \ldots, n_{j-1}}_{n_j}.
$$
These cascades are clearly positive martingales. A sufficient condition for it to converge in $L^1$ to a positive limit $M_\infty$ is that 
there exists $r>1$ such that 
$\E[ \sum_n (U_n)^r ]$ and $\E[ ( \sum_n U_n)^r ]$ are both finite
(these types of results were first established in the setting of branching random walks, see \cite{Biggins,BK,Liu}).
When such a convergence in $L^1$ holds, it is possible to also obtain convergence along ``stopping lines'' in the branching tree (using generalizations of the optional stopping theorem, so that the value of $M$ on this stopping line is the conditional expectation of $M_\infty$ given the information discovered before that line, see e.g., \cite {BK}). 

Such multiplicative cascades arise naturally in the settings of the tree structures ${\mathcal T}$ and $\wt {\mathcal T}$. Let us quickly explain how to quickly identify such a martingale in the latter case, using the fact that the expected value of the quantum area ${\mathcal A}$ of a quantum disk of boundary length $1$ is finite (we will briefly recall how to derive this fact in Remark~\ref{finiteexpectation}). One can draw on this quantum disk an independent nested $\CLE_\kappa$. Exploring such a nested $\CLE_\kappa$ provides (when the CPI branches are parameterized by the quantum length) exactly the branching structure ${\mathcal T}$ (this follows from our previous results, since in each of the discovered $\CLE_\kappa$ loops, the process continues its exploration). 
If one explores the nested $\CLE_\kappa$ up to some ``finite depth'' (meaning for instance that one has not explored inside any of the $k$th level $\CLE_\kappa$ loops for some finite $k$), then simply because the $\CLE_\kappa$ carpets have zero Lebesgue measure, it follows that ${\mathcal A}$ is equal to the sum of the areas of the remaining-to-be-explored disks. If we order their boundary lengths in decreasing order by $L_0, L_1, \ldots$, we immediately get (using the scaling rule for the area) that 
$ {\mathcal A} = \sum_{n \ge 0} L_n^2 {\mathcal A}_n$,
where $({\mathcal A}_n)_{n \ge 0}$ is a collection of i.i.d.\ copies of ${\mathcal A}$. In particular, if ${\mathcal F}$ denotes the information collected before this partial exploration, one gets 
$ \E[ {\mathcal A}  | {\mathcal F} ]$ is equal to  $\sum_{n \ge 0} L_n^2$.  
So, we see readily that when one chooses $(U_n)_{n \ge 0}$ to have the law of $(L_n^2)_{n \ge 0}$, then one is in the previous setting. The fact that $M_1 = \E[ {\mathcal A} | {\mathcal F} ]$ readily implies that the martingale is uniformly integrable and converges in $L^1$ and almost surely to ${\mathcal A}$ (so in this particular case, the convergence result can be derived directly, without referring to the general results on these cascades).

Let us now turn our attention to the tree ${\mathcal T}$. Suppose that we  only explore the particular branch of the labeled tree with the property that when a splitting occurs, one explores the largest of the two available descendants (but along the way, one keep tracks of the sizes of the unexplored descendants). We choose to stop 
exploring this branch at the first time at which the size exits the interval $(1/2, 2)$ 
(this choice turns out to be convenient for our purposes here). In the CLE exploration wording, this means that we 
are using the aforementioned $(q=\infty)$-exploration (corresponding to the CLE exploration of an LQG disk that always branches into the disk with the largest boundary length among the two available choices) -- and we stop it when the boundary length of the currently explored disk exits $(1/2, 2)$.

At that time $T$, we order the sizes of the tips of the unexplored branches of ${\mathcal T}$ (i.e., the boundary lengths of  to-be-explored quantum disks in the carpet) in decreasing order $l_0, l_1, \ldots$. Note that by construction, all the sizes of offsprings that were left aside in the exploration are smaller than $1$. Hence, all $l_j$'s with the possible exception of $l_0$ (if the size at $T$ is $2$ or larger) are smaller than $1$.

We can also interpret the same exploration as an exploration of the larger tree-structure $\wt{\mathcal T}$. In other words,  we can now keep track of the positive jumps corresponding to the discovered CLE loops as well as the $l_n$'s. In this setting, at time $T$, one then has a larger collection of to-be-explored disks with boundary lengths $L_0, L_1, \ldots$, where $(l_n)$ forms a subsequence of $(L_n)$. Note that by construction, all of the $L_n$'s are smaller than $1$, with the possible exception of $L_0$ and $L_1$. The same argument as above shows that $\E[ \sum_{n \ge 0} L_n^2 ] = 1$. 

It is easy to get some crude information on the joint law of $(l_n)_{n \ge 0}$: 

\begin{enumerate} 
\item The law of $T$ has an exponential tail (just because of the lower bound on positive jumps of size greater than $2$ before $T$).
\item The largest boundary length 
$l_0$ has a polynomial tail distribution of the type $c x^{-2 \alpha -1 } \le \p [ l_0 \ge x ] \le C x^{-2\alpha - 1} $ 
as $x \to \infty$ (this just comes from the fact that the intensity measure of positive jumps is bounded from above and from below by a constant  times  $dl / l^{2\alpha+2}$ as $l \to \infty$) for some absolute constants $c$ and $C$. 
\item One has also good information about the numbers of small positive and small negative jumps before $T$. For instance, by trivially bounding the intensity rates of negative jumps of size $l$ before time $T$ by a constant times 
$1/l^{\alpha+1}$, one sees readily that conditionally on $T$, the expected number of $l_n$'s that lie in $[2^{-k}, 2^{-k+1}]$ is bounded by $C 2^{\alpha k} T$ for some constant $C$. 
\end {enumerate} 
Such estimates can then be easily combined to check that  $\E[ \sum_n l_n^{\beta } ] < \infty$ for all $\beta \in (\alpha, 2\alpha +1)$. Furthermore, this quantity will tend to $\infty$ as $\beta \to \alpha$ (using the lower bound on the tail of $l_0$), and we know on the other hand that 
$ \E[ \sum_n l_n^2 ]  < \E[ \sum_n L_n^2 ] = 1$. 
It follows that there exists a value $\delta \in (\alpha, 2)$ such that 
$ \E[ \sum_n l_n^\delta ] = 1$.
We will see in a moment that $\delta$ is in fact equal to $\alpha+1/2$ but this will not be needed at this point. 
So, if we choose $(U_n := l_n^{\delta})_{n \ge 1}$, then we are in the setup described above. One can check (using the 
previous type of estimates on the joint law of $l_n$'s) that the $L^1$-convergence criteria are fulfilled. 
In particular, one can then deduce that if $M_\infty$ denotes the limit of the martingale $M_n$, and if one stops the exploration of the 
tree at any ``finite'' stopping line, the sum of the $\delta$-th powers of the boundary length of the remaining-to-be-explored disks is equal to the conditional expectation of $M_\infty$. All this corresponds exactly to the convergence of the {\em Malthusian} martingales in the branching structure ${\mathcal T}$ stated in \cite{bbck}, see also \cite[Section~3.1]{dadoun}.

The quantity $Y:= M_\infty$ is called the {\em intrinsic area} of the growth-fragmentation tree ${\mathcal T}$ in \cite {bbck,bck2}. This terminology comes from the fact that this tree arises as scaling limit of peeling processes on random planar maps with large faces, and that this {\em intrinsic area} then corresponds to the scaling limit of the renormalized number of faces of these planar maps. 

In our continuum setting where ${\mathcal T}$ is directly related to a CLE on LQG, we will want to interpret $Y$ as the ``total natural quantum mass'' of the CLE carpet. There is however a catch: 
This construction of $Y$ goes via ${\mathcal T}$ and therefore uses also the randomness that is given by the exploration tree. It is not clear at this stage that it is a deterministic function of the LQG measure and the CLE only (i.e., that it does not depend on the additional randomness of the CLE exploration).

\begin{remark} 
In the growth-fragmentation tree set-up, the very same construction allows to define the 
quantum area of any subtree of ${\mathcal T}$ consisting of the descendants of one point in the tree, which in turn allows to define an actual proper {\em area measure} on the boundary of the growth-fragmentation tree (that corresponds to the scaling limit of the counting measures on the associated planar maps).  

Following the ideas of \cite{lpp} (this is also mentioned and used in \cite{dadoun}), it is then also possible to describe the time-reversal of the exploration process that is targeting a point chosen according to this  measure on the boundary of the tree, in terms of another L\'evy process. This type of feature can be useful in the context of planar maps, as it would correspond to exploring the map towards a uniformly chosen point.  
Note also that further features, such as the conditional law of ${\mathcal T}$ given $Y$, or properties of the law of $Y$ are derived in the recent paper \cite {bck2}.
\end{remark}

\begin {remark}[Comment on the value of the exponent]
The fact that $\delta = \alpha + 1/2$ is actually stated in \cite[Proposition 5.2]{bbck}. Indeed, the particular, the jump measure described in the displayed equation just before (28) in \cite{bbck} is precisely the one that we are investigating here, so that we are looking for the value of $q \in (\alpha, 2)$ for which $\kappa_\alpha (q) =0$ in the last formula of \cite[Proposition 5.2]{bbck}, which is obviously $q=\alpha + 1/2$ (due to the $\cos (\pi (q-\alpha))$ term).

Note that \cite[Proposition 5.2]{bbck} and its proof give a more general result than the particular case that we need here. Proving that $\delta = \alpha +1/2$ directly is actually not easy to perform without the help of a software package like Mathematica for guidance about the right change of variables. One has  to check that the contribution of the positive jumps, of the negative jumps and of the L\'evy compensation for this process exactly cancel out. In particular, one has to evaluate the asymptotics as $\eps \to 0$ 
\[ \int_\eps^\infty \frac {- \cos (4 \pi/\kappa) dl}{ l^{\alpha+1} (1+ l)^{\alpha +1} }( (1+l)^{\alpha+1/2} - 1) 
- \int_\eps^{1/2} \frac {dl}{l^{\alpha+1} (1-l)^{\alpha +1}} (l^{\alpha+1/2} + (1-l)^{\alpha +1/2} - 1 ).\] 

Actually, an indication of the fact that $\delta = \alpha + 1/2$ (that could be turned into a lengthy and very convoluted proof compared to the previous direct computation) is to notice that the Hausdorff dimension of the $\CLE_\kappa$ carpet (in the Euclidean metric) for $\kappa \in (8/3, 4]$ is known to be almost surely equal to 
$ D =  2 - (8-\kappa)(3\kappa-8)/(32\kappa)$ 
(see \cite{ssw2009radii} for the upper bound and the result for the ``expectation dimension'' which is the relevant one to apply the KPZ ideas -- see \cite{nw2011carpet} for the lower bound). 
On the other hand, this CLE carpet is independent of the LQG measure, so that the KPZ ideas from \cite{DS08} should be applicable here. And, when one applies the KPZ formula, one indeed gets the formula $\delta = \alpha + 1/2$.  
\end {remark}

\ww{
\begin {remark}[Finite expectation] 
\label {finiteexpectation}
For sake of completeness, let us comment on the fact (that will also be used in the next section) that the expected quantum area of a disk of boundary length $1$ is finite i.e., that $\E [ {\mathcal A} ] < \infty$ in the notation used above.

There are several rather direct ways to explain this based on the current literature. One option is to use the L\'evy tree results of the present paper. Indeed, when one explores all branches of the L\'evy tree of the CLE exploration, up to the first time (along each of the branches) at which one discovers a $\CLE_\kappa$ loop (this corresponds to stopping at each first positive jump), one gets a collection of boundary lengths $(u_n)_{n \ge 1}$ of to be explored loops. The sum of the quantum areas of these disks is almost surely equal to the total area of the initial disk. In other words, using the scaling rule
$$ \CA = \sum_{n \ge 1} u_n^2 \CA_n,$$ 
where $(\CA_n)_{n \ge 1}$ is a collection of i.i.d. copies of ${\mathcal A}$, that are also independent of $(u_n)_{n \ge 1}$.

If we iterate this CLE exploration procedure $k$ times (i.e. we look at the nested $\CLE_\kappa$ loops at depth $k$), we get a similar expression 
\begin {equation} 
 \label {iteration} 
\CA = \sum_{n \ge 1} u_{k,n}^2 \CA_{k,n}.
\end {equation}
On the other hand, $Z_k:= \sum_{n \ge 1} u_{k,n}^2$ is exactly the martingale that appears in the multiplicative cascades associated to the branching tree, as studied for instance in \cite {ccm2017lengths}. It is known to converge almost surely and in $L^1$ to a finite limit $Z_\infty$. Furthermore it is easy to see (either directly from the property of the LQG measure, or because of the convergence of $Z_k$) that $\max_{n} u_{k, n} \to 0$ almost surely as $k \to \infty$. 

We can now see that if $\E [ \CA ] = \infty$, then the right-hand side of (\ref {iteration}) would converge in probability to infinity i.e., that there is no way in which the sum on the right-hand side could converge in distribution to a finite positive random variable. Hence, we conclude that $\E[ \CA ] < \infty$, and furthermore (by looking at the limit in probability of the right-hand side of \eqref{iteration}) 
that ${\mathcal A} := Z_{\infty} \E [ {\mathcal A} ]$. 

We can notice at this point that in fact, the following stronger fact holds: 
\begin {equation} 
\label {integrab}
\E[\CA^p] < \infty \quad \text{for all}\quad p \in  (0,4 / \gamma^2). 
\end {equation}
This can be viewed as a direct consequence of the formula for the density of $\CA$ worked out in \cite{AG} (that therefore also provides a proof of the fact that $\E [ \CA ]$ is finite).

Note that this threshold $4  /  \gamma^2$ is the one that already appears for the expectation of the powers of the area measure of a bounded domain when one starts with a GFF with Dirichlet boundary conditions. Things are a little bit less straightforward here, as the boundary conditions are ``of Neumann type'', but loosely speaking, the boundary length constraints (since one now looks at a quantum disk with finite boundary length) ensure that the tail distribution of the quantum area gives rise to the same threshold. 
\end {remark} }

\subsection {Uniqueness of the LQG carpet mass} 
\label {5.4}

Our goal in this section is to show that the random variable $Y$ that is described in the previous section is equal to the limit in probability of some constant times $\eps^{\alpha + 1/2}$ times the number $N_{[\eps, 2\eps]}$ of outermost CLE loops with quantum boundary lengths in $[\eps, 2 \eps]$. As such a description of $Y$ only depends on the CLE and on the LQG measure, proving this statement would indeed show that $Y$ is independent of the additional randomness that comes from the CPI  tree.

Let us first emphasize that this can again be viewed a statement about the labeled branching tree ${\mathcal T}$ only. Indeed 
$N_{[\eps, 2\eps]}$ is just the total number of ``positive jumps'' of size in $[\eps, 2\eps]$ in the tree. So, in a way, one needs to transfer the previous construction of $Y$ (that was given more in terms of number of small negative jumps) into a description in terms of small positive jumps. While certainly not surprising to any specialist of these trees, this statement does not seem to be written up in the existing literature, so we provide a proof here. Before that, let us provide a brief heuristic description of what is going on.  

The random variable $Y$ describes the number of small LQG disks \jason{into} which the CLE exploration divides the CLE carpet (let us stress again that we are dealing with ${\mathcal T}$ here, and do not care about the domains encircled by $\CLE$-loops). It is for instance actually shown in \cite[Theorem 3.4]{dadoun} that if we stop the exploration mechanism in each branch of the tree ${\mathcal T}$ as soon as the label becomes smaller than $\eps$ (we will refer to  this as the ``stopping line'' ${\mathcal L}_\eps$ for this tree), then the empirical measure $\mu_\eps$ of the sizes of the collection of obtained labels behaves like 
$\eps^{-\alpha-1/2} \times Y \times \nu (\eps \cdot)$ 
as $\eps \to 0$, where $\nu$ is some finite measure on $[0,1]$. In particular, $\eps^{\alpha+1/2}$ times the number of disks-sizes/labels between $\eps/2$ and $\eps$ in this ``stopping line'' will converge in $L^1$ to some constant times $Y$.  
This suggests that the number of CLE loops with boundary length of the order of $\eps$ should also explode like some constant times $Y \eps^{-\alpha - 1/2 }$. Indeed, a typical small CLE loop will be discovered by the exploration within a disk of boundary length of the same order of magnitude as the boundary length of the CLE loop.

There are a number of possible concrete ways to turn the previous heuristics into a proof, building on the available results in the literature. Here is one outline, that builds on the aforementioned fact: 

\begin {prop}[Special case of Theorem 3.4 in \cite {dadoun}]
\label {propdadoun}
Let us denote the boundary lengths appearing at the stopping line ${\mathcal L}_y$ in decreasing order by $(Y_n^y, n \ge 1)$. There exists a deterministic measure $\nu$ on $[0,1]$, such that when $F$ is a measurable non-negative function on $[0,1]$, 
$$ y^{\alpha + 1/2} \sum_{n  \ge 1} F ( Y_n^y / y) \to Y \times \int F(u) d\nu (u)$$ 
in $L^1$ as $y \to 0$.
\end {prop}

Then: 
\begin {enumerate} 
\item Let us first make some simple a priori estimates. Note that the expectation of the number $N_{[x, \infty)}$ of upwards jumps of size greater than $x$ in the 
tree ${\mathcal T}$ (when started from a quantum disk of boundary length equal to $1$) is finite; it is actually trivially bounded by $1/x^2$ as the expected quantum area $C$ of the initial disk (here we mean the quantum area of the entire disk) is bounded from below by $\E[ N_{[x, \infty)} C x^2]$ (just sum the area of all the disks corresponding to the $\CLE_\kappa$ loops of boundary length greater than $x$).  

Next we can note that as $x \to 0$, the quantity $N_{[x,2x]} x^{\alpha +1/2}$ is bounded from below by some positive (random) number. We can for instance use the aforementioned result about the number of disks of boundary length in $[\eps/2,\eps]$ on the stopping line ${\mathcal L}_\eps$, and note that for each of these disks, the probability to then make a positive jump of size in $[x, 3x/2]$ (for any choice of $x \in [\eps/2, 2\eps]$) is bounded from below.  
 
 \item When $y \ge 4x$, let $N_{[x, 2x]}^y$ (resp. $D_{[x, 2x]}^y$) denote the total number of positive (resp.\ negative) jumps of size in $[x, 2x]$ that occur before the stopping line ${\mathcal L}_y$. 
 By comparing the jump rates of positive and negative jumps (noting that the latter is always bounded by the former for a given small-enough jump-size),
 $$ \E[ N_{[y/2, y]}^{2y}] \le C \E[ D_{[y/2, y]}^{2y} ]$$
 for some constant $C$. 
 But by construction, we note that 
 $$ \E [ D_{[y/2, y]}^{2y} ] \le \E [  \sum_{n \ge 1 } \one_{Y_n^{2y} \in [y/2, y]} ]  \le C' y^{-\alpha - 1/2}$$ 
 for some constant $C'$ (this follows for instance from Proposition~\ref{propdadoun}). 
 By comparing the jump rates of positive jumps of size in $[Mx, 2Mx]$ and of positive size in $[x, 2x]$ at times at which the labels are greater than $4Mx$, we then see that for some constant $C''$ that is independent of $M$ and $x$, 
 $$ \E [ N_{[x,2x]}^{4Mx} ] \le M^{\alpha} \E[ N_{[Mx,2Mx]}^{4Mx} ]  \le C'' M^{-1/2} x^{-\alpha - 1/2}.$$ 
 In particular, using Markov's inequality, we see that for any given $\delta$, if we choose $M$ very large, we can ensure that 
 $$ \p [ N_{[x, 2x]}^{4Mx} x^{\alpha + 1/2 } \ge \delta ] \le \delta $$ 
for all small $x$.  
\item 
We now fix a very large  $M$ and will study the parts of the tree {\em after} the stopping line ${\mathcal L}_{y}$ for $y =4Mx$ (we will use this value of $y$ from now on).  We denote by $Q_{[x, 2x]}^{y} = N_{[x, 2x]} - N_{[x, 2x]}^y$ the number of positive jumps of size in $[x, 2x]$ after that line. Recall that at this stopping line ${\mathcal L}_{y}$, one has a collection of disks of boundary lengths $Y_1^y, Y_2^y, \ldots$. For each $n$, we  denote by $Z_n^y$ the number of (outermost) CLE loops with quantum boundary length in $[x, 2x]$ that respectively lie in the disk of boundary length $Y_n^y$, so that 
$Q_{[x,2x]}^y = \sum_{n \ge 1} Z_n^y$.  Our goal is to show that the quantity $x^{\alpha + 1/2} Q_{[x, 2x]}^y$ converges in probability to some constant times $Y$.
For each $n$, we let $\overline Z_n^y$ denote the conditional expectation of $Z_n^y$ given $Y_n^y$. 

Let $F_0(z)$  denote the expected number of
(outermost) CLE loops of quantum boundary length in $[z/(4M), 2z/(4M)]$ in a CLE drawn on a quantum disk of boundary length $1$. We can first apply Proposition~\ref{propdadoun} to $F(z)=F_0(z) \one_{z \le \eps /(4M)}$, and see that it implies that when $\eps$ is chosen to be very small,
$$ x^{\alpha + 1/2} \sum_{n \ge 1} \overline Z_n^{y} \one_{Y_n^{y} \le \eps x }$$ converges in $L^1$ to some random variable that has a very small expectation. 
In particular, we deduce readily that for all given $\delta$, when $\eps$ is chosen small enough,
$$ \p[ x^{\alpha + 1/2} \sum_{n \ge 1} Z_n^{y} \one_{Y_n^{y} \le \eps x } \ge \delta ] \le \delta$$ 
for all $x$. 

\item 
It  remains to show that for fixed large $M$ and a fixed small $\eps$,  
$$x^{\alpha + 1/2} \sum_{n \ge 1} Z_n^y \one_{Y_n^{y} > \eps x} $$ 
converges in probability to some constant times $Y$. 
Here, we can first apply the same argument as above to the function 
$ F(z)= F_0 (z) \one_{ z \ge \eps / (4M)}$ to see that 
$$ x^{\alpha + 1/2} \sum_{n \ge 1} \overline Z_n^{y} \one_{Y_n^{y} > \eps x} $$ 
converges in probability to a constant times $Y$. 
Hence, it remains to argue that 
$$ x^{\alpha + 1/2} \sum_{n \ge 1} (Z_n^{y}- \overline Z_n^{y}) \one_{Y_n^{y} > \eps x } $$ 
converges to $0$ is probability. 
This is now essentially a variation of the law of large numbers: When $x$ is small and one conditions on the values $Y_n^y, \ldots$, this is bounded by a constant times the mean of a large number $m_x$ (that is greater than some positive number times $x^{-\alpha - 1/2}$ with high probability) of  independent random variables of mean $0$, and one has good uniform integrability control on these variables because of the integrability of the quantum area ${\mathcal A}$ of a quantum disk with boundary length $1$. More specifically, let $c$ denote the probability that a quantum disk with boundary length $1$ has quantum area at least $1$, and let $c'$ denote the infimum over all 
positive $m$ that the sum of $m$ independent Bernoulli random variables, that are each equal to $1$ with probability $c$, is greater than $mc/2$. Then, by bounding the sum of the quantum areas of the $Z_n^y$ to be explored disks by the area of the disk with boundary length $Y_n^y$, and using the scaling properties, we immediately see that for all $k$, and all $n$ with $Y_n^y > \eps x$,
\[ c'  \p [ Z_n^y > k ] \le \p [ {\mathcal A} > (k/2 ) \times (\eps/2M)^2 ],\]
which is sufficient to conclude.
\end {enumerate}

\subsection {From the total mass to the measure} 

We have now shown how to construct and describe the ``total LQG mass'' $Y$ of a CLE carpet drawn on a quantum disk. 
To conclude the proof of Theorem \ref{thmmeasure}, it now remains to explain how to construct the actual measure on the CLE carpet, and to check that it does not depend on the additional randomness that comes from the CPI structure within the CLE carpet. 

One can first recall that in the growth-fragmentation tree ${\mathcal T}$ setup, the martingale ideas do actually construct directly a measure on the boundary of the  tree -- which is referred to as its  {\em  intrinsic area measure} in \cite {bbck}. 
This cannot be used directly here when one wants to define the measure in our CLE on LQG setting, because the embedding of the boundary of the tree into the CLE carpet is rather complicated (so that describing the measure here simply as the pushforward of the measure on the boundary of the tree via the embedding would require some justification), and the tree construction (and therefore the embedding too) does depend on the additional CPI randomness.  

Let us now outline one possible self-contained simple way to proceed: First, we note that it suffices to construct the measure for any concrete embedding of the quantum disk in some planar domain (the measure for another embedding will then be the pushforward via the corresponding conformal map). Let us for instance choose an embedding where three randomly chosen points according to the LQG boundary length are sent to three points chosen independently according to the Lebesgue on the unit circle (with the same circular orientation).

We now proceed iteratively as follows: At the first step, in the CLE duality picture, we draw the boundary touching CLE$_{\kappa'}$ loops alongside all the CLE$_\kappa$ loops that one encountered (independently of the LQG structure). This amounts to drawing a CPI branching tree starting from one boundary point and targeting all other boundary points -- so one discovers (in a Markovian way) part of the exploration tree structure.  
If one discards the discovered CLE$_{\kappa}$ loops (and only keeps the information of the boundary length of the remaining domains, but not of their embedding), the to-be-explored domains form a collection of quantum disks. 

Then, iteratively, one repeats this operation independently inside each of these quantum disks (so that each of them is then again split into a countable collection of quantum disks). We denote by $(D_j^k)_{j \in J_k}$ the collection of quantum disks that one obtains after the $k$-th iteration. 

Our previous results then make it possible for each of these quantum disks to almost surely define its ``CLE carpet quantum mass'' that we will denote by $\mu (D_j^k)$, which is the limit as $n \to \infty $ of the suitably renormalized number of CLE loops with boundary length in $[ \eps_n, 2 \eps_n ]$ that are contained in $D_j^k$ (for a well-chosen fixed sequence $\eps_n$ that tends to $0$). 

Each of our CPI exploration steps corresponds to an exploration (i.e., a ``stopping line'') on the growth-fragmentation tree ${\mathcal T}$. It follows that in fact almost surely $\sum_{j \in J_1} \mu (D_j^1) = \mu (D)$ (because the corresponding fact holds for the measure on the boundary of the tree), and therefore iteratively, that $\sum_{j \in J_k} \mu (D_j^k) = \mu (D)$ for all $k \ge 1$.

Next, we consider a deterministic countable dense collection of pairs $(a_i, b_i)_{i \in I}$ of points on $\partial D$, and denote by  $l_i$ the ``geodesic'' circle-arc joining the points $a_i$ and  $b_i$ in $D$. Then: 

\begin{lemma} \label {premeasure}
For any given $l_i$, the sum $S_i^k$ of all $\mu (D_j^k)$ over all $D_j^k$ that intersect $l_i$ almost surely tends to $0$ as $k \to \infty$.
\end{lemma} 
\begin{proof}
Suppose that for some $i_0$ and some arbitrarily small but fixed $a$, the limit $S_{i_0}$ of the  decreasing sequence $S_{i_0}^k$ has a probability greater than $a$ to be greater than $a$. Then, for any large $K$, by our randomized choice of embedding and the density of the pairs of points $(a_i, b_i)$, one can find $K$ disjoint lines $l_{i_1}, \ldots, l_{i_K}$ very near $l_{i_0}$ such that for each of them $\p[ S_i  > a] \ge a/2$. On the other hand, for this given family of $K$ lines, when $k$ is chosen large enough, we know that with probability at least $a/4$, none of the $D_j^k$ will intersect more than one of these $K$ lines (just because the maximum over $j$ of the diameter of the $D_j^k$ goes to $0$ as $k \to \infty$) -- we call this event $E_k$. From there, we can readily conclude that for large enough $k$, 
\[ \E [ \mu (D) ] \ge \E [ \mu (D) 1_{E_k}] 
\ge \sum_{m \le K} a \p [ S_{i_m} > a,\ E_k] 
\ge Ka \times (a/2 - a/4) \ge Ka^2 /4. \]
This leads to a contradiction (when $K$ is chosen to be large enough) since the expected value of $\mu(D)=Y$ is almost surely finite. 
\end {proof} 

Let us finally outline how to now prove Theorem \ref {thmmeasure}:  
\begin {itemize} 
 \item The collection ${\mathcal U}$ of all finite unions $U$ of finite intersections of (closed or open) hyperbolic half-spaces that have one of the $l_i$'s as part of their boundaries forms a ring of sets that generates the family of Borel sets in $D$.
 \item 
For each $U \in {\mathcal U}$, we now define for each $k$,  $\mu_k (U)$ to be the sum of the $\mu (D_j^k)$ for $D_j^k \subset U$. This is then clearly a non-decreasing function of $k$, and we then define $\wt \mu(U) = \lim_{k \to \infty} \mu_k (U)$. 
For each $U \in {\mathcal U}$, we know that 
$$ \mu (D) - \mu_k (U) - \mu_k (D \setminus U)
=  \sum_{j \in J_k} \mu (D_j^k)- \mu_k (U) - \mu_k (D \setminus U)
= S_{\partial U}^k $$ 
where $S_{\partial U}^k$ denotes the sum over all $\mu (D_j^k)$ for $D_j^k \cap \partial U \not= \emptyset$, and this quantity tends to $0$ as $k \to \infty$ by Lemma \ref {premeasure}, so that 
$\wt \mu (U) + \wt \mu (D \setminus U) = \mu (D)$. 
Similarly, the same argument ensures that finite additivity holds for $\wt \mu$ on the ring ${\mathcal U}$. 

 \item The definition of $\mu$ shows readily that for all $U$ in ${\mathcal U}$, almost surely, for every $k$,   
 $$ \mu_k (U) \le \liminf_{ n\to \infty} \eps_n^{\alpha + 1/2} N_{[\eps_n, 2 \eps_n]} (U), $$
 where $N_{[\eps, 2\eps]}(U)$ denote the number of CLE loops with quantum length in $[\eps, 2 \eps]$ that are contained in $U$. The same inequality therefore holds for the limit $\wt \mu (U)$.  
 On the other hand, we know that
 $$ \wt \mu (U) + \wt \mu (D \setminus U) = \mu (D) = \lim_{n \to \infty}\eps_n^{\alpha + 1/2} N_{[\eps_n, 2 \eps_n]} (D), $$
 from which it follows readily that almost surely, 
  $$ \wt \mu (U) = \lim_{ n\to \infty } \eps_n^{\alpha + 1/2} N_{[\eps_n, 2 \eps_n]} (U).$$ 
  This in particular ensures that $\wt \mu$ (and therefore also the measure that will be obtained via the extension theorem) does not depend on the additional CPI randomness used in our iteration scheme. 
  
  \item To construct the measure ${\mathcal Y}$ that uniquely extends $\wt \mu$ via Carath\'eodory's extension theorem, it remains to check the countable additivity property, namely that if $U \in {\mathcal U}$ is the increasing limit of sets $U_m$ in ${\mathcal U}$, then $\wt \mu (U) = \lim_{m \to \infty} \wt \mu (U_m)$. This can be easily checked to be a consequence of the fact that for all $i$, the sum of all $\mu (D_j^k)$ over all $D_j^k$ that are contained in the $\delta$-neighborhood of $l_i$ does tend to $0$ as $\delta \to 0$, uniformly over $k$, which can be proved using the same ideas as for Lemma \ref {premeasure} (we safely leave this to the reader). 
  
  \item Finally, we can check that the obtained measure does also not depend on the concrete choice of the family $(a_i, b_i)_{i \in I}$; indeed, if one considers two different families, one can simply apply the previous argument to the union of these two families.  This can then also be used to see that the obtained measure ${\mathcal Y}$ is in fact independent of the embedding (meaning that the measure for one embedding is the conformal pushforward of the measure for another embedding -- so that in some sense, the randomness that comes from our embedding procedure does not influence the construction of the measure).
\end {itemize}

\begin {remark} 
Instead of using this ``randomized'' embedding, one could have used absolute continuity properties of the GFF, which would have let to slightly stronger statements of the following type: For any given embedding and any given $U \in {\mathcal U}$, ${\mathcal Y}(U)$ is the limit in probability of $\eps^{\alpha + 1/2} N_{[\eps, 2 \eps]} (U)$. 
It would be also possible to show stronger statements: For instance, that the convergence holds for any open set $U$ (such that the Lebesgue measure of $\partial U$ is $0$), but this is not needed here. 
\end {remark}

\appendix

\section{Proof of Lemmas \ref{qdright} and \ref{qdrighttwo}}
\label{proofoflemma} 

Let us start with the following general statement, that will be one key to the proof of \jason{Lemmas~\ref{qdright} and~\ref{qdrighttwo}}.

\begin{lemma}
\label{lem:disk_limit}
Fix $W \in (0,\gamma^2-2)$.  Suppose that $\CB$ is a bead of a quantum wedge of weight $W$ with left and right boundary length conditioned to be equal to $1$ and $\epsilon$ respectively.  Then in the limit as $\epsilon \to 0$, $\CB$ converges to a unit boundary length quantum disk.
\end{lemma}

In this lemma (and throughout this appendix), we are working with probability measures on beads of thin wedges conditioned on their left and right boundary lengths (see Remark \ref {beadboundaries}).  
Throughout this appendix, we say that a bead of a wedge has boundary length $(l,r)$ if the clockwise (resp.\ counterclockwise) boundary arcs of the \jason{bead} have respective lengths~$l$ and~$r$.

The sense of convergence that we consider in this lemma is the following.  Let $z \in \CB$ be chosen from the quantum measure and let $h$ be the field obtained by embedding $\CB$ into $\D$ using a conformal transformation which takes $z$ to $0$ and post-composed with a rotation by a uniformly random angle.  Then for any finite collection $\phi_1,\ldots,\phi_n$ of $C_0^\infty(\D)$ functions we have that the joint law of $(h,\phi_j)$ for $1 \leq j \leq n$ converges as $\epsilon \to 0$ to that when $h$ is the field which describes a unit boundary length quantum disk with the same embedding (uniformly random interior point taken to the origin post-composed with a rotation with a uniformly random angle).

\begin{proof}[\jason{Proof of Lemma~\ref{lem:disk_limit}}]
\jason{We are going to prove the result by realizing a bead of a quantum wedge of weight $W \in (0,\gamma^2 -2)$ inside of an ambient quantum disk with boundary length $1$ by cutting the latter with a segment of an $\SLE$ curve which connects $-i$ to a point on the counterclockwise arc of $\partial \D$ which is close to $-i$.  The idea is to argue that as this point tends to $-i$, the ambient quantum disk and the bead converges to it.  Some subtleties will arise because the resulting bead has random left and right boundary lengths.}

\jason{We will now make the above sketch more precise.}  Suppose that $\CD = (\D,h,-i,i)$ is a quantum disk with boundary length $1$.  Let $x$ be the point on the unit circle so that the counterclockwise arc connecting $-i$ and $x$ has quantum length $\epsilon$, so that $(\D, h, -i, x)$ is a bead of a $\gamma^2 - 2$ wedge with boundary lengths $(1-\eps, \eps)$.   Let $\eta_x$ be an independent $\SLE_\kappa(W-2 ; \kappa-4-W)$ process in $\CD$ from $-i$ to $x$.  We know from Theorem~\ref{thm:wedge_welding}-(ii) applied to the thin wedge of weight $\gamma^2- 2$ (the beads of which are \jason{quantum} disks) that the quantum surfaces parameterized by the components of $\D \setminus \eta$ which are to the left of $\eta$ are beads of a wedge of weight $W$ -- in particular, this will hold for the component $\CB$ with the largest boundary length.  

The goal is now to study what happens if we let $\epsilon$ tend to $0$. 
Conditionally on the boundary lengths $(L,R)$ of its two sides, $\CB$ has the law of a bead of a wedge of weight $W$ (with those boundary lengths).  Moreover, we can note that (simply because $\eta$ gets smaller and smaller in $\D$ in this setup), $\CB$ converges to $\CD$ itself as $\epsilon \to 0$ \jason{(in the sense described just above)}, i.e., it becomes a quantum disk with unit boundary length. So, we have the convergence of some bead of a quantum wedge of weight $W$ to a quantum disk as in the \jason{statement of the lemma}. However, the boundary lengths \jason{$(L,R)$} of \jason{$\CB$} are here random (but when $\eps$ is small, then $R$ is small and $L$ is close to $1$), so some little work is needed to deduce the lemma itself. 

Let us first provide more information about the law of $(L,R)$.  \jason{In particular, we aim to show that $\epsilon^{-1}(1-L,R)$ converges to a limit as $\epsilon \to 0$.}  Let $\varphi \colon \D \to \h$ be the unique conformal map which sends $-i$ to $0$, $x$ to $1$, and $i$ to $\infty$. 
We know from the definitions of a quantum disk and quantum half-plane that \jason{the pair consisting of}
\[ h \circ \varphi^{-1} + Q \log|(\varphi^{-1})'| + \frac{2}{\gamma} \log \frac{1}{\epsilon}\]
\jason{and $\varphi(\eta)$} converges as $\epsilon \to 0$ to the law of a quantum half-plane (embedded so that the boundary length of $[0,1]$ is $1$), in the strong sense of convergence \jason{which gives} that the restriction of the \jason{field/path pair} to any compact set converges in total variation).  Let $X$ be the quantum length of the part of the boundary of the unbounded component of $\h \setminus \varphi(\eta)$ which is on $\varphi (\eta)$ and let $Y$ be the quantum length of the part of $\partial \h$ which is cut off from $\infty$ by \jason{$\varphi(\eta)$}.  It follows from the above convergence that the law of $\epsilon^{-1}(1-L,R)$ converges in total variation to the law of $(X,Y)$ as $\eps \to 0$.  
Hence, if we combine this with the conclusion of the previous paragraph we get that if we consider a bead of a wedge with weight $W$ and respective boundary lengths $(1- \eps X, \eps Y)$, then as $\eps \to 0$, it converges (in distribution) to a quantum disk with unit boundary length. By simple scaling, the same will 
hold true if we take the boundary lengths to be $(1, \eps Y/(1-\eps X))$.

\jason{Finally, we note that for each $\delta > 0$ we have that the probability that $\eta$ is contained in $B(-i,\delta)$ tends to $1$ as $\epsilon \to 0$.  On the event that $\eta$ is contained in $B(-i,\delta)$, we have that $(L,R)$ is determined by the values of $h$ in $B(-i,\delta)$ and $\eta$.  On the other hand, the conditional law of $h$ given its values in $B(-i,\delta)$ converges to its unconditioned law as $\delta \to 0$ (by the backward martingale convergence theorem and the triviality of $\cap_{\delta > 0} \sigma(h|_{B(-i,\delta)})$).  The above facts altogether imply that both $\CB$ converges to $\CD$ as $\epsilon \to 0$ and, moreover, that the conditional law of $\CB$ given $(L,R)$ is close to its unconditioned law.  This implies that it is possible to condition on the values of $X$ and $Y$ above, from which the result follows.} 
\end{proof}

We are now finally ready to complete our proofs: 

\begin {proof}[Proof of Lemma~\ref{qdright} and of Lemma~\ref{qdrighttwo}]
We will write down the proof of Lemma \ref {qdright} -- exactly the same arguments work for Lemma \ref {qdrighttwo}, one just has to replace $3\kappa/2-6$ with $\rho$ and $-\kappa/2$ with $\kappa-6-\rho$.	

We will focus on completing the proof of the first part of Lemma~\ref{qdright} since the proof of the second part follows from the same argument.

We start by repeating some of the arguments of the proof of Lemma~\ref{lem:disk_limit} with a doubly marked quantum disk $\CD = (\D,h,-i,i)$, and we let $\eta$ be an independent $\SLE_\kappa(3\kappa/2-6;-\kappa/2)$ process in $\D$ from $-i$ to $i$.  As a quantum disk is a bead of a quantum wedge of weight $\gamma^2-2$, it follows from Theorem~\ref{thm:wedge_welding}-(ii) that the quantum surfaces parameterized by the components of $\D \setminus \eta$ which are to the left (resp.\ right) of $\eta$, when conditioned on their left and right boundary lengths, are beads of a quantum wedge of weight $3\gamma^2/2-4$ (resp.\ $2-\gamma^2/2$).  In particular, the beads to the left of $\eta$ are exactly of the type considered in the first  part of Lemma~\ref{qdright}.

\begin{figure}[ht!]
\includegraphics[scale=.50]{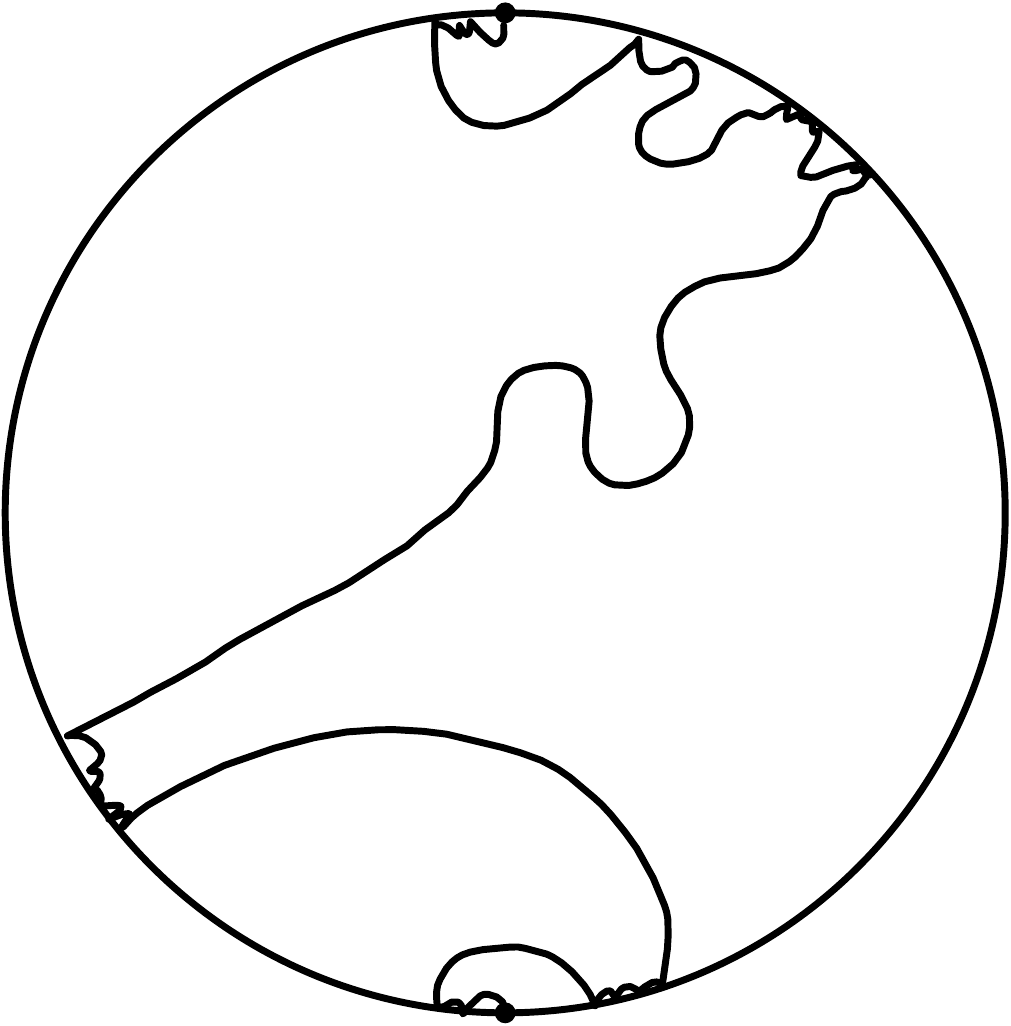} 
\quad
\includegraphics[scale=.50]{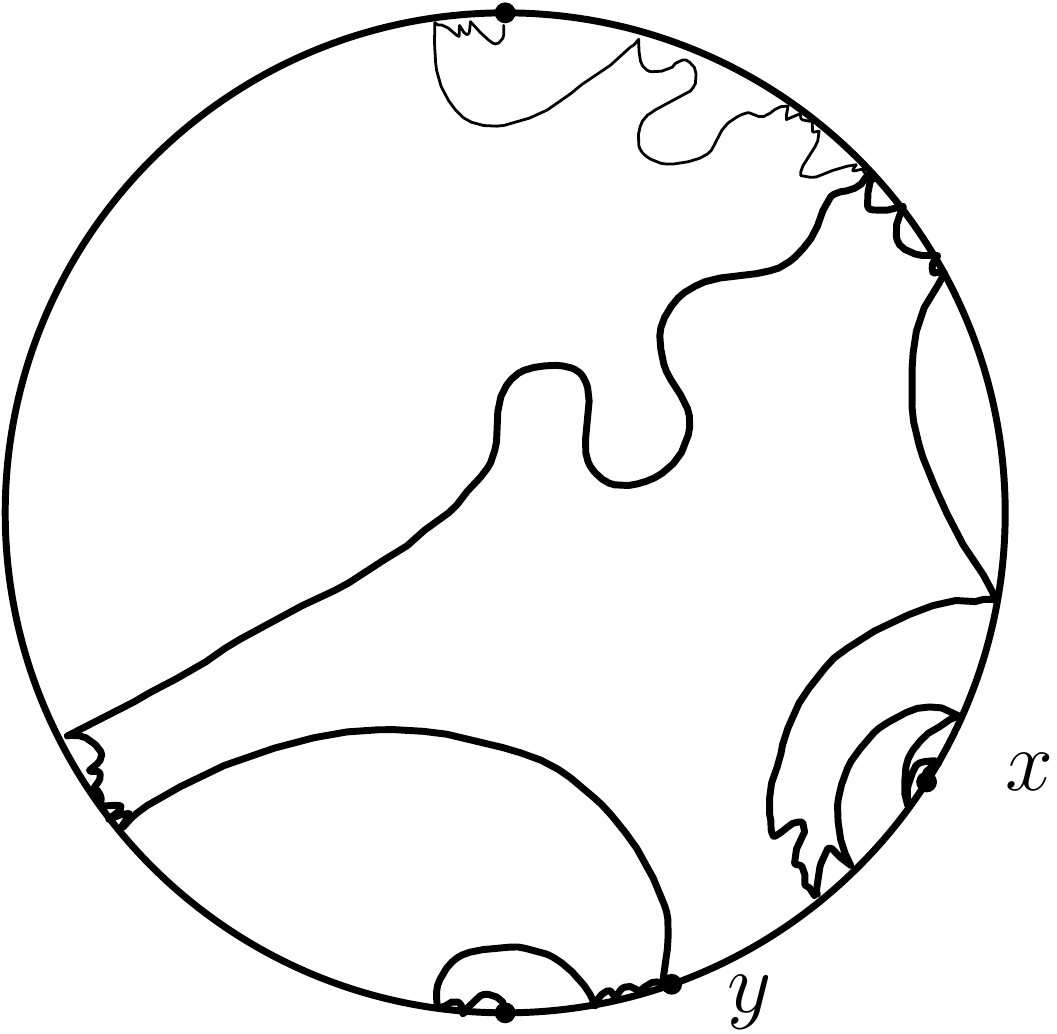} 
	\caption{The curve $\eta$; the coupling with $\eta_x$.}
	\label{F10}
\end{figure}

Let $x \in \partial \D$ be picked from the quantum length measure independently of everything else.  Let $\eta_x$ be an $\SLE_\kappa(3\kappa/2-6;-\kappa/2)$ process in $\D$ from $-i$ to $x$ which is coupled together with $\eta$ so as to agree until the first time that their target points have been separated and then to evolve independently afterwards (we also assume that the pair $(\eta,\eta_x)$ is independent of $\CD$), see Figure~\ref{F10}. Let $\CB_x$ be the bead of $\D \setminus \eta$ with $x$ on its boundary.  The target-independence of these $\SLE_\kappa (\rho; \kappa -6- \rho)$ processes show readily that the part of $\eta_x$ in $\CB_x$, viewed as a process starting from the branching point $z$ and which is targeted at the {\em first} point $y$ on $\partial \CB_x$ visited by $\eta$, is an 
$\SLE_\kappa(3\kappa/2-6)$ process (here and in the following paragraph, we will implicitly use the fact that $x$ is chosen randomly, and has a positive probability to actually be very close to $y$). We can actually couple it with an $\SLE_\kappa (3 \kappa /2 -6)$  that goes all the way to $y$ and that we denote by $\wh \eta$.

On the event that $x$ is on the counterclockwise segment of $\partial \D$ from $-i$ to $i$, the same arguments as above 
show that the beads parameterized by the components of $\D \setminus \eta_x$ which are to the left (resp. right) of $\eta_x$ are beads of a quantum wedge of weight $3\gamma^2/2-4$
(resp. $2 - \gamma^2 /2$).  
Putting all these items together, one gets that the following is true: Suppose that $\wt{\CB}$ is a bead of a wedge of weight $2-\gamma^2/2$  with prescribed left and right boundary lengths (so that $\wt {\CB}$ has the same weight as $\CB_x$ -- we can note that the way in which $\CB_x$ was selected induces some bias in the law of its left and right boundary lengths, but not on its conditional law given these two boundary lengths). Let $\wt{\eta}$ be an independent $\SLE_\kappa(3\kappa/2-6)$ process between the two marked points of $\wt{\CB}$ (i.e., same law as $\wh \eta$ in $\CB_x$).  Then the components which are to the left of $\wt{\eta}$ are independent given their left and right boundary lengths and (when conditioned on their left and right boundary lengths) are beads of a wedge of weight $3\gamma^2/2-4$. Mind again that this does not provide information about the joint law of the left and right boundary lengths of these components -- but says only that when one conditions on these lengths, then the conditional law of the surfaces are those of beads with prescribed boundary lengths.

\begin{figure}[ht!]
\includegraphics[scale=.50]{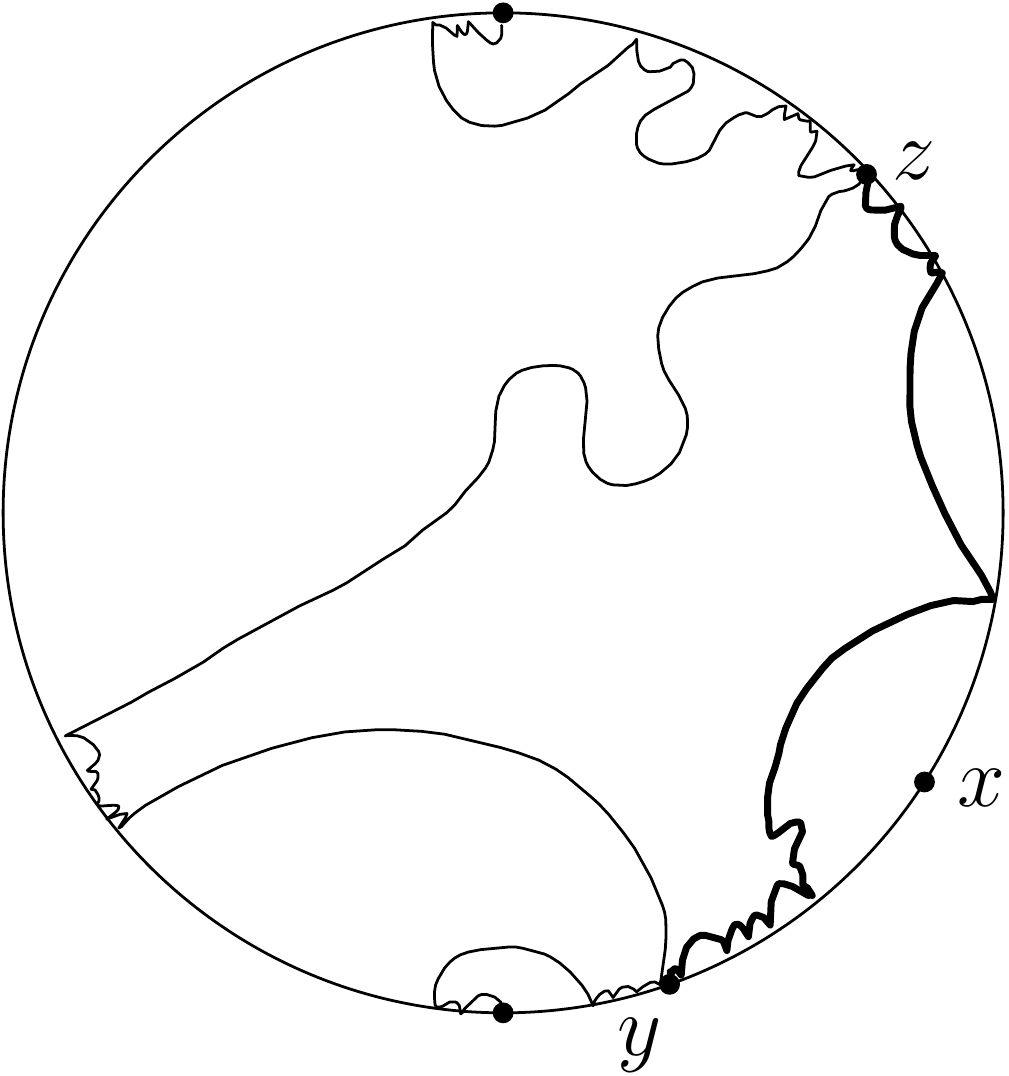}
\quad 
\includegraphics[scale=.50]{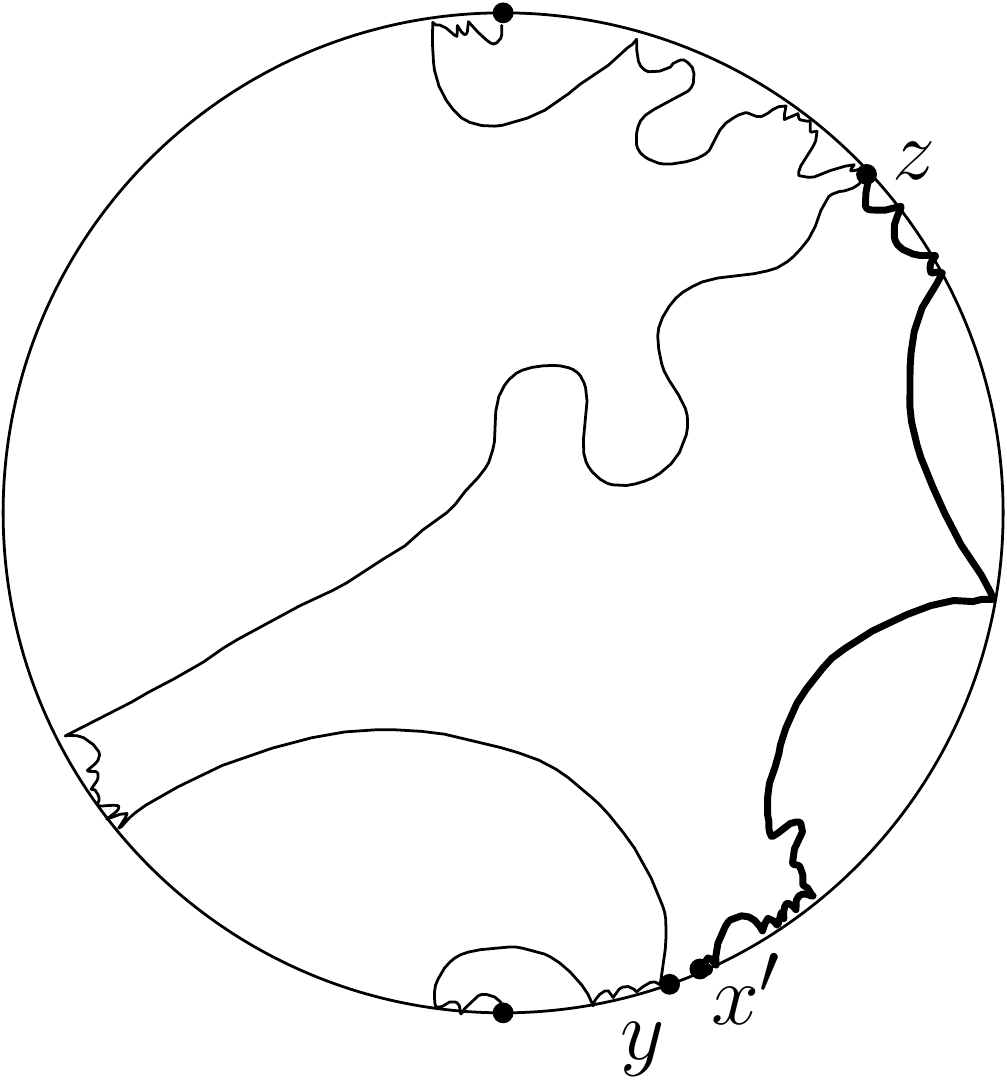}
	\caption{The curve $\wh \eta$ from $z$ to $y$ in the bead $\CB_x$. The curve just before completion can be viewed as part of $\eta_{x'}$ for $x'$ near $y$ and therefore defines a bead of a wedge on its right.}
	\label{F11}
\end{figure}

To complete the proof, we need to show that the surface which is to the right of $\wt{\eta}$ is a quantum disk:
If we stop $\wt{\eta}$ at a time before it reaches its target point and just disconnects a boundary arc to its left (so that  it could actually choose to branch to its left at that point) as in Figure~\ref{F11}, then the same argument shows that the surface which is to its right is a bead of a quantum wedge of weight $2-\gamma^2/2$ (conditionally on its boundary lengths).  But as $\wt{\eta}$ approaches its target point, the left boundary length of this bead tends to $0$ while its right boundary length increases. One can then just apply Lemma~\ref{lem:disk_limit} to conclude. 
\end {proof}


\begin{thebibliography}{99}

\bibitem {AG} 
M. Ang, E. Gwynne. 
Liouville quantum gravity surfaces with boundary as matings of trees. 
{\em Ann. Inst. H. Poincaré Probab. Statist.}, 57(1):1--53, 2021. 


\bibitem{ay2017twoperspectives}
J.~Aru, Y.~Huang, and X.~Sun.
 Two perspectives of the 2{D} unit area quantum sphere and their
  equivalence.
 {\em Comm. Math. Phys.}, 356(1):261--283, 2017.

\bibitem{BEN_HONG}
S.~Benoist and C.~Hongler.
 The scaling limit of critical {I}sing interfaces is {CLE$_3$}.
 {\em Ann. Probab.}, 47(4):2049--2086, 2019.

\bibitem{bertoin96levy}
J.~Bertoin.
 {\em L\'evy processes}, volume 121 of {\em Cambridge Tracts in
  Mathematics}.
 Cambridge University Press, Cambridge, 1996.

\bibitem{Bertoin}
J.~Bertoin.
 Markovian growth-fragmentation processes.
 {\em Bernoulli}, 23(2):1082--1101, 2017.

\bibitem{bbck}
J.~Bertoin, T.~Budd, N.~Curien, and I.~Kortchemski.
 Martingales in self-similar growth-fragmentations and their
  connections with random planar maps.
 {\em Probab. Theory Related Fields}, 172(3-4):663--724, 2018.

\bibitem{bck}
J.~Bertoin, N.~Curien, and I.~Kortchemski.
 Random planar maps and growth-fragmentations.
 {\em Ann. Probab.}, 46(1):207--260, 2018.

\bibitem{bck2}
J.~Bertoin, N.~Curien, and I.~Kortchemski.
 On conditioning a self-similar growth-fragmentation by its intrinsic
  area.
 arXiv:1908.07830, 2019.

\bibitem{Biggins}
J.~D. Biggins.
 Uniform convergence of martingales in the branching random walk.
 {\em Ann. Probab.}, 20(1):137--151, 1992.

\bibitem{BK}
J.~D. Biggins and A.~E. Kyprianou.
 Measure change in multitype branching.
 {\em Adv. in Appl. Probab.}, 36(2):544--581, 2004.

\bibitem{BBG1}
G.~Borot, J.~Bouttier, and E.~Guitter.
 A recursive approach to the {$O(n)$} model on random maps via nested
  loops.
 {\em J. Phys. A}, 45(4):045002, 38, 2012.

\bibitem{budd}
T.~Budd.
 The peeling process on random planar maps coupled to an $O(n)$ loop
  model (with an appendix by Linxiao Chen), arXiv:1809.02012, 2018.

\bibitem{c2019disk}
B.~{Cercl{\'e}}.
 {Unit boundary length quantum disk: a study of two different
  perspectives and their equivalence}.
  arXiv:1912.08012, 
  2019.

\bibitem{ccm2017lengths}
L.~{Chen}, N.~{Curien}, and P.~{Maillard}.
 {The perimeter cascade in critical Boltzmann quadrangulations
  decorated by an $O(n)$ loop model}.
 {\em Ann. IHP (D)}, to appear.

\bibitem{curienrichier}
N.~Curien and L.~Richier.
 Duality of random planar maps via percolation, 
 {\em Ann. Inst. Fourier}, to appear.
 
\bibitem{dadoun}
B.~Dadoun.
 Asymptotics of self-similar growth-fragmentation processes.
 {\em Electron. J. Probab.}, 22:Paper No. 27, 30, 2017.

\bibitem{lqg_sphere}
F.~David, A.~Kupiainen, R.~Rhodes, and V.~Vargas.
 Liouville quantum gravity on the {R}iemann sphere.
 {\em Comm. Math. Phys.}, 342(3):869--907, 2016.

\bibitem{dddf2019tightness}
J.~{Ding}, J.~{Dub{\'e}dat}, A.~{Dunlap}, and H.~{Falconet}.
 {Tightness of Liouville first passage percolation for $\gamma \in
  (0,2)$}.
 {\em Pub. Math. IHES}, 132:353–403, 2020.
  
\bibitem{dms2014mating}
B.~{Duplantier}, J.~{Miller}, and S.~{Sheffield}.
 {Liouville quantum gravity as a mating of trees}.
 {\em Ast\'erisque}, to appear.

\bibitem{DS08}
B.~Duplantier and S.~Sheffield.
 Liouville quantum gravity and {KPZ}.
 {\em Invent. Math.}, 185(2):333--393, 2011.

\bibitem{gw2019fk}
C.~Garban and H.~Wu.
 On the convergence of {FK}-{I}sing percolation to {SLE}(16/3,(16/3)-6).
 {\em J. Theor. Probab.}, 33:828–865, 2020.
 

\bibitem{gm2017sle6_disk}
E.~Gwynne and J.~Miller.
 Chordal {${\rm SLE}_6$} explorations of a quantum disk.
 {\em Electron. J. Probab.}, 23:Paper No. 66, 24, 2018.

\bibitem{gm2019exunique}
E.~{Gwynne} and J.~{Miller}.
 {Existence and uniqueness of the Liouville quantum gravity metric for
  $\gamma \in (0,2)$}.
  {\em Inventiones Math.}, to appear.

\bibitem{hk1971quantum}
R.~H{\o}egh-Krohn.
 A general class of quantum fields without cut-offs in two space-time
  dimensions.
 {\em Comm. Math. Phys.}, 21:244--255, 1971.

\bibitem{hrv2018disk}
Y.~Huang, R.~Rhodes, and V.~Vargas.
 Liouville quantum gravity on the unit disk.
 {\em Ann. Inst. Henri Poincar\'{e} Probab. Stat.}, 54(3):1694--1730,
  2018.

\bibitem{kahane1985gmc}
J.-P. Kahane.
 Sur le chaos multiplicatif.
 {\em Ann. Sci. Math. Qu\'{e}bec}, 9(2):105--150, 1985.

\bibitem{ks2016fkconvergence}
A.~{Kemppainen} and S.~{Smirnov}.
 {Conformal invariance in random cluster models. II. Full scaling
  limit as a branching SLE}.
  arXiv:1609.08527, 2016.

\bibitem{lgm2011large_faces}
J.-F. Le~Gall and G.~Miermont.
 Scaling limits of random planar maps with large faces.
 {\em Ann. Probab.}, 39(1):1--69, 2011.

\bibitem{Liu}
Q.~Liu.
 On generalized multiplicative cascades.
 {\em Stochastic Process. Appl.}, 86(2):263--286, 2000.

\bibitem{lpp}
R.~Lyons, R.~Pemantle, and Y.~Peres.
 Conceptual proofs of {$L\log L$} criteria for mean behavior of
  branching processes.
 {\em Ann. Probab.}, 23(3):1125--1138, 1995.

\bibitem{MS_IMAG}
J.~Miller and S.~Sheffield.
 Imaginary geometry {I}: interacting {SLE}s.
 {\em Probab. Theory Related Fields}, 164(3-4):553--705, 2016.

\bibitem{MS_IMAG2}
J.~Miller and S.~Sheffield.
 Imaginary geometry {II}: reversibility of {${\rm SLE}\sb
  \kappa(\rho\sb 1;\rho\sb 2)$} for {$\kappa\in(0,4)$}.
 {\em Ann. Probab.}, 44(3):1647--1722, 2016.

\bibitem{qlebm}
J.~{Miller} and S.~{Sheffield}.
 {Liouville quantum gravity and the Brownian map I: The QLE(8/3,0)
  metric}.
 {\em Inventiones Math.}, to appear.

\bibitem{mswupcoming}
J.~{Miller}, S.~{Sheffield}, and W.~{Werner}.
 {Non-simple conformal loop ensembles on Liouville Quantum Gravity
 and the law of CLE percolation interfaces}.
 arXiv:2006.14605, 2020.

\bibitem{cle_percolations}
J.~Miller, S.~Sheffield, and W.~Werner.
 C{LE} percolations.
 {\em Forum Math. Pi}, 5:e4, 102 pages, 2017.

\bibitem{msw2016notdetermined}
J.~{Miller}, S.~{Sheffield}, and W.~{Werner}.
 {Non-simple SLE curves are not determined by their range}.
 {\em J. Europ. Math. Soc.}, 22:669–716, 2020.

\bibitem{nw2011carpet}
S.~Nacu and W.~Werner.
 Random soups, carpets and fractal dimensions.
 {\em J. Lond. Math. Soc. (2)}, 83(3):789--809, 2011.

\bibitem{rv2010revisited}
R.~Robert and V.~Vargas.
 Gaussian multiplicative chaos revisited.
 {\em Ann. Probab.}, 38(2):605--631, 2010.

\bibitem{RS05}
S.~Rohde and O.~Schramm.
 Basic properties of {SLE}.
 {\em Ann. of Math.}, 161(2):883--924, 2005.

\bibitem{S0}
O.~Schramm.
 Scaling limits of loop-erased random walks and uniform spanning
  trees.
 {\em Israel J. Math.}, 118:221--288, 2000.

\bibitem{ssw2009radii}
O.~Schramm, S.~Sheffield, and D.~B. Wilson.
 Conformal radii for conformal loop ensembles.
 {\em Comm. Math. Phys.}, 288(1):43--53, 2009.

\bibitem{SHE_CLE}
S.~Sheffield.
 Exploration trees and conformal loop ensembles.
 {\em Duke Math. J.}, 147(1):79--129, 2009.

\bibitem{SHE_WELD}
S.~Sheffield.
 Conformal weldings of random surfaces: {SLE} and the quantum gravity
  zipper.
 {\em Ann. Probab.}, 44(5):3474--3545, 2016.

\bibitem{sw2016measure}
S.~{Sheffield} and M.~{Wang}.
 {Field-measure correspondence in Liouville quantum gravity almost
  surely commutes with all conformal maps simultaneously}.
  arXiv:1605.06171, 2016.

\bibitem{SHE_WER_CLE}
S.~Sheffield and W.~Werner.
 Conformal loop ensembles: the {M}arkovian characterization and the
  loop-soup construction.
 {\em Ann. of Math.}, 176(3):1827--1917, 2012.

\bibitem{s2010ising}
S.~Smirnov.
 Conformal invariance in random cluster models. {I}. {H}olomorphic
  fermions in the {I}sing model.
 {\em Ann. of Math.}, 172(2):1435--1467, 2010.

\bibitem{ww2013conformally}
W.~{Werner} and H.~{Wu}.
 {On Conformally Invariant CLE Explorations}.
 {\em Communications in Mathematical Physics}, 320:637--661, 2013.

\end{thebibliography}
\end{document}